\documentclass{amsart}
\newtheorem{theorem}{Theorem}[section]
\newtheorem{lemma}[theorem]{Lemma}

\theoremstyle{definition}

\theoremstyle{remark}
\newtheorem{remark}[theorem]{Remark}
\usepackage{amssymb, graphicx, amsmath,epstopdf,bm}
\usepackage{amsthm}
\usepackage{algorithm}
\usepackage{algpseudocode}
\usepackage{cases}
\usepackage{amsfonts,amsmath,amsthm,amsfonts,amssymb}
\usepackage{graphics,graphicx,subfigure,color}
\numberwithin{equation}{section}
\allowdisplaybreaks[4]
\begin{document}

\title[Bilinear Optimal Control of an Advection-reaction-diffusion System]{Bilinear Optimal Control of an Advection-reaction-diffusion System}


\author[R. Glowinski]{Roland Glowinski}
\address{Department of Mathematics, University of Houston, 44800 Calhoun Road, Houston, TX 77204, USA. }
\curraddr{}
\email{roland@math.uh.edu}
\thanks{The work of R.G is supported by the US Department of Energy (ORNL) and the Hong Kong based Kennedy Wong Foundation. The work of X.Y is supported by the General Research Fund 12302318 from the Hong Kong Research Grants Council.}

\author[Y. Song]{Yongcun Song}
\address{Department of Mathematics, The University of Hong Kong, Pok Fu Lam, Hong Kong, PRC.}
\curraddr{}
\email{ysong307$@$hku.hk.}
\thanks{}

\author[X. Yuan]{Xiaoming Yuan}
\address{Department of Mathematics, The University of Hong Kong, Pok Fu Lam, Hong Kong, PRC.}
\curraddr{}
\email{xmyuan$@$hku.hk.}
\thanks{}

\author[H. Yue]{Hangrui Yue}
\address{Department of Mathematics, The University of Hong Kong, Pok Fu Lam, Hong Kong, PRC.}
\curraddr{}
\email{yuehangrui$@$gmail.com.}
\thanks{}
\date{January 7, 2021}
\dedicatory{To the memory of J. L. Lions (1928-2001) who suggested investigating problems like (BCP) below.}

\begin{abstract}
We consider the bilinear optimal control of an advection-reaction-diffusion system, where the control arises as the velocity field in the advection term. Such a problem is generally challenging from both theoretical analysis and algorithmic design perspectives mainly because the state variable depends nonlinearly on the control variable and an additional divergence-free constraint on the control is coupled together with the state equation.
 Mathematically, the proof of the existence of optimal solutions is delicate, and up to now, only some results are known for a few special cases where additional restrictions are imposed on the space dimension and the regularity of the control.  We prove the existence of optimal controls and derive the first-order optimality conditions in general settings without any extra assumption. Computationally, the well-known conjugate gradient (CG) method can be applied conceptually. However, due to the additional divergence-free constraint on the control variable and the nonlinear relation between the state and control variables, it is challenging to compute the gradient and the optimal stepsize at each CG iteration, and thus nontrivial to implement the CG method. To address these issues, we advocate a fast inner preconditioned CG method to ensure the  divergence-free constraint and an efficient inexactness strategy to determine an appropriate stepsize. An easily implementable nested CG method is thus proposed for solving such a complicated problem. For the numerical discretization, we combine finite difference methods for the time discretization and finite element methods for the space discretization. Efficiency of the proposed nested CG method is promisingly validated by the results of some preliminary numerical experiments.
\end{abstract}

\maketitle
\section{Introduction}
\subsection{Background and Motivation}
The optimal control of distributed parameter systems has important applications in various scientific areas, such as physics, chemistry, engineering, medicine, and finance. We refer to, e.g. \cite{glowinski1994exact, glowinski1995exact, glowinski2008exact, lions1971optimal, troltzsch2010optimal,zuazua2006}, for a few references. In a typical mathematical model of a controlled distributed parameter system, either boundary or internal locally distributed controls are usually used; these controls  have localized support and  are called additive controls because they arise in the model equations as additive terms.   Optimal control problems with additive controls have received
a significant attention in past decades following the pioneering work of J. L. Lions \cite{lions1971optimal}, and many mathematical and computational tools have been developed, see e.g., \cite{glowinski1994exact,glowinski1995exact,glowinski2008exact,lions1988,zuazua2005,zuazua2007}. However, it is worth noting that additive controls describe the effect of external added sources or forces and they do not change the principal intrinsic properties of the controlled system. Hence, they are not
suitable to deal with processes whose principal intrinsic properties should be changed by some control actions. For instance, if we aim at changing the reaction rate in some chain reaction-type processes from biomedical, nuclear, and chemical applications, additive controls amount to controlling the chain reaction by
adding into or withdrawing out of a certain amount of the reactants, which is not realistic. To address this issue, a natural idea is to use certain catalysts or smart materials to control the systems, which can be mathematically modeled by optimal control problems with bilinear controls. We refer to \cite{khapalov2010} for more detailed discussions.

Bilinear controls, also known as multiplicative controls, enter the model as coefficients of the corresponding partial differential equations (PDEs). These bilinear controls can change some
 main physical
characteristics of the system under investigation, such as a natural frequency response of a beam or the rate of a chemical reaction. In the literature, bilinear controls of
distributed parameter systems have become an increasingly popular topic and bilinear optimal control problems constrained by various PDEs, such as elliptic equations \cite{kroner2009}, convection-diffusion equations \cite{borzi2015},
parabolic equations \cite{khapalov2003}, the Schr{\"o}dinger equation \cite{kunisch2007} and the Fokker-Planck equation \cite{fleig2017}, have been widely studied both mathematically and computationally.

In particular, bilinear controls play a crucial role in optimal control problems modeled by advection-reaction-diffusion systems. 
On one hand, the control can be the coefficient of the diffusion or the reaction term. For instance, a system controlled by the so-called catalysts that can accelerate or
slow down various chemical or biological reactions can be modeled by a bilinear optimal control problem for an advection-reaction-diffusion equation  where the control arises as the coefficient of the reaction term \cite{khapalov2003}; this kind of bilinear optimal control problems have been studied in e.g., \cite{borzi2015,cannarsa2017,khapalov2003,khapalov2010}. 
On the other hand, the systems can also be controlled by the velocity field in the advection term, which captures important applications in e.g., bioremediation \cite{lenhart1998}, environmental remediation process \cite{lenhart1995}, and mixing enhancement of different fluids \cite{liu2008}. We note that there is a very limited research being done on the velocity field controlled bilinear optimal control problems;  and only some special one-dimensional space cases have been studied in \cite{lenhart1998,joshi2005,lenhart1995} for the existence of an optimal control and the derivation of first-order optimality conditions.  To the best of our knowledge, no work has been done yet to develop efficient numerical methods for solving multi-dimensional bilinear optimal control problems controlled by the velocity field in the advection term.
All these facts motivate us to study bilinear optimal control problems constrained by an advection-reaction-diffusion equation, where the control enters into the model as the velocity field in the advection term. Actually, investigating this kind of problems was suggested to one of us (R. Glowinski), in the late 1990's, by J. L. Lions (1928-2001).

\subsection{Model}
Let $\Omega$ be a bounded domain of $\mathbb{R}^d$ with $d\geq 1$ and let $\Gamma$ be its
boundary. We consider the following bilinear optimal control problem:
\begin{flalign}\tag{BCP}
& \left\{
\begin{aligned}
& \bm{u}\in \mathcal{U}, \\
&J(\bm{u})\leq J(\bm{v}), \forall \bm{v}\in \mathcal{U},
\end{aligned}
\right.
\end{flalign}
with the objective functional $J$ defined by
\begin{equation}\label{objective_functional}
J(\bm{v})=\frac{1}{2}\iint_Q|\bm{v}|^2dxdt+\frac{\alpha_1}{2}\iint_Q|y-y_d|^2dxdt+\frac{\alpha_2}{2}\int_\Omega|y(T)-y_T|^2dx,
\end{equation}
and $y=y(t;\bm{v})$ the solution of the
following advection-reaction-diffusion equation
\begin{flalign}\label{state_equation}
& \left\{
\begin{aligned}
\frac{\partial y}{\partial t}-\nu \nabla^2y+\bm{v}\cdot \nabla y+a_0y&=f\quad \text{in}\quad Q, \\
y&=g\quad \text{on}\quad \Sigma,\\
y(0)&=\phi.
\end{aligned}
\right.
\end{flalign}
Above and below, $Q=\Omega\times (0,T)$ and $\Sigma=\Gamma\times (0,T)$ with $0<T<+\infty$; $\alpha_1\geq 0, \alpha_2\geq 0, \alpha_1+\alpha_2>0$; the target functions $y_d$ and $y_T$ are given in $L^2(Q)$ and $L^2(\Omega)$, respectively; the diffusion coefficient $\nu>0$ and the reaction coefficient $a_0$ are assumed to be constants; the functions $f\in L^2(Q)$, $g\in L^2(0,T;H^{1/2}(\Gamma))$ and $\phi\in L^2(\Omega)$. The set $\mathcal{U}$ of the admissible controls is defined by
\begin{equation*}
\mathcal{U}:=\{\bm{v}|\bm{v}\in [L^2(Q)]^d, \nabla\cdot\bm{v}=0\}.
\end{equation*}
Clearly, the control variable $\bm{v}$ arises in (BCP) as a flow velocity field in the advection term of (\ref{state_equation}), and the divergence-free constraint $\nabla\cdot\bm{v}=0$ implies that the flow is incompressible. One can control the system by changing the flow velocity $\bm{v}$ in order that $y$ and $y(T)$ are good approximations to $y_d$ and $y_T$, respectively.  

\subsection{Difficulties and Goals}

In this paper, we intend to study the bilinear optimal control problem (BCP) in the general case of $d\geq 2$ both mathematically and computationally. Precisely, we first study the well-posedness of (\ref{state_equation}), the existence of an optimal control $\bm{u}$, and its first-order optimality condition. Then, computationally,
we propose an efficient and relatively easy to implement numerical method to solve (BCP). For this purpose, we advocate combining a conjugate gradient (CG) method with a finite difference method (for the time discretization) and a finite element method (for the space discretization) for the numerical solution of (BCP). Although these numerical approaches have been well developed in the literature, it is nontrivial to implement them to solve (BCP) as discussed below, due to the complicated problem settings.

\subsubsection{Difficulties in Algorithmic Design}
Conceptually, a CG method for solving (BCP) can be easily derived following \cite{glowinski2008exact}. However, CG algorithms are challenging to implement
numerically for the following reasons: 1). The state $y$ depends non-linearly on the control $\bm{v}$ despite the fact that the state equation (\ref{state_equation}) is linear. 2). The additional divergence-free constraint on the control $\bm{v}$, i.e., $\nabla\cdot\bm{v}=0$, is coupled together with the state equation (\ref{state_equation}). 

To be more precise, the fact that the state $y$ is a nonlinear function of the control $\bm{v}$ makes the optimality 
system a nonlinear problem. Hence, seeking a suitable stepsize in each CG iteration requires solving an optimization problem and it can not be as easily computed as in the linear case \cite{glowinski2008exact}. Note that commonly used line search strategies are too expensive to employ in our settings because they require evaluating the objective functional value $J(\bm{v})$ repeatedly and every evaluation of $J(\bm{v})$ entails solving the state equation (\ref{state_equation}). The same concern on the computational cost also applies when the Newton method is employed to solve the corresponding optimization problem for finding a stepsize. To tackle this issue, we propose an efficient inexact stepsize strategy which requires solving only one additional linear parabolic problem and is cheap to implement as shown in Section \ref{se:cg}.

Furthermore, due to the divergence-free constraint $\nabla\cdot\bm{v}=0$, an extra projection onto the admissible set $\mathcal{U}$ is required to compute the first-order differential of $J$ at each CG iteration in order that all iterates of the CG method are feasible. Generally, this projection subproblem has no closed-form solution and has to be solved iteratively. Here, we introduce a Lagrange multiplier associated with the constraint $\nabla\cdot\bm{v}=0$, then the computation of the first-order differential $DJ(\bm{v})$ of $J$ at $\bm{v}$ is equivalent to solving a Stokes type problem. Inspired by \cite{glowinski2003}, we advocate employing a preconditioned CG method, which operates on the space of the Lagrange multiplier, to solve the resulting Stokes type problem. With an appropriately chosen preconditioner, a fast convergence of the resulting preconditioned CG method can be expected in practice (and indeed, has been observed). 


\subsubsection{Difficulties in Numerical Discretization}

For the numerical discretization of (BCP), we note that if an implicit finite difference scheme is used for the time discretization of the state equation (\ref{state_equation}), a stationary advection-reaction-diffusion equation should be solved at each time step.  To solve this stationary advection-reaction-diffusion equation, it is well known that standard finite element techniques may lead to strongly oscillatory solutions unless the mesh-size is sufficiently
small with respect to the ratio between $\nu$ and $\|\bm{v}\|$. In the context of optimal control problems, to overcome such difficulties, different stabilized finite element methods
have been proposed and analyzed, see e.g., \cite{BV07,DQ05}. Different from the above references, we implement the time discretization by a semi-implicit finite difference method for simplicity, namely, we use explicit advection and reaction terms and treat the diffusion term implicitly. Consequently, only a simple linear elliptic equation is required to be solved at each time step. We then implement the space discretization of the resulting elliptic equation at each time step by a standard piecewise linear finite element method and the resulting linear system is very easy to solve.

Moreover, we recall that the divergence-free constraint $\nabla\cdot \bm{v}=0$ leads to a projection subproblem, which is equivalent to a Stokes type problem, at each iteration of the CG algorithm. As discussed in \cite{glowinski1992}, to discretize a Stokes type problem, direct applications of standard finite element methods always lead to an ill-posed discrete problem. To overcome this difficulty, one can use different types of element approximations for pressure and velocity.
Inspired by \cite{glowinski1992,glowinski2003}, we employ the Bercovier-Pironneau finite element pair \cite{BP79} (also known as $P_1$-$P_1$ iso $P_2$ finite element) to approximate the control $\bm{v}$  and the Lagrange multiplier associated with the divergence-free constraint. More concretely, we approximate the Lagrange multiplier by a piecewise linear finite element space which is twice coarser than the one for the control $\bm{v}$.  In this way, the discrete problem is well-posed and can be solved by a preconditioned CG method. As a byproduct of the above discretization, the total number of degrees of freedom of the discrete Lagrange multiplier is only $\frac{1}{d2^d}$ of the number of the discrete control. Hence, the inner preconditioned CG method is implemented in a lower-dimensional space than that of the state equation (\ref{state_equation}), implying a computational cost reduction.  With the above mentioned discretization schemes, we can relatively easily obtain the fully discrete version of (BCP) and derive the discrete analogue of our proposed nested CG method.

\subsection{Organization}
An outline of this paper is as follows. In Section \ref{se:existence and oc}, we prove the existence of optimal controls for (BCP) and derive the associated first-order optimality conditions. An easily implementable nested CG method is proposed in Section \ref{se:cg} for solving  (BCP) numerically. In Section \ref{se:discretization}, we discuss the numerical discretization of (BCP) by finite difference and finite element methods. Some preliminary numerical results are reported in Section \ref{se:numerical} to validate the efficiency of our proposed numerical approach. Finally, some conclusions are drawn in Section \ref{se:conclusion}.

\section{Existence of optimal controls and first-order optimality conditions}\label{se:existence and oc}
In this section, first we present some notation and known results from the literature that will be used in later analysis. Then, we prove the existence of optimal controls for (BCP) and derive the associated first-order optimality conditions. Without loss of generality, we assume that $f=0$ and $g=0$ in (\ref{state_equation}) for convenience. 
\subsection{Preliminaries}
Throughout, we denote by $L^s(\Omega)$ and $H^s(\Omega)$ the
usual Sobolev spaces for any $s>0$. The space $H_0^s(\Omega)$ denotes the completion
of $C_0^{\infty}(\Omega)$ in $H^s(\Omega)$, where $C_0^{\infty}(\Omega)$ denotes the space of all infinitely differentiable
functions over $\Omega$ with a compact support in $\Omega$. In addition, we shall also use the following vector-valued function spaces:
\begin{eqnarray*}
	&&\bm{L}^2(\Omega):=[L^2(\Omega)]^d,\\
	&&\bm{L}_{div}^2(\Omega):=\{\bm{v}\in \bm{L}^2(\Omega),\nabla\cdot\bm{v}=0~\text{in}~\Omega\}.
\end{eqnarray*}
Let $X$ be a Banach space with a norm $\|\cdot\|_X$, then the space $L^2(0, T;X)$ consists of all measurable functions $z:(0,T)\rightarrow X$ satisfying
$$
\|z\|_{L^2(0, T;X)}:=\left(\int_0^T\|z(t)\|_X^2dt \right)^{\frac{1}{2}}<+\infty.
$$
With the above notation, it is clear that the admissible set $\mathcal{U}$ can be denoted as $\mathcal{U}:=L^2(0,T; \bm{L}_{div}^2(\Omega))$.
Moreover, the space $W(0,T)$ consists of all functions $z\in L^2(0, T; H_0^1(\Omega))$ such that $\frac{\partial z}{\partial t}\in L^2(0, T; H^{-1}(\Omega))$ exists in a weak sense, i.e.  
$$
W(0,T):=\{z|z\in L^2(0,T; H_0^1(\Omega)), \frac{\partial z}{\partial t}\in L^2(0,T; H^{-1}(\Omega))\},
$$
where $H^{-1}(\Omega)(=H_0^1(\Omega)^\prime)$ is the dual space of $H_0^1(\Omega)$.

Next, we summarize some known results for the advection-reaction-diffusion equation (\ref{state_equation}) in the literature for the convenience of further analysis. 

The variational formulation of the state equation (\ref{state_equation}) reads: find $y\in W(0,T)$ such that $y(0)=\phi$ and $\forall z\in L^2(0,T;H_0^1(\Omega))$,
\begin{equation}\label{weak_form}
\int_0^T\left\langle\frac{\partial y}{\partial t}, z \right\rangle_{H^{-1}(\Omega),H_0^1(\Omega)} dt+\nu\iint_{Q} \nabla y\cdot\nabla z dxdt+\iint_{Q}\bm{v}\cdot\nabla yzdxdt+a_0\iint_{Q} yzdxdt=0,
\end{equation}
where $\left\langle\cdot,\cdot\right\rangle_{H^{-1}(\Omega),H_0^1(\Omega)}$ denotes the duality pairing between $H^{-1}(\Omega)$ and $H_0^1(\Omega)$.
The existence and uniqueness of the solution $y\in W(0,T)$ to problem (\ref{weak_form}) can be proved by standard arguments relying on the Lax-Milgram theorem, we refer to \cite{lions1971optimal} for the details. Moreover, we can define the control-to-state operator $S:\mathcal{U}\rightarrow W(0,T)$, which maps $\bm{v}$ to $y=S(\bm{v})$. Then, the objective functional $J$ in (BCP) can be reformulated as 
\begin{equation*}
J(\bm{v})=\frac{1}{2}\iint_Q|\bm{v}|^2dxdt+\frac{\alpha_1}{2}\iint_Q|S(\bm{v})-y_d|^2dxdt+\frac{\alpha_2}{2}\int_\Omega|S(\bm{v})(T)-y_T|^2dx,
\end{equation*}
and the nonlinearity of the solution operator $S$ implies that (BCP) is nonconvex.

For the solution $y\in W(0,T)$, we have the following estimates.

\begin{lemma}
	Let $\bm{v}\in L^2(0,T; \bm{L}^2_{div}(\Omega))$, then the solution $y\in W(0,T)$ of the state equation (\ref{state_equation}) satisfies the following estimate:
	\begin{equation}\label{est_y}
	\|y(t)\|_{L^2(\Omega)}^2+2\nu\int_0^t\|\nabla y(s)\|_{L^2(\Omega)}^2ds+2a_0\int_0^t\| y(s)\|_{L^2(\Omega)}^2ds=\|\phi\|_{L^2(\Omega)}^2.
	\end{equation}
\end{lemma}
\begin{proof}
	We first multiply the state equation (\ref{state_equation}) by $y(t)$, then
	applying the Green's formula in space yields
	\begin{equation}\label{e1}
	\frac{1}{2}\frac{d}{dt}\|y(t)\|_{L^2(\Omega)}^2=-\nu\|\nabla y(t)\|_{L^2(\Omega)}^2-a_0\|y(t)\|_{L^2(\Omega)}^2.
	\end{equation}
	The desired result (\ref{est_y}) can be directly obtained by integrating (\ref{e1}) over $[0,t]$.
\end{proof}

Above estimate implies that
\begin{equation}\label{bd_y}
y~\text{is bounded in}~L^2(0,T; H_0^1(\Omega)).
\end{equation} 
On the other hand, 
$$
\frac{\partial y}{\partial t}=\nu \nabla^2y-\bm{v}\cdot \nabla y-a_0y,
$$
and the right hand side is bounded in $L^2(0,T; H^{-1}(\Omega))$. Hence, 
\begin{equation}\label{bd_yt}
\frac{\partial y}{\partial t}~\text{is bounded in}~ L^2(0,T; H^{-1}(\Omega)).
\end{equation}

Furthermore, since $\nabla\cdot\bm{v}=0$, it is clear that
	$$\iint_Q\bm{v}\cdot\nabla yzdxdt=\iint_Q\nabla y\cdot (\bm{v}z)dxdt=-\iint_Q y\nabla\cdot(\bm{v}z)dxdt=-\iint_Q y(\bm{v}\cdot\nabla z)dxdt,\forall z\in L^2(0,T;H_0^1(\Omega)).$$
	
	Hence, the variational formulation (\ref{weak_form}) can be equivalently written as:“
	find $y\in W(0,T)$ such that $y(0)=\phi$ and
    $\forall z\in L^2(0,T;H_0^1(\Omega))$,
	\begin{equation*}
	\int_0^T\left\langle\frac{\partial y}{\partial t}, z \right\rangle_{H^{-1}(\Omega),H_0^1(\Omega)} dt+\nu\iint_{Q} \nabla y\cdot\nabla z dxdt-\iint_Q(\bm{v}\cdot\nabla z) ydxdt+a_0\iint_{Q} yzdxdt=0.
	\end{equation*}

\subsection{Existence of Optimal Controls}
With above preparations, we prove in this subsection the existence of optimal controls for (BCP).
For this purpose, we first show that the objective functional $J$ is weakly lower semi-continuous.
\begin{lemma}\label{wlsc}
	The objective functional $J$ given by (\ref{objective_functional}) is weakly lower semi-continuous. That is, if a sequence $\{\bm{v}_n\}$ converges weakly to $\bar{\bm{v}}$ in $L^2(0,T; \bm{L}^2_{div}(\Omega))$, we have
	$$
	J(\bar{\bm{v}})\leq \underset{n\rightarrow \infty}{\lim\inf} J(\bm{v}_n).
	$$
\end{lemma}

\begin{proof}
	Let $\{\bm{v}_n\}$ be a sequence that converges weakly to $\bar{\bm{v}}$ in $L^2(0,T;\bm{L}^2_{div}(\Omega))$ and $y_n:=y(x,t;\bm{v}_n)$ the solution of the following variational problem:
	find $y_n\in W(0,T)$ such that $y_n(0)=\phi$ and $\forall z\in L^2(0,T;H_0^1(\Omega))$,
\begin{equation}\label{seq_state}
\int_0^T\left\langle\frac{\partial y_n}{\partial t}, z \right\rangle_{H^{-1}(\Omega),H_0^1(\Omega)} dt+\nu\iint_{Q} \nabla y_n\cdot\nabla z dxdt-\iint_Q(\bm{v}_n\cdot\nabla z) y_ndxdt+a_0\iint_{Q} y_nzdxdt=0.
\end{equation}
 Moreover, it follows from (\ref{bd_y}) and (\ref{bd_yt}) that there exists a subsequence of $\{y_n\}$, still denoted by $\{y_n\}$ for convenience, such that 
	$$y_n\rightarrow\bar{y}~\text{weakly in}~ L^2(0,T; H_0^1(\Omega)),$$ 
	and 
	$$\frac{\partial y_n}{\partial t}\rightarrow\frac{\partial \bar{y}}{\partial t} ~\text{weakly in}~L^2(0,T; H^{-1}(\Omega)).$$ 
	Since $\Omega$ is bounded, it follows directly from the compactness property (also known as Rellich's Theorem) that
		$$y_n\rightarrow\bar{y}~\text{strongly in}~ L^2(0,T; L^2(\Omega)).$$ 
	Taking $\bm{v}_n\rightarrow \bar{\bm{v}}$ weakly in $L^2(0,T; \bm{L}_{div}^2(\Omega))$ into account, we can pass the limit in (\ref{seq_state}) and derive that $\bar{y}(0)=\phi$ and $\forall z\in L^2(0,T;H_0^1(\Omega))$,
\begin{equation*}
\int_0^T\left\langle\frac{\partial \bar{y}}{\partial t}, z \right\rangle_{H^{-1}(\Omega),H_0^1(\Omega)} dt+\nu\iint_{Q} \nabla \bar{y}\cdot\nabla z dxdt-\iint_Q(\bar{\bm{v}}\cdot\nabla z) \bar{y}dxdt+a_0\iint_{Q} \bar{y}zdxdt=0,
\end{equation*}
which implies that $\bar{y}$ is the solution of the state equation (\ref{state_equation}) associated with $\bar{\bm{v}}$.

	Since any norm of a Banach space is weakly lower semi-continuous, we have that
	\begin{equation*}
	\begin{aligned}
	&\underset{n\rightarrow \infty}{\lim\inf} J(\bm{v}_n)\\
	= &\underset{n\rightarrow \infty}{\lim\inf}\left( \frac{1}{2}\iint_Q|\bm{v}_n|^2dxdt+\frac{\alpha_1}{2}\iint_Q|y_n-y_d|^2dxdt+\frac{\alpha_2}{2}\int_\Omega|y_n(T)-y_T|^2dx\right)\\
	\geq& \frac{1}{2}\iint_Q|\bar{\bm{v}}|^2dxdt+\frac{\alpha_1}{2}\iint_Q|\bar{y}-y_d|^2dxdt+\frac{\alpha_2}{2}\int_\Omega|\bar{y}(T)-y_T|^2dx\\
	=& J(\bar{\bm{v}}).
	\end{aligned}
	\end{equation*}
	We thus obtain that the objective functional $J$ is weakly lower semi-continuous and complete the proof.
\end{proof}

Now, we are in a position to prove the existence of an optimal solution $\bm{u}$ to (BCP).

\begin{theorem}\label{thm_existence}
	There exists at least one optimal control $\bm{u}\in \mathcal{U}=L^2(0,T; \bm{L}_{div}^2(\Omega))$ such that $J(\bm{u})\leq J(\bm{v}),\forall\bm{v}\in \mathcal{U}$.
\end{theorem}
\begin{proof}
	We first observe that $J(\bm{v})\geq 0,\forall\bm{v}\in \mathcal{U}$, then the infimum of $J(\bm{v})$ exists and we denote it as
	$$
	j=\inf_{\bm{v}\in\mathcal{U}}J(\bm{v}),
	$$
	and there is a minimizing sequence $\{\bm{v}_n\}\subset\mathcal{U}$ such that
	$$
	\lim_{n\rightarrow \infty}J(\bm{v}_n)=j.
	$$
	This fact, together with $\frac{1}{2}\|\bm{v}_n\|^2_{L^2(0,T; \bm{L}^2_{div}(\Omega))}\leq J(\bm{v}_n)$, implies that $\{\bm{v}_n\}$ is bounded in $L^2(0,T;\bm{L}^2_{div}(\Omega))$. Hence, there exists a subsequence, still denoted by $\{\bm{v}_n\}$, that converges weakly to $\bm{u}$ in $L^2(0,T; \bm{L}^2_{div}(\Omega))$. It follows from Lemma \ref{wlsc} that $J$ is weakly lower semi-continuous and we thus have
	$$
	J(\bm{u})\leq \underset{n\rightarrow \infty}{\lim\inf} J(\bm{v}_n)=j.
	$$
	Since $\bm{u}\in\mathcal{U}$, we must have $J(\bm{u})=j$, and $\bm{u}$ is therefore an optimal control. 
\end{proof}
We  note that the uniqueness of optimal control $\bm{u}$ cannot be guaranteed and only a local optimal solution can be pursued because the objective functional $J$ is nonconvex due to the nonlinear relationship between the state $y$ and the control $\bm{v}$. 

\subsection{First-order Optimality Conditions}
Let $DJ(\bm{v})$ be the first-order differential of $J$ at $\bm{v}$ and $\bm{u}$ an optimal control of (BCP). It is clear that the first-order optimality condition of (BCP) reads
\begin{equation*}
DJ(\bm{u})=0.
\end{equation*}
In the sequel of this subsection, we discuss the computation of $DJ(\bm{v})$, which will play an important role in subsequent sections.

To compute $DJ(\bm{v})$, we employ a formal perturbation analysis as in \cite{glowinski2008exact}. First, let $\delta \bm{v}\in \mathcal{U}$ be a perturbation of $\bm{v}\in \mathcal{U}$, we clearly have
\begin{equation}\label{Dj and delta j}
\delta J(\bm{v})=\iint_{Q}DJ(\bm{v})\cdot\delta \bm{v} dxdt,
\end{equation}
and also
\begin{eqnarray}{\label{def_delta_j}}
\begin{aligned}
&\delta J(\bm{v})=\iint_{Q}\bm{v}\cdot\delta \bm{v} dxdt+\alpha_1\iint_{Q}(y-y_d)\delta y dxdt+\alpha_2\int_\Omega(y(T)-y_T)\delta y(T)dx,
\end{aligned}
\end{eqnarray}
in which $\delta y$ is the solution of
\begin{flalign}\label{perturbation_state_eqn}
&\left\{
\begin{aligned}
\frac{\partial \delta y}{\partial t}-\nu \nabla^2\delta y+\delta \bm{v}\cdot \nabla y+\bm{v}\cdot\nabla\delta y+a_0\delta y&=0\quad \text{in}\quad Q, \\
\delta y&=0\quad \text{on}\quad \Sigma,\\
\delta y(0)&=0.
\end{aligned}
\right.
\end{flalign}

Consider now a function $p$ defined over $\overline{Q}$ (the closure of $Q$); and assume that $p$ is a differentiable function of $x$ and $t$. Multiplying both sides of the first equation in (\ref{perturbation_state_eqn}) by $p$ and integrating over $Q$, we obtain
\begin{equation*}
\iint_{Q}p\frac{\partial }{\partial t}\delta ydxdt-\nu\iint_{Q}p \nabla^2\delta ydxdt+\iint_Q\delta \bm{v}\cdot \nabla ypdxdt+\iint_Q\bm{v}\cdot\nabla\delta ypdxdt+a_0\iint_{Q}p\delta ydxdt=0.
\end{equation*}
Integration by parts in time and application of Green's formula in space yield
\begin{eqnarray}{\label{weakform_p}}
\begin{aligned}
\int_\Omega p(T)\delta y(T)dx-\int_\Omega p(0)\delta y(0)dx+\iint_{Q}\Big[ -\frac{\partial p}{\partial t}
-\nu\nabla^2p-\bm{v}\cdot\nabla p+a_0p\Big]\delta ydxdt\\
+\iint_Q\delta \bm{v}\cdot \nabla ypdxdt-\nu\iint_{\Sigma}(\frac{\partial\delta y}{\partial \bm{n}}p-\frac{\partial p}{\partial \bm{n}}\delta y)dxdt+\iint_\Sigma p\delta y\bm{v}\cdot \bm{n}dxdt=0.
\end{aligned}
\end{eqnarray}
where $\bm{n}$ is the unit outward  normal vector at $\Gamma$.

Next, let us assume that the function $p$ is the solution to the following adjoint system
\begin{flalign}\label{adjoint_equation}
&\qquad \left\{
\begin{aligned}
-\frac{\partial p}{\partial t}
-\nu\nabla^2p-\bm{v}\cdot\nabla p +a_0p&=\alpha_1(y-y_d)~ \text{in}~ Q, \\
p&=0~\qquad\quad\quad\text{on}~ \Sigma,\\
p(T)&=\alpha_2(y(T)-y_T).
\end{aligned}
\right.
\end{flalign}
It follows from (\ref{def_delta_j}), (\ref{perturbation_state_eqn}), (\ref{weakform_p}) and (\ref{adjoint_equation}) that
\begin{equation*}
\delta J(\bm{v})=\iint_{Q}(\bm{v}-p\nabla y)\cdot\delta \bm{v} dxdt.
\end{equation*}
which, together with (\ref{Dj and delta j}), implies that
\begin{equation}\label{gradient}
\left\{
\begin{aligned}
&DJ(\bm{v})\in \mathcal{U},\\
&\iint_QDJ(\bm{v})\cdot \bm{z}dxdt=\iint_Q(\bm{v}-p\nabla y)\cdot \bm{z}dxdt,\forall \bm{z}\in \mathcal{U}.
\end{aligned}
\right.
\end{equation}

From the discussion above, the first-order optimality condition of (BCP) can be summarized as follows.
\begin{theorem}
	Let $\bm{u}\in \mathcal{U}$ be a solution of (BCP). Then, it satisfies the following optimality condition
	\begin{equation*}
	\iint_Q(\bm{u}-p\nabla y)\cdot \bm{z}dxdt=0,\forall \bm{z}\in \mathcal{U},
	\end{equation*}
	where $y$ and $p$ are obtained from $\bm{u}$ via the solutions of the
	following two parabolic equations:
	\begin{flalign*}\tag{state equation}
	&\quad\qquad\qquad\qquad\left\{
	\begin{aligned}
\frac{\partial y}{\partial t}-\nu \nabla^2y+\bm{u}\cdot \nabla y+a_0y&=f\quad \text{in}~ Q, \\
y&=g\quad \text{on}~\Sigma,\\
y(0)&=\phi,
	\end{aligned}
	\right.&
	\end{flalign*}
	and
	\begin{flalign*}\tag{adjoint equation}
	&\qquad\qquad\qquad\left\{
	\begin{aligned}
	-\frac{\partial p}{\partial t}
	-\nu\nabla^2p-\bm{u}\cdot\nabla p +a_0p&=\alpha_1(y-y_d)\quad \text{in}~ Q, \\
	p&=0
	 \quad\qquad\qquad\text{on}~\Sigma,\\
	p(T)&=\alpha_2(y(T)-y_T).
	\end{aligned}
	\right.&
	\end{flalign*}
\end{theorem}

\section{An Implementable Nested Conjugate Gradient Method}\label{se:cg}
In this section, we discuss the application of a CG strategy to solve (BCP). In particular, we elaborate on the computation of the gradient and the stepsize at each CG iteration, and thus obtain an easily implementable algorithm.
\subsection{A Generic Conjugate Gradient Method for (BCP)}
Conceptually, implementing the CG method to (BCP), we readily obtain the following  algorithm:
\begin{enumerate}
	\item[\textbf{(a)}] Given $\bm{u}^0\in \mathcal{U}$.
	\item [\textbf{(b)}] Compute $\bm{g}^0=DJ(\bm{u}^0)$. If $DJ(\bm{u}^0)=0$, then $\bm{u}=\bm{u}^0$; otherwise set $\bm{w}^0=\bm{g}^0$.
	\item[]\noindent For $k\geq 0$, $\bm{u}^k,\bm{g}^k$ and $\bm{w}^k$ being known, the last two different from $\bm{0}$, one computes $\bm{u}^{k+1}, \bm{g}^{k+1}$ and $\bm{w}^{k+1}$ as follows:
	\item[\textbf{(c)}] Compute the stepsize $\rho_k$ by solving the following optimization problem
	\begin{flalign}\label{op_step}
	&\left\{
	\begin{aligned}
	& \rho_k\in \mathbb{R}, \\
	&J(\bm{u}^k-\rho_k\bm{w}^k)\leq J(\bm{u}^k-\rho \bm{w}^k), \forall \rho\in \mathbb{R}.
	\end{aligned}
	\right.
	\end{flalign}
    \item[\textbf{(d)}] Update $\bm{u}^{k+1}$ and $\bm{g}^{k+1}$, respectively, by
    $$\bm{u}^{k+1}=\bm{u}^k-\rho_k \bm{w}^k,$$
     and
     $$\bm{g}^{k+1}=DJ(\bm{u}^{k+1}).$$
    \item[] If $DJ(\bm{u}^{k+1})=0$, take $\bm{u}=\bm{u}^{k+1}$; otherwise,
    \item[\textbf{(e)}] Compute $$\beta_k=\frac{\iint_{Q}|\bm{g}^{k+1}|^2dxdt}{\iint_{Q}|\bm{g}^k|^2dxdt},$$
    and then update
    $$\bm{w}^{k+1}=\bm{g}^{k+1}+\beta_k \bm{w}^k.$$
    \item[] Do $k+1\rightarrow k$ and return to (\textbf{c}).
\end{enumerate}

The above iterative method looks very simple, but practically, the implementation of the CG method (\textbf{a})--(\textbf{e}) for the solution of (BCP) is nontrivial. In particular, it is numerically challenging to compute $DJ(\bm{v})$, $\forall \bm{v}\in\mathcal{U}$ and $\rho_k$ as illustrated below. We shall discuss how to address these two issues in the following part of this section.

\subsection{Computation of $DJ(\bm{v})$}\label{com_gra}
It is clear that the implementation of the generic CG method (\textbf{a})--(\textbf{e}) for the solution of (BCP) requires the knowledge of $DJ(\bm{v})$ for various $\bm{v}\in \mathcal{U}$, and this has been conceptually provided in (\ref{gradient}). However, it is numerically challenging to compute $DJ(\bm{v})$ by (\ref{gradient}) due to the restriction $\nabla\cdot DJ(\bm{v})=0$ which ensures that all iterates $\bm{u}^k$ of the CG method meet the additional divergence-free constraint $\nabla\cdot \bm{u}^k=0$. In this subsection, we show that equation (\ref{gradient}) can be reformulated as a saddle point problem by introducing a Lagrange multiplier associated with the constraint $\nabla\cdot DJ(\bm{v})=0$. Then, a preconditioned CG method is proposed to solve this saddle point problem.

We first note that equation (\ref{gradient}) can be equivalently reformulated as
\begin{equation}\label{gradient_e}
\left\{
\begin{aligned}
&DJ(\bm{v})(t)\in \mathbb{S},\\
&\int_\Omega DJ(\bm{v})(t)\cdot \bm{z}dx=\int_\Omega(\bm{v}(t)-p(t)\nabla y(t))\cdot \bm{z}dx,\forall \bm{z}\in \mathbb{S},
\end{aligned}
\right.
\end{equation}
where 
\begin{equation*}
\mathbb{S}=\{\bm{z}|\bm{z}\in [L^2(\Omega)]^d, \nabla\cdot\bm{z}=0\}.
\end{equation*}
Clearly, problem (\ref{gradient_e}) is a particular case of 
\begin{equation}\label{gradient_e2}
\left\{
\begin{aligned}
&\bm{g}\in \mathbb{S},\\
&\int_\Omega \bm{g}\cdot \bm{z}dx=\int_\Omega\bm{f}\cdot \bm{z}dx,\forall \bm{z}\in \mathbb{S},
\end{aligned}
\right.
\end{equation}
with $\bm{f}$ given in $[L^2(\Omega)]^d$.

 Introducing a Lagrange multiplier $\lambda\in H_0^1(\Omega)$ associated with the constraint $\nabla\cdot\bm{z}=0$, it is clear that problem (\ref{gradient_e2}) is equivalent to the following saddle point problem
\begin{equation}\label{gradient_e3}
\left\{
\begin{aligned}
&(\bm{g},\lambda)\in [L^2(\Omega)]^d\times H_0^1(\Omega),\\
&\int_\Omega \bm{g}\cdot \bm{z}dx=\int_\Omega\bm{f}\cdot \bm{z}dx+\int_\Omega\lambda\nabla\cdot \bm{z}dx,\forall \bm{z}\in [L^2(\Omega)]^d,\\
&\int_\Omega\nabla\cdot \bm{g}qdx=0,\forall q\in H_0^1(\Omega),
\end{aligned}
\right.
\end{equation}
which is actually a Stokes type problem.

In order to solve problem (\ref{gradient_e3}), we advocate a CG method inspired from \cite{glowinski2003, glowinski2015}.
For this purpose,
one has to specify the inner product to be used over $H_0^1(\Omega)$.
As discussed in \cite{glowinski2003}, the usual $L^2$-inner product,
namely, $\{q,q'\}\rightarrow\int_\Omega qq'dx$
leads to a CG method with poor convergence properties. Indeed, using some arguments similar to those in \cite{glowinski1992,glowinski2003}, we can show that the saddle point problem (\ref{gradient_e3}) can be reformulated as a linear variational problem in terms of the Lagrange multiplier $\lambda$. The corresponding coefficient matrix after space
discretization with mesh size $h$ has a condition number of the order of $h^{-2}$, which is ill-conditioned especially for small $h$ and makes the CG method converges fairly slow. Hence, preconditioning is necessary for solving problem (\ref{gradient_e3}). As suggested in \cite{glowinski2003}, we choose $-\nabla\cdot\nabla$ as a preconditioner for  problem (\ref{gradient_e3}), and the corresponding preconditioned CG method operates in the space $H_0^1(\Omega)$ equipped with the inner product $\{q,q'\}\rightarrow\int_\Omega\nabla q\cdot\nabla q'dx$ and the associated norm $\|q\|_{H_0^1(\Omega)}=(\int_\Omega|\nabla q|^2dx)^{1/2}, \forall q,q'\in H_0^1(\Omega)$. The resulting algorithm reads as:
\begin{enumerate}
	\item [\textbf{G1}] Choose $\lambda^0\in H_0^1(\Omega)$.
	\item [\textbf{G2}] Solve 
	\begin{equation*}
	\left\{
	\begin{aligned}
	&\bm{g}^0\in [L^2(\Omega)]^d,\\
	&\int_\Omega \bm{g}^0\cdot \bm{z}dx=\int_\Omega\bm{f}\cdot \bm{z}dx+\int_\Omega\lambda^0\nabla\cdot \bm{z}dx,\forall \bm{z}\in [L^2(\Omega)]^d,
	\end{aligned}
	\right.
	\end{equation*}
	and
	\begin{equation*}
	\left\{
	\begin{aligned}
	&r^0\in H_0^1(\Omega),\\
	&\int_\Omega \nabla r^0\cdot \nabla qdx=\int_\Omega\nabla\cdot \bm{g}^0qdx,\forall q\in H_0^1(\Omega).
	\end{aligned}
	\right.
	\end{equation*}
	\smallskip
	If $\frac{\int_\Omega|\nabla r^0|^2dx}{\max\{1,\int_\Omega|\nabla \lambda^0|^2dx\}}\leq tol_1$, take $\lambda=\lambda^0$ and $\bm{g}=\bm{g}^0$; otherwise set $w^0=r^0$.
	For $k\geq 0$, $\lambda^k,\bm{g}^k, r^k$ and $w^k$ being known with the last two different from 0, we compute
	$\lambda^{k+1},\bm{g}^{k+1}, r^{k+1}$ and if necessary $w^{k+1}$, as follows:
	\smallskip
	\item[\textbf{G3}] Solve 
	\begin{equation*}
	\left\{
	\begin{aligned}
	&\bar{\bm{g}}^k\in [L^2(\Omega)]^d,\\
	&\int_\Omega \bar{\bm{g}}^k\cdot \bm{z}dx=\int_\Omega w^k\nabla\cdot \bm{z}dx,\forall \bm{z}\in [L^2(\Omega)]^d,
	\end{aligned}
	\right.
	\end{equation*}
	and
	\begin{equation*}
	\left\{
	\begin{aligned}
	&\bar{r}^k\in H_0^1(\Omega),\\
	&\int_\Omega \nabla \bar{r}^k\cdot \nabla qdx=\int_\Omega\nabla\cdot \bar{\bm{g}}^kqdx,\forall q\in H_0^1(\Omega),
	\end{aligned}
	\right.
	\end{equation*}
	and compute the stepsize via
	$$
	\eta_k=\frac{\int_\Omega|\nabla r^k|^2dx}{\int_\Omega\nabla\bar{r}^k\cdot\nabla w^kdx}.
	$$
	\item[\textbf{G4}] Update $\lambda^k, \bm{g}^k$ and $r^k$ via
	$$\lambda^{k+1}=\lambda^k-\eta_kw^k,\bm{g}^{k+1}=\bm{g}^k-\eta_k \bar{\bm{g}}^k,~\text{and}~r^{k+1}=r^k-\eta_k \bar{r}^k.$$
	\smallskip
	If $\frac{\int_\Omega|\nabla r^{k+1}|^2dx}{\max\{1,\int_\Omega|\nabla r^0|^2dx\}}\leq tol_1$, take $\lambda=\lambda^{k+1}$ and $\bm{g}=\bm{g}^{k+1}$; otherwise, 
	\item[\textbf{G5}] Compute
	$$\gamma_k=\frac{\int_\Omega|\nabla r^{k+1}|^2dx}{\int_\Omega|\nabla r^k|^2dx},$$
	and update $w^k$ via
	$$w^{k+1}=r^{k+1}+\gamma_k w^{k}.$$
	Do $k+1\rightarrow k$ and return to \textbf{G3}.
\end{enumerate}

Clearly, one only needs to solve two simple linear equations at each iteration of the preconditioned CG algorithm (\textbf{G1})-(\textbf{G5}), which implies that the algorithm is easy and cheap to implement. Moreover, due to the well-chosen preconditioner $-\nabla\cdot\nabla$, one can expect the above preconditioned CG algorithm to have a fast convergence; this will be validated by the numerical experiments reported in Section \ref{se:numerical}.

\subsection{Computation of the Stepsize $\rho_k$}\label{com_step}
Another crucial step to implement the CG method \textbf{(a)}--\textbf{(e)} is the computation of the stepsize $\rho_k$. It is the solution of the optimization problem (\ref{op_step}) which is numerically expensive to be solved exactly or up to a high accuracy. For instance, to solve (\ref{op_step}), one may consider the Newton method applied to the
solution of
$$
H_k'(\rho_k)=0,
$$
where
$$H_k(\rho)=J(\bm{u}^k-\rho\bm{w}^k).$$
The Newton method requires the second-order derivative $H_k''(\rho)$ which can be computed
via an iterated adjoint technique requiring the solution of \emph{four} parabolic
problems per Newton's iteration. Hence, the implementation of the Newton method is numerically expensive. 

The high computational load for solving (\ref{op_step}) motivates us to implement certain
stepsize rule to determine an approximation of $\rho_k$. Here, we advocate the following procedure to
compute an approximate stepsize $\hat{\rho}_k$. 

 For a given $\bm{w}^k\in\mathcal{U}$, we replace the state $y=S(\bm{u}^k-\rho\bm{w}^k)$ in the objective functional $J(\bm{u}^k-\rho\bm{w}^k)$ by
$$
S(\bm{u}^k)-\rho S'(\bm{u}^k)\bm{w}^k,
$$
which is indeed the linearization of the mapping $\rho \mapsto S(\bm{u}^k - \rho \bm{w}^k)$ at $\rho= 0$. We thus obtain the following quadratic approximation of $H_k(\rho)$:
\begin{equation}\label{q_rho}
Q_k(\rho):=\frac{1}{2}\iint_Q|\bm{u}^k-\rho \bm{w}^k|^2dxdt+\frac{\alpha_1}{2}\iint_Q|y^k-\rho z^k-y_d|^2dxdt+\frac{\alpha_2}{2}\int_\Omega|y^k(T)-\rho z^k(T)-y_T|^2dx,
\end{equation}
where $y^k=S(\bm{u}^k)$ is the solution of the state equation (\ref{state_equation}) associated with $\bm{u}^k$, and $z^k=S'(\bm{u}^k)\bm{w}^k$ satisfies the following linear parabolic problem
\begin{flalign}\label{linear_state}
&\left\{
\begin{aligned}
\frac{\partial  z^k}{\partial t}-\nu \nabla^2 z^k+\bm{w}^k\cdot \nabla y^k    +\bm{u}^k\cdot\nabla z^k+a_0 z^k&=0\quad \text{in}\quad Q, \\
z^k&=0\quad \text{on}\quad \Sigma,\\
z^k(0)&=0.
\end{aligned}
\right.
\end{flalign}

Then, it is easy to show that the equation
$
Q_k'(\rho)=0
$
admits a unique solution
\begin{equation}\label{step_size}
\hat{\rho}_k =\frac{\iint_Q\bm{g}^k\cdot \bm{w}^k dxdt}{\iint_Q|\bm{w}^k|^2dxdt+ \alpha_1\iint_Q|z^k|^2dxdt+\alpha_2\int_\Omega|z^k(T)|^2dx},
\end{equation}
and we take $\hat{\rho}_k$, which is clearly an approximation of $\rho_k$, as the stepsize in each CG iteration. 

Altogether, with the stepsize given by (\ref{step_size}), every iteration of the resulting CG algorithm requires solving only  \emph{three} parabolic problems, namely,  the state equation (\ref{state_equation}) forward in time and the associated adjoint equation (\ref{adjoint_equation})
backward in time for the computation of $\bm{g}^k$, and to solving the linearized
parabolic equation (\ref{linear_state}) forward in time for the stepsize $\hat{\rho}_k$. For comparison, if the Newton method is employed to compute the stepsize $\rho_k$ by solving (\ref{op_step}), at least \emph{six} parabolic problems are required to be solved at each iteration of the CG method, which is much more  expensive numerically.
\begin{remark}
To find an appropriate stepsize, a natural idea is to employ some line search strategies, such as the backtracking strategy based on the Armijo--Goldstein condition or the Wolf condition, see e.g., \cite{nocedal2006}. It is worth noting that these line search strategies require the evaluation of $J(\bm{v})$ repeatedly, which is numerically expensive because every evaluation of $J(\bm{v})$ for a given $\bm{v}$ requires solving the state equation (\ref{state_equation}). Moreover, we have implemented the CG method for solving (BCP) with various line search strategies and observed from the numerical results that line search strategies always lead to tiny stepsizes making extremely slow the convergence of the CG method.
\end{remark}
\subsection{A Nested CG Method for Solving (BCP)}
Following Sections \ref{com_gra} and \ref{com_step}, we advocate the following nested CG method for solving (BCP):
\begin{enumerate}
	\item[\textbf{I.}] Given $\bm{u}^0\in \mathcal{U}$.
	\item[\textbf{II.}] Compute $y^0$ and $p^0$ by solving the state equation (\ref{state_equation}) and the adjoint equation (\ref{adjoint_equation}) corresponding to $\bm{u}^0$.
	Then, for a.e. $t \in(0, T)$, solve 
	\begin{equation*}
	\left\{
	\begin{aligned}
	&\bm{g}^0(t)\in \mathbb{S},\\
	&\int_\Omega \bm{g}^0(t)\cdot \bm{z}dx=\int_\Omega(\bm{u}^0(t)-p^0(t)\nabla y^0(t))\cdot \bm{z}dx,\forall \bm{z}\in \mathbb{S},
	\end{aligned}
	\right.
	\end{equation*}
	by the preconditioned CG algorithm (\textbf{G1})--(\textbf{G5}); and set $\bm{w}^0=\bm{g}^0.$
	
	\medskip
	\noindent For $k\geq 0$, $\bm{u}^k, \bm{g}^k$ and $\bm{w}^k$ being known, the last two different from $\bm{0}$, one computes $\bm{u}^{k+1}, \bm{g}^{k+1}$ and $\bm{w}^{k+1}$ as follows:
	\medskip
	\item[\textbf{III.}] Compute the stepsize $\hat{\rho}_k$ by (\ref{step_size}).
	\item[\textbf{IV.}] Update $\bm{u}^{k+1}$ by $$\bm{u}^{k+1}=\bm{u}^k-\hat{\rho}_k\bm{w}^k.$$
	Compute $y^{k+1}$ and $p^{k+1}$ by solving the state equation (\ref{state_equation}) and the adjoint equation (\ref{adjoint_equation}) corresponding to $\bm{u}^{k+1}$;
%
	and for a.e. $t \in(0, T)$, solve 
	\begin{equation*}
	\left\{
	\begin{aligned}
	&\bm{g}^{k+1}(t)\in \mathbb{S},\\
	&\int_\Omega \bm{g}^{k+1}(t)\cdot \bm{z}dx=\int_\Omega(\bm{u}^{k+1}(t)-p^{k+1}(t)\nabla y^{k+1}(t))\cdot \bm{z}dx,\forall \bm{z}\in \mathbb{S},
	\end{aligned}
	\right.
	\end{equation*}
	by the preconditioned CG algorithm (\textbf{G1})--(\textbf{G5}).
	
	\medskip
	\noindent If $\frac{\iint_Q|\bm{g}^{k+1}|^2dxdt}{\iint_Q|\bm{g}^{0}|^2dxdt}\leq tol$, take $\bm{u} = \bm{u}^{k+1}$; else
	\medskip
	\item[\textbf{V.}] Compute $$\beta_k=\frac{\iint_Q|\bm{g}^{k+1}|^2dxdt}{\iint_Q|\bm{g}^{k}|^2dxdt},~\text{and}~\bm{w}^{k+1} = \bm{g}^{k+1} + \beta_k\bm{w}^k.$$
	 Do $k+1\rightarrow k$ and return to \textbf{III}.
\end{enumerate}

\section{Space and time discretizations}\label{se:discretization}
In this section, we discuss first the numerical discretization of the bilinear optimal control problem (BCP). We achieve the
time discretization by a semi-implicit finite difference method and the space discretization by a piecewise linear finite element
method. Then, we discuss an implementable nested CG method for solving the fully discrete bilinear optimal control problem.
\subsection{Time Discretization of (BCP)}
First, we define a time discretization step $\Delta t$ by $\Delta t= T/N$,
with $N$ a positive integer. Then, we approximate the control space $\mathcal{U}=L^2(0, T;\mathbb{S})$ by 
$
\mathcal{U}^{\Delta t}:=(\mathbb{S})^N;
$
and
equip $\mathcal{U}^{\Delta t}$ with the following inner product
$$
(\bm{v},\bm{w})_{\Delta t} = \Delta t \sum^N_{n=1}\int_\Omega \bm{v}_n\cdot \bm{w}_ndx, \quad\forall \bm{v}= \{\bm{v}_n\}^N_{n=1}, \bm{w} = \{\bm{w}_n\}^N_{n=1}
 \in\mathcal{U}^{\Delta t},
$$
and the norm
$$
\|\bm{v}\|_{\Delta t} = \left(\Delta t \sum^N_{n=1}\int_\Omega |\bm{v}_n|^2dx\right)^{\frac{1}{2}}, \quad\forall \bm{v}= \{\bm{v}_n\}^N_{n=1}
\in\mathcal{U}^{\Delta t}.
$$
Then, (BCP) is approximated by the following semi-discrete bilinear control problem (BCP)$^{\Delta t}$:
\begin{flalign*}
&\hspace{-4.5cm}\text{(BCP)}^{\Delta t}\qquad\qquad\qquad\qquad\left\{
\begin{aligned}
& \bm{u}^{\Delta t}\in \mathcal{U}^{\Delta t}, \\
&J^{\Delta t}(\bm{u}^{\Delta t})\leq J^{\Delta t}(\bm{v}),\forall \bm{v}=\{\bm{v}_n\}_{n=1}^N\in\mathcal{U}^{\Delta t},
\end{aligned}
\right.
\end{flalign*}
where the cost functional $J^{\Delta t}$ is defined by
\begin{equation*}
 J^{\Delta t}(\bm{v})=\frac{1}{2}\Delta t \sum^N_{n=1}\int_\Omega |\bm{v}_n|^2dx+\frac{\alpha_1}{2}\Delta t \sum^N_{n=1}\int_\Omega |y_n-y_d^n|^2dx+\frac{\alpha_2}{2}\int_\Omega|y_N-y_T|^2dx,
\end{equation*}
with $\{y_n\}^N_{n=1}$ the solution of the following semi-discrete state
equation: $y_0=\phi$; then for $n=1,\ldots,N$, with $y_{n-1}$ being known, we obtain $y_n$ from the
solution of the following linear elliptic problem:
\begin{flalign}\label{state_semidis}
&\left\{
\begin{aligned}
\frac{{y}_n-{y}_{n-1}}{\Delta t}-\nu \nabla^2{y}_n+\bm{v}_n\cdot\nabla{y}_{n-1}+a_0{y}_{n-1}&= f_n\quad \text{in}\quad \Omega, \\
y_n&=g_n\quad \text{on}\quad \Gamma.
\end{aligned}
\right.
\end{flalign}
\begin{remark}
	For simplicity, we have chosen a one-step semi-explicit
	scheme to discretize system (\ref{state_equation}). This scheme is
	first-order accurate and reasonably robust, once combined to an appropriate
	space discretization. The application of second-order accurate time discretization schemes to optimal control problems has been discussed in e.g., \cite{carthelglowinski1994}.
\end{remark}
\begin{remark}
	At each step of scheme (\ref{state_semidis}), we only need to solve a simple
	linear elliptic problem to obtain $y_n$ from $y_{n-1}$, and there is no particular
	difficulty in solving such a problem.
\end{remark} 
The existence of a solution to the semi-discrete bilinear optimal control problem (BCP)$^{\Delta t}$ can be proved in a similar way as what we have done for the continuous case. 
Let $\bm{u}^{\Delta t}$ be a solution of (BCP)$^{\Delta t}$, then it verifies the following first-order optimality condition:
\begin{equation*}
DJ^{\Delta t}(\bm{u}^{\Delta t}) = 0,
\end{equation*}
where $DJ^{\Delta t}(\bm{v})$ is the first-order differential of the functional $J^{\Delta t}$ at $\bm{v}\in\mathcal{U}^{\Delta t}$.

Proceeding as in the continuous case, we can show that $DJ^{\Delta t}(\bm{v})=\{\bm{g}_n\}_{n=1}^N\in\mathcal{U}^{\Delta t}$ where
\begin{equation*}
\left\{
\begin{aligned}
&\bm{g}_n\in \mathbb{S},\\
&\int_\Omega \bm{g}_n\cdot \bm{w}dx=\int_\Omega(\bm{v}_n-p_n\nabla y_{n-1})\cdot \bm{w}dx,\forall\bm{w}\in \mathbb{S},
\end{aligned}
\right.
\end{equation*}
and the vector-valued function $\{p_n\}^N_{n=1}$ is the solution of the semi-discrete adjoint system below:
\begin{equation*}
{p}_{N+1}=\alpha_2({y}_N-y_T);
\end{equation*}
for $n=N$, solve
 \begin{flalign*}
\qquad \left\{
\begin{aligned}
\frac{{p}_N-{p}_{N+1}}{\Delta t}-\nu \nabla^2{p}_N&= \alpha_1({y}_N-y_d^N)&\quad \text{in}\quad \Omega, \\
p_N&=0&\quad \text{on}\quad \Gamma,
\end{aligned}
\right.
\end{flalign*}
and for $n=N-1,\cdots,1,$ solve
\begin{flalign*}
\qquad \left\{
\begin{aligned}
\frac{{p}_n-{p}_{n+1}}{\Delta t}-\nu\nabla^2{p}_n-\bm{v}_{n+1}\cdot\nabla{p}_{n+1}+a_0{p}_{n+1}&=  \alpha_1({y}_n-y_d^n)&\quad \text{in}\quad \Omega, \\
p_n&=0&\quad \text{on}\quad \Gamma.
\end{aligned}
\right.
\end{flalign*}

\subsection{Space Discretization of (BCP)$^{\Delta t}$}

In this subsection, we discuss the space discretization of (BCP)$^{\Delta t}$, obtaining thus a full space-time discretization of (BCP).
 For simplicity, we suppose from now on that $\Omega$ is a polygonal
domain of $\mathbb{R}^2$ (or has been approximated by a
family of such domains). 

Let $\mathcal{T}_H$ be a classical triangulation of $\Omega$, with $H$ the largest length of the edges of the triangles of $\mathcal{T}_H$. From $\mathcal{T}_{H}$ we construct $\mathcal{T}_{h}$ with $h=H/2$ by joining the mid-points of the
edges of the triangles of $\mathcal{T}_{H}$.

 We first consider the finite element space $V_h$ defined by
\begin{equation*}
	V_h = \{\varphi_h| \varphi_h\in C^0(\bar{\Omega}); { \varphi_h\mid}_{\mathbb{T}}\in P_1, \forall\, {\mathbb{T}}\in\mathcal{T}_h\}
\end{equation*}
with $P_1$ the space of the polynomials of two variables of degree $\leq 1$. Two useful sub-spaces of $V_h$
are
\begin{equation*}
V_{0h} =\{\varphi_h| \varphi_h\in V_h, \varphi_h\mid_{\Gamma}=0\}:=V_h\cap H_0^1(\Omega),
\end{equation*}
and (assuming that $g(t)\in C^0(\Gamma)$)
\begin{eqnarray*}
V_{gh}(t) =\{\varphi_h| \varphi_h\in V_h, \varphi_h(Q)=g(Q,t), \forall\, Q ~\text{vertex of} ~\mathcal{T}_h~\text{located on}~\Gamma \}.
\end{eqnarray*}
In order to construct the discrete control space, we introduce first
\begin{equation*}
\Lambda_H = \{\varphi_H| \varphi_H\in C^0(\bar{\Omega}); { \varphi_H\mid}_{\mathbb{T}}\in P_1, \forall\, {\mathbb{T}}\in\mathcal{T}_H\},~\text{and}~\Lambda_{0H} =\{\varphi_H| \varphi_H\in \Lambda_H, \varphi_H\mid_{\Gamma}=0\}.
\end{equation*}
Then, the discrete control space $\mathcal{U}_h^{\Delta t}$ is defined by
\begin{equation*}
\mathcal{U}_h^{\Delta t}=(\mathbb{S}_h)^N,~\text{with}~\mathbb{S}_h=\{\bm{v}_h|\bm{v}_h\in V_h\times V_h,\int_\Omega \nabla\cdot\bm{v}_hq_Hdx\left(=-\int_\Omega\bm{v}_h\cdot\nabla q_Hdx\right)=0,\forall q_H\in \Lambda_{0H}\}.
\end{equation*}

With the above finite element spaces, we approximate (BCP) and (BCP)$^{\Delta t}$ by (BCP)$_h^{\Delta t}$ defined by
\begin{flalign*}
&\hspace{-4.2cm}\text{(BCP)}_h^{\Delta t}\qquad\qquad\qquad\qquad\qquad\qquad\left\{
\begin{aligned}
& \bm{u}_h^{\Delta t}\in \mathcal{U}_h^{\Delta t}, \\
&J_h^{\Delta t}(\bm{u}_h^{\Delta t})\leq J_h^{\Delta t}(\bm{v}_h^{\Delta t}),\forall \bm{v}_h^{\Delta t}\in\mathcal{U}_h^{\Delta t},
\end{aligned}
\right.
\end{flalign*}
where the fully discrete cost functional $J_h^{\Delta t}$ is defined by 
\begin{equation}\label{obj_fuldis}
J_h^{\Delta t}(\bm{v}_h^{\Delta t})=\frac{1}{2}\Delta t \sum^N_{n=1}\int_\Omega |\bm{v}_{n,h}|^2dx+\frac{\alpha_1}{2}\Delta t \sum^N_{n=1}\int_\Omega |y_{n,h}-y_d^n|^2dx+\frac{\alpha_2}{2}\int_\Omega|y_{N,h}-y_T|^2dx
\end{equation}
with $\{y_{n,h}\}^N_{n=1}$ the solution of the following fully discrete state
equation: $y_{0,h}=\phi_h\in V_h$, where $\phi_h$ verifies 
$$
\phi_h\in V_h, \forall\, h>0,~\text{and}~\lim_{h\rightarrow 0}\phi_h=\phi,~\text{in}~L^2(\Omega),
$$
then, for $n=1,\ldots,N$, with $y_{n-1,h}$ being known, we obtain $y_{n,h}\in V_{gh}(n\Delta t)$ from the
solution of the following linear variational problem:
\begin{equation}\label{state_fuldis}
\int_\Omega\frac{{y}_{n,h}-{y}_{n-1,h}}{\Delta t}\varphi dx+\nu \int_\Omega\nabla{y}_{n,h}\cdot\nabla\varphi dx+\int_\Omega\bm{v}_n\cdot\nabla{y}_{n-1,h}\varphi dx+\int_\Omega a_0{y}_{n-1,h}\varphi dx= \int_\Omega f_{n}\varphi dx,\forall \varphi\in V_{0h}.
\end{equation}
In the following discussion, the subscript $h$ in all variables will be omitted for simplicity.

In a similar way as what we have done in the continuous case, one can show that the first-order differential of $J_h^{\Delta t}
$ at $\bm{v}\in\mathcal{U}_h^{\Delta t}$ is $DJ_h^{\Delta t}(\bm{v})=\{\bm{g}_n\}_{n=1}^N\in (\mathbb{S}_h)^N$ where
\begin{equation}\label{gradient_ful}
\left\{
\begin{aligned}
&\bm{g}_n\in \mathbb{S}_h,\\
&\int_\Omega \bm{g}_n\cdot \bm{z}dx=\int_\Omega(\bm{v}_n-p_n\nabla y_{n-1})\cdot \bm{z}dx,\forall\bm{z}\in \mathbb{S}_h, 
\end{aligned}
\right.
\end{equation}
and the vector-valued function $\{p_n\}^N_{n=1}$ is the solution of the following fully discrete adjoint system:
\begin{equation}\label{ful_adjoint_1}
{p}_{N+1}=\alpha_2({y}_N-y_T);
\end{equation}
for $n=N$, solve
\begin{flalign}\label{ful_adjoint_2}
\qquad \left\{
\begin{aligned}
&p_N\in V_{0h},\\
&\int_\Omega\frac{{p}_N-{p}_{N+1}}{\Delta t}\varphi dx+\nu\int_\Omega \nabla{p}_N\cdot\nabla \varphi dx= \int_\Omega\alpha_1({y}_N-y_d^N)\varphi dx,\forall \varphi\in V_{0h},
\end{aligned}
\right.
\end{flalign}
then, for $n=N-1,\cdots,1,$,  solve
\begin{flalign}\label{ful_adjoint_3}
\qquad \left\{
\begin{aligned}
&p_n\in V_{0h},\\
&\int_\Omega\frac{{p}_n-{p}_{n+1}}{\Delta t}\varphi dx+\nu\int_\Omega\nabla{p}_n\cdot\nabla\varphi dx-\int_\Omega\bm{v}_{n+1}\cdot\nabla{p}_{n+1}\varphi dx\\
&\qquad\qquad\qquad\qquad\qquad\qquad\qquad+a_0\int_\Omega{p}_{n+1}\varphi dx=\int_\Omega  \alpha_1({y}_n-y_d^n)\varphi dx,\forall \varphi\in V_{0h}.
\end{aligned}
\right.
\end{flalign}

It is worth mentioning that the so-called discretize-then-optimize strategy is employed here, which implies that we first discretize (BCP), and to compute the gradient in a discrete setting, the fully discrete adjoint equation (\ref{ful_adjoint_1})--(\ref{ful_adjoint_3})
has been derived from the fully discrete cost functional $J_h^{\Delta t}(\bm{v})$ (\ref{obj_fuldis}) and
the fully discrete state equation (\ref{state_fuldis}). This implies that the fully discrete
state equation (\ref{state_fuldis}) and the  fully discrete adjoint equation (\ref{ful_adjoint_1})--(\ref{ful_adjoint_3}) are
strictly in duality. This fact guarantees that $-DJ_h^{\Delta t}(\bm{v})$ is a descent direction of the fully discrete bilinear optimal control problem (BCP)$_h^{\Delta t}$.

\begin{remark}
	A natural alternative has been advocated in the literature: (i) Derive the adjoint equation to compute the first-order differential of the cost functional in a continuous setting; (ii) Discretize the state and adjoint state
	equations by certain numerical schemes; (iii) Use the
	resulting discrete analogs of $y$ and $p$ to compute a
	discretization of the differential of the cost functional. The main problem with this optimize-then-discretize approach is
	that it may not preserve a strict duality between the
	discrete state equation and the discrete adjoint
	equation. This fact implies in turn that the resulting discretization of the continuous
	gradient may not be a gradient of the discrete optimal control problem. As a consequence, the resulting algorithm is not a descent algorithm and divergence may take place (see \cite{GH1998} for a related discussion).
\end{remark}

\subsection{A Nested CG Method for Solving the Fully Discrete Problem (BCP)$_h^{\Delta t}$}\label{DCG}
In this subsection, we propose a nested CG method for solving the fully discrete problem (BCP)$_h^{\Delta t}$. As discussed in Section \ref{se:cg}, the implementation of CG requires the knowledge of $DJ_h^{\Delta t}(\bm{v})$ and an appropriate stepsize. In the following discussion, we address these two issues by extending the results for the continuous case in Sections \ref{com_gra} and \ref{com_step} to the fully discrete settings; and derive the corresponding CG algorithm.

First, it is clear that one can compute $DJ_h^{\Delta t}(\bm{v})$ via the solution of the $N$ linear variational problems encountered in (\ref{gradient_ful}). For this purpose, we introduce a Lagrange multiplier $\lambda\in \Lambda_{0H}$ associated with the divergence-free constraint, then problem (\ref{gradient_ful}) is equivalent to the following saddle point system
\begin{equation}\label{fulgradient_e}
\left\{
\begin{aligned}
&(\bm{g}_n,\lambda)\in (V_h\times V_h)\times \Lambda_{0H},\\
&\int_\Omega \bm{g}_n\cdot \bm{z}dx=\int_\Omega (\bm{v}_n-p_n\nabla y_{n-1})\cdot \bm{z}dx+\int_\Omega\lambda\nabla\cdot \bm{z}dx,\forall \bm{z}\in V_h\times V_h,\\
&\int_\Omega\nabla\cdot\bm{g}_nqdx=0,\forall q\in \Lambda_{0H}.
\end{aligned}
\right.
\end{equation}

As discussed in Section \ref{com_gra}, problem (\ref{fulgradient_e}) can be solved 
 by the following preconditioned CG algorithm, which is actually a discrete analogue of (\textbf{G1})--(\textbf{G5}).
\begin{enumerate}
	\item [\textbf{DG1}] Choose $\lambda^0\in \Lambda_{0H}$.
	\item [\textbf{DG2}] Solve 
	\begin{equation*}
	\left\{
	\begin{aligned}
	&\bm{g}_n^0\in V_h\times V_h,\\
	&\int_\Omega\bm{g}_n^0\cdot \bm{z}dx=\int_\Omega(\bm{v}_n-p_n\nabla y_{n-1})\cdot \bm{z}dx+\int_\Omega\lambda^0\nabla\cdot \bm{z}dx,\forall \bm{z}\in V_h\times V_h,
	\end{aligned}
	\right.
	\end{equation*}
	and
	\begin{equation*}
	\left\{
	\begin{aligned}
	&r^0\in \Lambda_{0H},\\
	&\int_\Omega \nabla r^0\cdot \nabla qdx=\int_\Omega\nabla\cdot \bm{g}_n^0qdx,\forall q\in \Lambda_{0H}.
	\end{aligned}
	\right.
	\end{equation*}
	\smallskip
	If $\frac{\int_\Omega|\nabla r^0|^2dx}{\max\{1,\int_\Omega|\nabla \lambda^0|^2dx\}}\leq tol_1$, take $\lambda=\lambda^0$ and $\bm{g}_n=\bm{g}_n^0$; otherwise set $w^0=r^0$.
	For $k\geq 0$, $\lambda^k,\bm{g}_n^k, r^k$ and $w^k$ being known with the last two different from 0, we define
	$\lambda^{k+1},\bm{g}_n^{k+1}, r^{k+1}$ and if necessary $w^{k+1}$, as follows:
	\smallskip
	\item[\textbf{DG3}] Solve 
	\begin{equation*}
	\left\{
	\begin{aligned}
	&\bar{\bm{g}}_n^k\in V_h\times V_h,\\
	&\int_\Omega \bar{\bm{g}}_n^k\cdot \bm{z}dx=\int_\Omega w^k\nabla\cdot \bm{z}dx,\forall \bm{z}\in V_h\times V_h,
	\end{aligned}
	\right.
	\end{equation*}
	and
	\begin{equation*}
	\left\{
	\begin{aligned}
	&\bar{r}^k\in \Lambda_{0H},\\
	&\int_\Omega \nabla \bar{r}^k\cdot \nabla qdx=\int_\Omega\nabla\cdot \bar{\bm{g}}_n^kqdx,\forall  q\in\Lambda_{0H},
	\end{aligned}
	\right.
	\end{equation*}
	and compute
	$$
	\eta_k=\frac{\int_\Omega|\nabla r^k|^2dx}{\int_\Omega\nabla\bar{r}^k\cdot\nabla w^kdx}.
	$$
	\item[\textbf{DG4}] Update $\lambda^k,\bm{g}_n^k$ and $r^k$ via
	$$\lambda^{k+1}=\lambda^k-\eta_kw^k,\bm{g}_n^{k+1}=\bm{g}_n^k-\eta_k\bar{\bm{g}}_n^k,~\text{and}~r^{k+1}=r^k-\eta_k \bar{r}^k.$$
	\smallskip
	If $\frac{\int_\Omega|\nabla r^{k+1}|^2dx}{\max\{1,\int_\Omega|\nabla r^0|^2dx\}}\leq tol_1$, take $\lambda=\lambda^{k+1}$ and $\bm{g}_n=\bm{g}_n^{k+1}$; otherwise, 
	\item[\textbf{DG5}] Compute
	$$\gamma_k=\frac{\int_\Omega|\nabla r^{k+1}|^2dx}{\int_\Omega|\nabla r^k|^2dx},$$
	and update $w^k$ via
	$$w^{k+1}=r^{k+1}+\gamma_k w^{k}.$$
	Do $k+1\rightarrow k$ and return to \textbf{DG3}.
\end{enumerate}

To find an appropriate stepsize in the CG iteration for the solution of (BCP)$_h^{\Delta t}$, we note that, for any $\{\bm{w}_n\}_{n=1}^N\in (\mathbb{S}_h)^N$, the fully discrete analogue of $Q_k(\rho)$ in (\ref{q_rho}) reads as
$$
Q_h^{\Delta t}(\rho)=\frac{1}{2}\Delta t \sum^N_{n=1}\int_\Omega |\bm{u}_n-\rho\bm{w}_n|^2dx+\frac{\alpha_1}{2}\Delta t \sum^N_{n=1}\int_\Omega |y_{n}-\rho z_{n}-y_d^n|^2dx+\frac{\alpha_2}{2}\int_\Omega|y_{N}-\rho z_{N}-y_T|^2dx,
$$
where the vector-valued function $\{z_n\}^N_{n=1}$ is obtained as follows: $z_0=0$; then for $n=1,\ldots,N$, with $z_{n-1}$ being known, $z_n$ is obtained from the
solution of the linear variational problem
$$
\left\{
\begin{aligned}
&z_n\in V_{0h},\\
&\int_\Omega\frac{{z}_n-{z}_{n-1}}{\Delta t}\varphi dx+\nu\int_\Omega \nabla{z}_n\cdot\nabla\varphi dx+\int_\Omega\bm{w}_n\cdot\nabla y_n\varphi dx\\
&\qquad\qquad\qquad\qquad\qquad+\int_\Omega\bm{u}_n\cdot\nabla{z}_{n-1}\varphi dx+a_0\int_\Omega{z}_{n-1}\varphi dx= 0,\forall\varphi\in V_{0h}. \\
\end{aligned}
\right.
$$
As discussed in Section \ref{com_step} for the continuous case, we take the unique solution of ${Q_h^{\Delta t}}'(\rho)=0$ as the stepsize in each CG iteration, that is
\begin{equation}\label{step_ful}
\hat{\rho}_h^{\Delta t} =\frac{\Delta t\sum_{n=1}^{N}\int_\Omega\bm{g}_n\cdot \bm{w}_n dx}{\Delta t\sum_{n=1}^{N}\int_\Omega|\bm{w}_n|^2dxdt+ \alpha_1\Delta t\sum_{n=1}^{N}\int_\Omega|z_n|^2dxdt+\alpha_2\int_\Omega|z_N|^2dx}.
\end{equation}

Finally, with above preparations, we propose the following nested CG algorithm for the solution of the fully discrete control problem
(BCP)$_h^{\Delta t}$.
\begin{enumerate}
	\item[\textbf{DI.}] Given $\bm{u}^0:=\{\bm{u}_n^0\}_{n=1}^N\in (\mathbb{S}_h)^N$.
	\item[\textbf{DII.}] Compute $\{y_n^0\}_{n=0}^N$ and $\{p^0_n\}_{n=1}^{N+1}$ by solving the fully discrete state equation (\ref{state_fuldis}) and the fully discrete adjoint equation (\ref{ful_adjoint_1})--(\ref{ful_adjoint_3}) corresponding to $\bm{u}^0$.
	Then, for $n=1,\cdots, N$ solve 
	\begin{equation*}
	\left\{
	\begin{aligned}
	&\bm{g}_n^0\in \mathbb{S}_h,\\
	&\int_\Omega \bm{g}_n^0\cdot \bm{z}dx=\int_\Omega(\bm{u}_n^0-p_n^0\nabla y_{n-1}^0)\cdot \bm{z}dx,\forall \bm{z}\in \mathbb{S}_h,
	\end{aligned}
	\right.
	\end{equation*}
	by the preconditioned CG algorithm (\textbf{DG1})--(\textbf{DG5}), and set $\bm{w}^0_n=\bm{g}_n^0.$
	
	\medskip
	\noindent For $k\geq 0$, $\bm{u}^k, \bm{g}^k$ and $\bm{w}^k$ being known, the last two different from $\bm{0}$, one computes $\bm{u}^{k+1}, \bm{g}^{k+1}$ and $\bm{w}^{k+1}$ as follows:
	\medskip
	\item[\textbf{DIII.}] Compute the stepsize $\hat{\rho}_k$ by (\ref{step_ful}).
	\item[\textbf{DIV.}] Update $\bm{u}^{k+1}$ by $$\bm{u}^{k+1}=\bm{u}^k-\hat{\rho}_k\bm{w}^k.$$
	Compute $\{y_n^{k+1}\}_{n=0}^N$ and $\{p_n^{k+1}\}_{n=1}^{N+1}$ by solving the fully discrete state equation (\ref{state_fuldis}) and the fully discrete adjoint equation (\ref{ful_adjoint_1})--(\ref{ful_adjoint_3}) corresponding to $\bm{u}^{k+1}$.
	Then, for $n=1,\cdots,N$, solve 
	\begin{equation}\label{dis_gradient}
	\left\{
	\begin{aligned}
	&\bm{g}_n^{k+1}\in \mathbb{S}_h,\\
	&\int_\Omega \bm{g}_n^{k+1}\cdot \bm{z}dx=\int_\Omega(\bm{u}_n^{k+1}-p_n^{k+1}\nabla y_{n-1}^{k+1})\cdot \bm{z}dx,\forall \bm{z}\in \mathbb{S}_h,
	\end{aligned}
	\right.
	\end{equation}
	by the preconditioned CG algorithm (\textbf{DG1})--(\textbf{DG5}).
	
	\medskip
	\noindent If $\frac{\Delta t\sum_{n=1}^N\int_\Omega|\bm{g}_n^{k+1}|^2dx}{\Delta t\sum_{n=1}^N\int_\Omega|\bm{g}_n^{0}|^2dx}\leq tol$, take $\bm{u} = \bm{u}^{k+1}$; else
	\medskip
	\item[\textbf{DV.}] Compute $$\beta_k=\frac{\Delta t\sum_{n=1}^N\int_\Omega|\bm{g}_n^{k+1}|^2dx}{\Delta t\sum_{n=1}^N\int_\Omega|\bm{g}_n^{k}|^2dx},~\text{and}~\bm{w}^{k+1} = \bm{g}^{k+1} + \beta_k\bm{w}^k.$$
	Do $k+1\rightarrow k$ and return to \textbf{DIII}.
\end{enumerate}

 Despite its apparent complexity, the CG algorithm (\textbf{DI})-(\textbf{DV}) is easy to implement. Actually, one of the main computational difficulties in the implementation of the above algorithm seems to be the solution of $N$  linear systems (\ref{dis_gradient}), which is time-consuming. However, it is worth noting that the linear systems (\ref{dis_gradient}) are separable with respect to different $n$ and they can be solved in parallel. As a consequent, one can compute the gradient $\{\bm{g}^{k}_n\}_{n=1}^N$ simultaneously and the computation time can be reduced significantly. 
	
Moreover, it is clear that the computation of $\{\bm{g}^{k}_n\}_{n=1}^N$ requires the storage of the solutions of (\ref{state_fuldis}) and (\ref{ful_adjoint_1})-(\ref{ful_adjoint_3}) at all points in space and time. For large scale problems, especially in three space dimensions, it will be very
memory demanding and maybe even impossible to store the full sets $\{y_n^k\}_{n=0}^N$ and $\{p_n^k\}_{n=1}^{N+1}$ simultaneously. To tackle this issue, one can employ the strategy described in e.g., \cite[Section 1.12]{glowinski2008exact} that can drastically reduce the
storage requirements at the expense of a small CPU increase.

\section{Numerical Experiments}\label{se:numerical}
	In this section, we report some preliminary numerical results validating the efficiency of the proposed CG algorithm (\textbf{DI})--(\textbf{DV}) for (BCP). All codes were written in MATLAB R2016b and numerical experiments were conducted on a Surface Pro 5 laptop with 64-bit Windows 10.0 operation system, Intel(R) Core(TM) i7-7660U CPU (2.50 GHz), and 16 GB RAM.
	
	\medskip
\noindent\textbf{Example 1.} We consider the bilinear optimal control problem (BCP) on the domain $Q=\Omega\times(0,T)$ with $\Omega=(0,1)^2$ and $T=1$. In particular, we take the control $\bm{v}(x,t)$ in a finite-dimensional space, i.e. $\bm{v}\in L^2(0,T;\mathbb{R}^2)$.
In addition, we set $\alpha_2=0$ in (\ref{objective_functional}) and consider the following tracking-type bilinear optimal control problem:
\begin{equation}\label{model_ex1}
\min_{\bm{v}\in L^2(0,T;\mathbb{R}^2)}J(\bm{v})=\frac{1}{2}\int_0^T|\bm{v}(t)|^2dt+\frac{\alpha_1}{2}\iint_Q|y-y_d|^2dxdt,
\end{equation}
where $|\bm{v}(t)|=\sqrt{\bm{v}_1(t)^2+\bm{v}_2(t)^2}$ is the canonical 2-norm, and $y$ is obtained from $\bm{v}$ via the solution of the state equation (\ref{state_equation}).

Since the control $\bm{v}$ is considered in a finite-dimensional space, the divergence-free constraint $\nabla\cdot\bm{v}=0$ is verified automatically. As a consequence, the first-order differential $DJ(\bm{v})$ can be easily computed. Indeed, it is easy to show that
\begin{equation}\label{oc_finite}
DJ(\bm{v})=\left\{\bm{v}_i(t)+\int_\Omega y(t)\frac{\partial p(t)}{\partial x_i}dx\right \}_{i=1}^2,~\text{a.e.~on}~(0,T),\forall \bm{v}\in L^2(0,T;\mathbb{R}^2),
\end{equation}
where $p(t)$ is the solution of the adjoint equation (\ref{adjoint_equation}). The inner preconditioned CG algorithm (\textbf{DG1})-(\textbf{DG5}) for the computation of the gradient $\{\bm{g}_n\}_{n=1}^N$ is thus avoided.

In order to examine the efficiency of the proposed CG algorithm (\textbf{DI})--(\textbf{DV}), we construct an example with a known exact solution. To this end,
we set $\nu=1$ and $a_0=1$ in (\ref{state_equation}), and
$$
y=e^t(-3\sin(2\pi x_1)\sin(\pi x_2)+1.5\sin(\pi x_1)\sin(2\pi x_2)),\quad p=(T-t)\sin \pi x_1 \sin \pi x_2.
$$
Substituting these two functions into the optimality condition $DJ(\bm{u}(t))=0$, we have
$$
\bm{u}=(\bm{u}_1,\bm{u}_2)^\top=(2e^t(T-t),-e^t(T-t))^\top.
$$
We further set
\begin{eqnarray*}
&&f=\frac{\partial y}{\partial t}-\nabla^2y+{\bm{u}}\cdot \nabla y+y,
 \quad\phi=-3\sin(2\pi x_1)\sin(\pi x_2)+1.5\sin(\pi x_1)\sin(2\pi x_2),\\
&&y_d=y-\frac{1}{\alpha_1}\left(-\frac{\partial p}{\partial t}
-\nabla^2p-\bm{u}\cdot\nabla p +p\right),\quad g=0.
\end{eqnarray*}
Then, it is easy to verify that $\bm{u}$ is a solution point of the problem (\ref{model_ex1}). 
We display the solution $\bm{u}$ and the target function $y_d$ at  different instants of time in Figure \ref{exactU_ex1} and Figure \ref{target_ex1}, respectively.
	\begin{figure}[htpb]
	\centering{
		\includegraphics[width=0.43\textwidth]{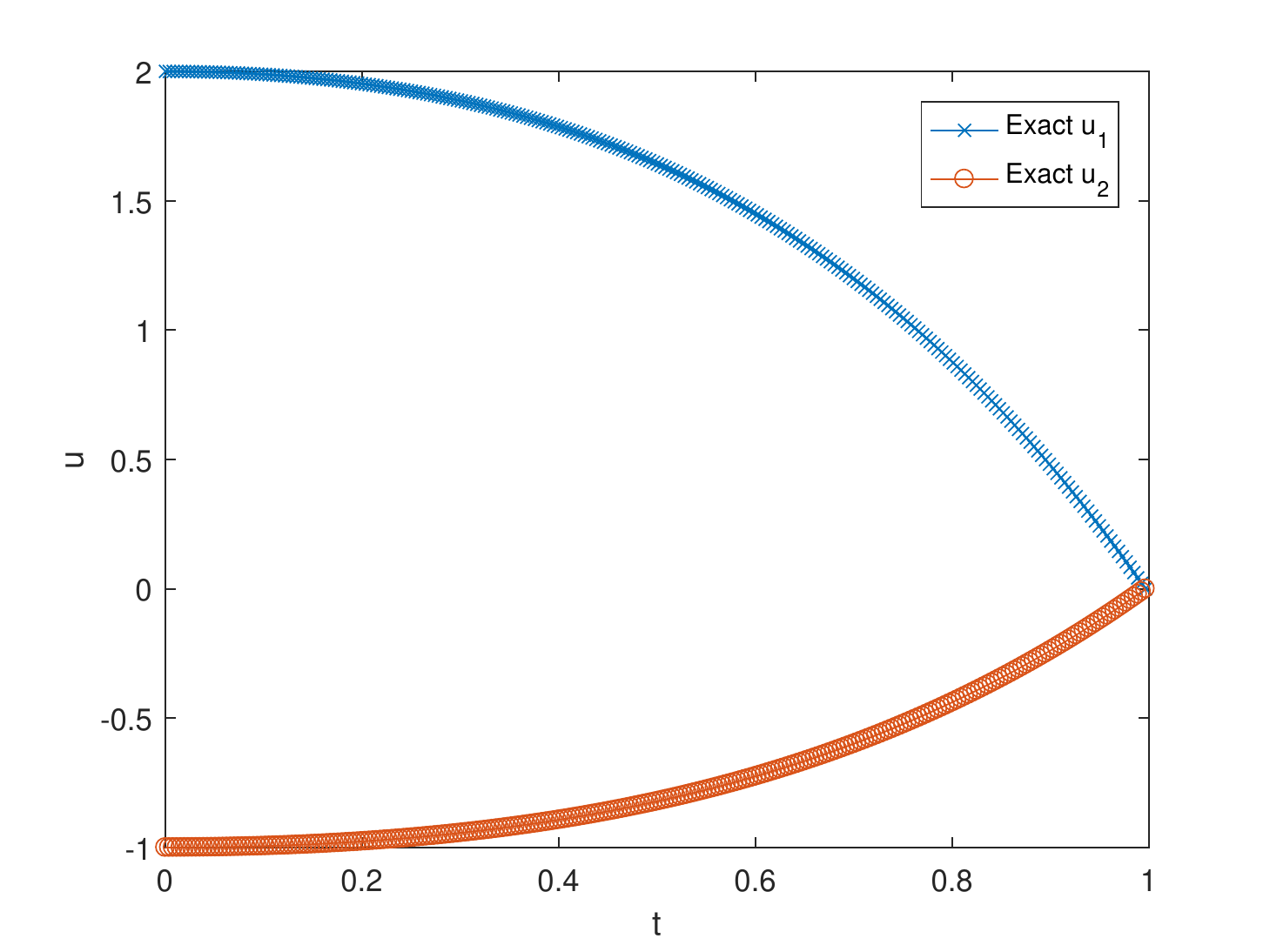}
	}
	\caption{The exact optimal control $\bm{u}$ for Example 1.}
	\label{exactU_ex1}
\end{figure}

	\begin{figure}[htpb]
	\centering{
		\includegraphics[width=0.3\textwidth]{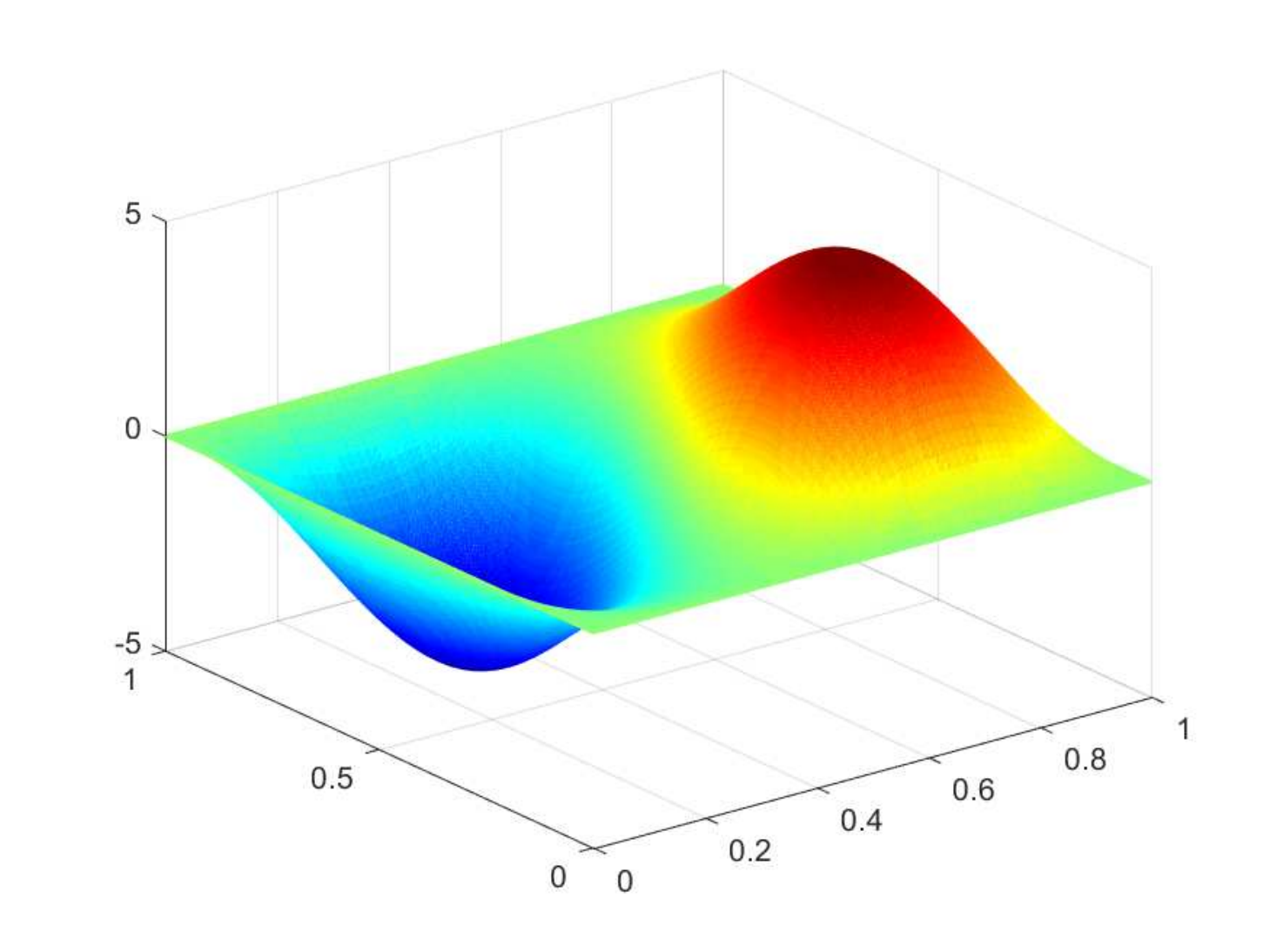}
	\includegraphics[width=0.3\textwidth]{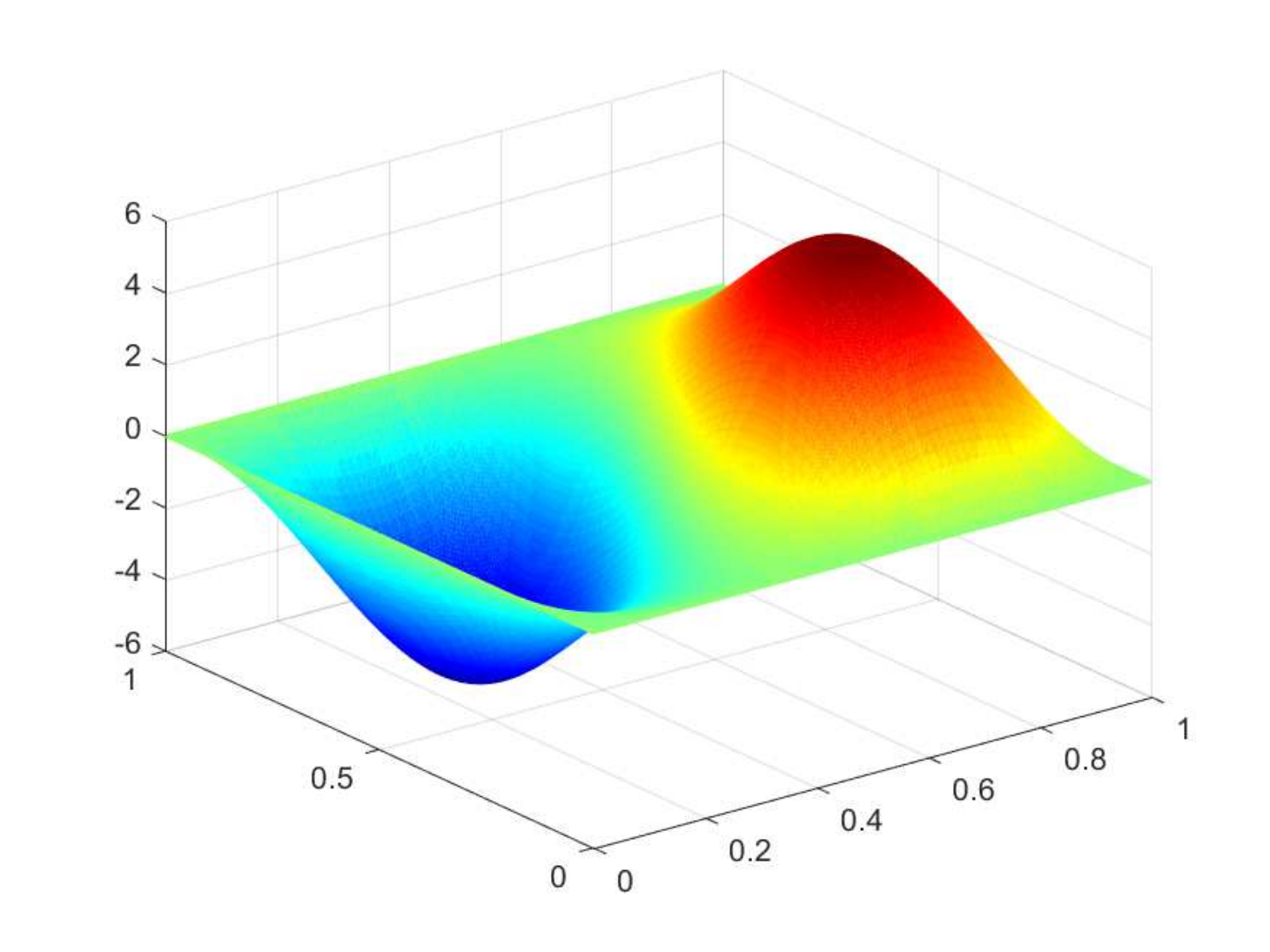}
			\includegraphics[width=0.3\textwidth]{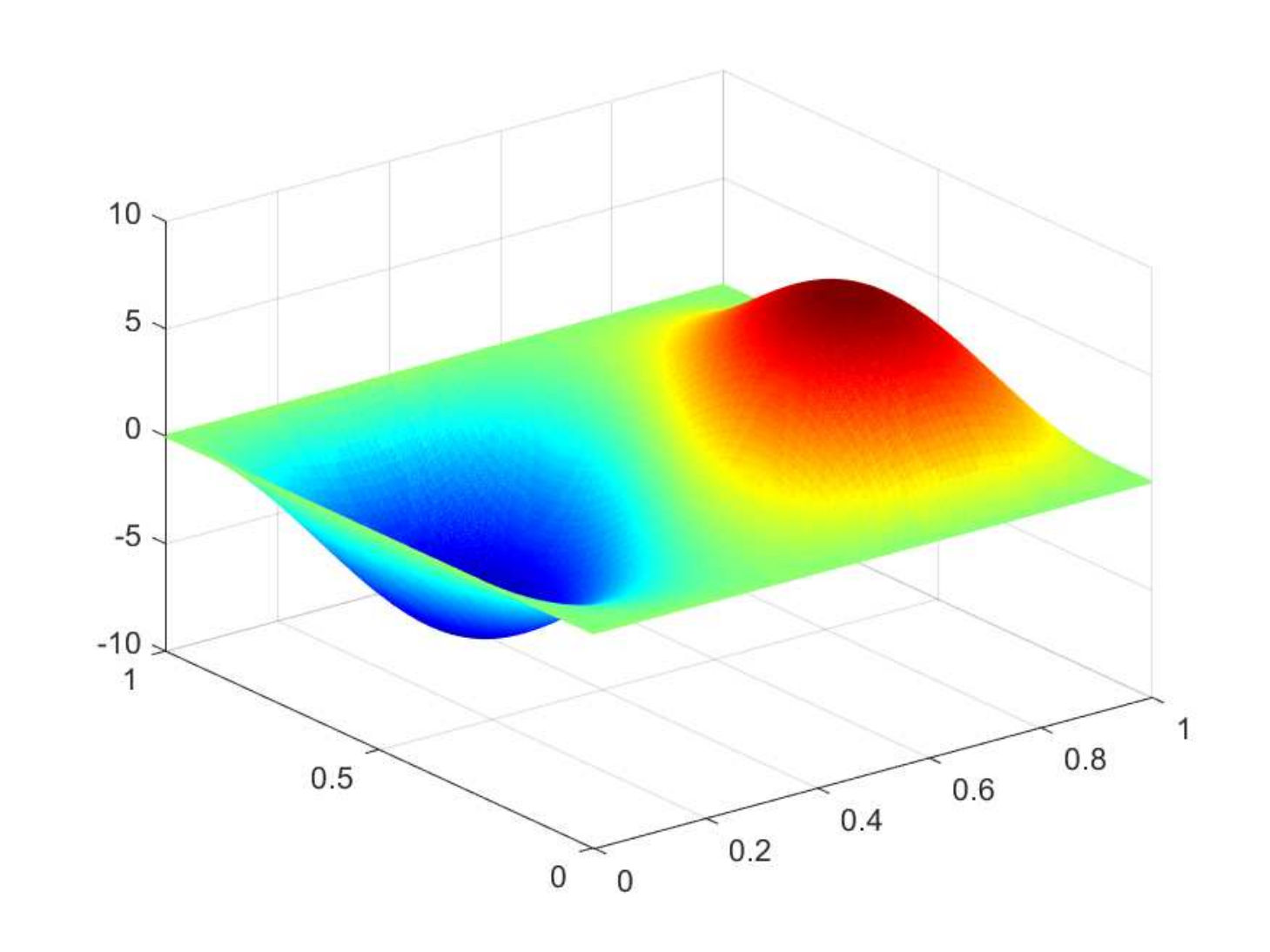}
		}
	\caption{The target function $y_d$ at $t=0.25, 0.5$ and $0.75$ (from left to right) for Example 1.}
	\label{target_ex1}
\end{figure}

The stopping criterion of the CG algorithm (\textbf{DI})--(\textbf{DV}) is set as
$$
\frac{\Delta t\sum_{n=1}^N|\bm{g}^{k+1}_n|^2}{\Delta t\sum_{n=1}^N|\bm{g}^{0}_n|^2}\leq 10^{-5}.
$$
The initial value is chosen as $\bm{u}^0=(0,0)^\top$; and we denote by $\bm{u}^{\Delta t}$ and $y_h^{\Delta t}$ the computed control and state, respectively. 

First, we take $h=\frac{1}{2^i}, i=5,6,7,8$, $\Delta t=\frac{h}{2}$ and $\alpha_1=10^6$, and implement the proposed CG algorithm (\textbf{DI})--(\textbf{DV}) for solving the problem (\ref{model_ex1}). The numerical results reported in Table \ref{tab:mesh_EX1} show that the CG algorithm converges fairly fast and is robust with respect to different mesh sizes. We also observe that the target function $y_d$ has been reached within a good accuracy. Similar comments hold for the approximation of the optimal control $\bm{u}$ and of the state $y$ of problem (\ref{model_ex1}). By taking $h=\frac{1}{2^7}$ and $\Delta t=\frac{1}{2^8}$, the computed state  $y_h^{\Delta t}$ and $y_h^{\Delta t}-y_d$ at $t=0.25,0.5$ and $0.75$ are reported in Figures \ref{stateEx1_1}, \ref{stateEx1_2} and \ref{stateEx1_3}, respectively; and the computed control $\bm{u}^{\Delta t}$ and error $\bm{u}^{\Delta t}-\bm{u}$ are visualized in Figure \ref{controlEx1}.
	\begin{table}[htpb] 
	{\small\centering
	\caption{Results of the CG algorithm (\textbf{DI})--(\textbf{DV}) with different $h$ and $\Delta t$ for Example 1.}	\label{tab:mesh_EX1}
	\begin{tabular}{|c|c|c|c|c|c|}
		\hline Mesh sizes &$Iter$& $\|\bm{u}^{\Delta t}-\bm{u}\|_{L^2(0,T;\mathbb{R}^2)}$&$\|y_h^{\Delta t}-y\|_{L^2(Q)}$& ${\|y_h^{\Delta t}-y_d\|_{L^2(Q)}}/{\|y_d\|_{{L^2(Q)}}}$  \\
		\hline $h=1/2^5,\Delta t=1/2^6$   & 117 &2.8820$\times 10^{-2}$ &1.1569$\times 10^{-2}$&3.8433$\times 10^{-3}$  
		\\
		\hline $h=1/2^6,\Delta t=1/2^7$   &48&1.3912$\times 10^{-2}$& 2.5739$\times 10^{-3}$&8.5623$\times 10^{-4}$  
		\\
		\hline $h=1/2^7,\Delta t=1/2^8$   &48&6.9095$\times 10^{-3}$&  4.8574$\times 10^{-4}$ &1.6516$\times 10^{-4}$  
		\\
		\hline $h=1/2^8,\Delta t=1/2^9$& 31 &3.4845$\times 10^{-3}$ &6.6231$\times 10^{-5}$  &2.2196$\times 10^{-5}$ 
		\\
		\hline
	\end{tabular}}
\end{table}

\begin{figure}[htpb]
	\centering{
		\includegraphics[width=0.3\textwidth]{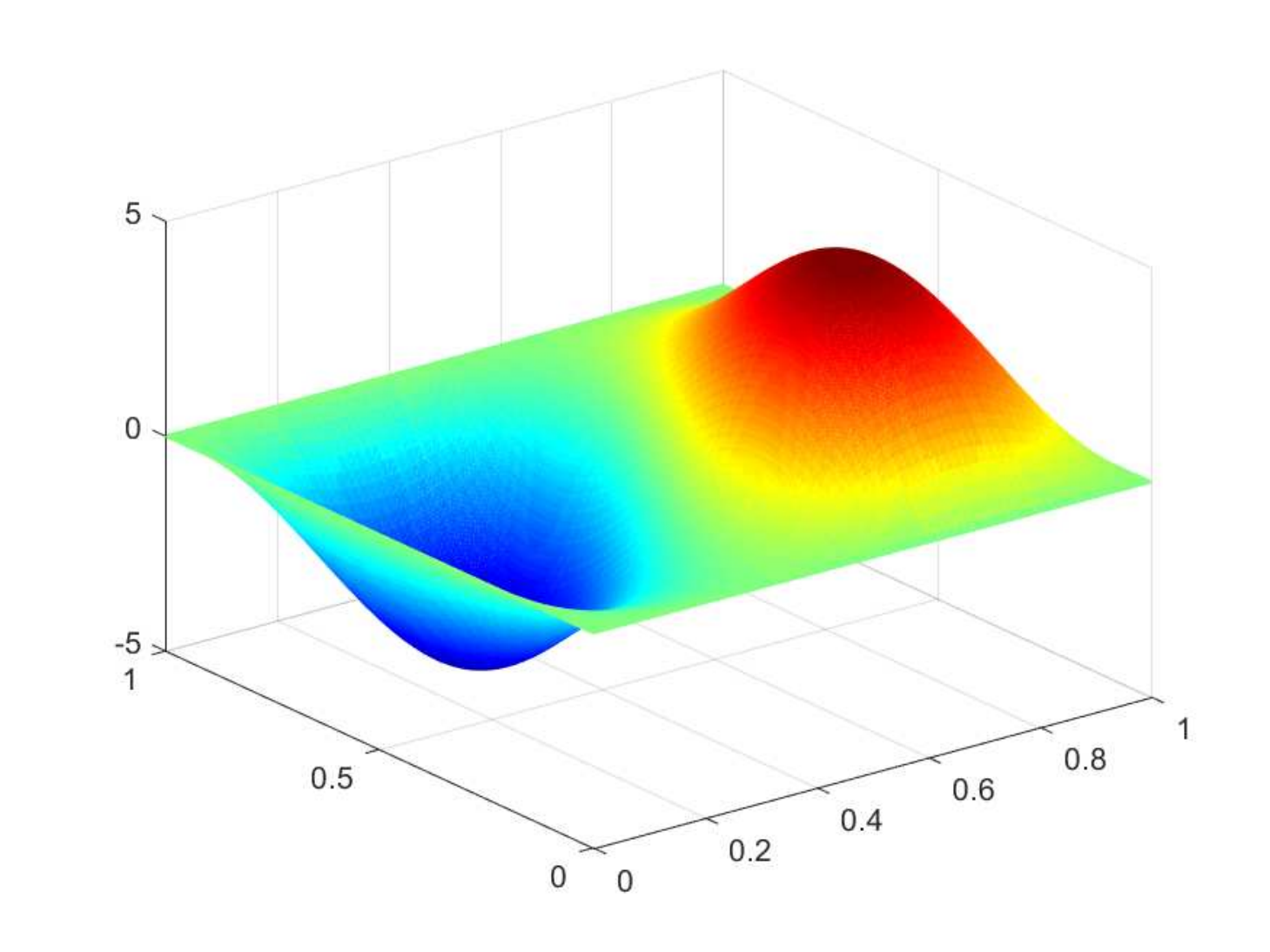}
	   \includegraphics[width=0.3\textwidth]{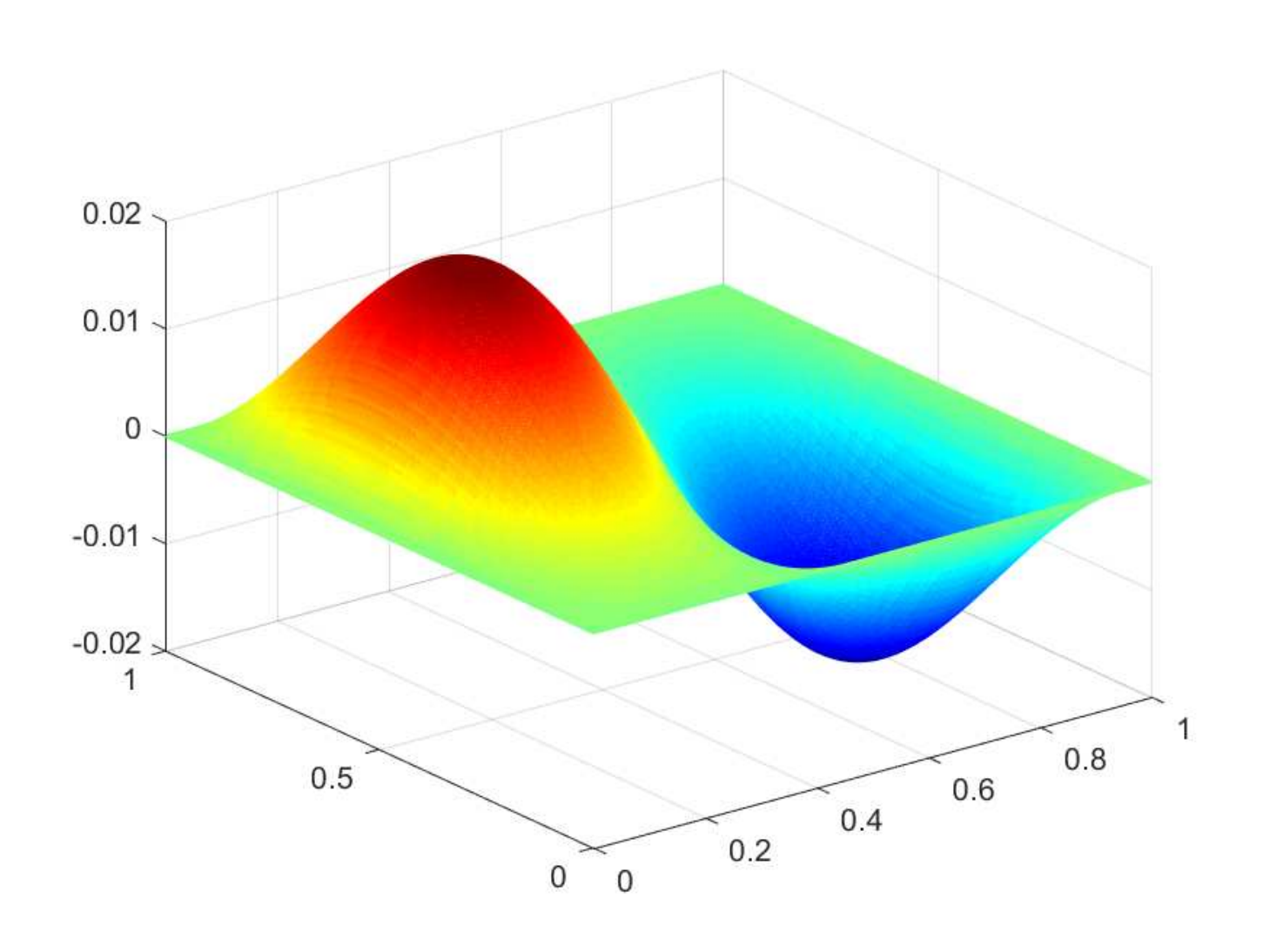}
		\includegraphics[width=0.3\textwidth]{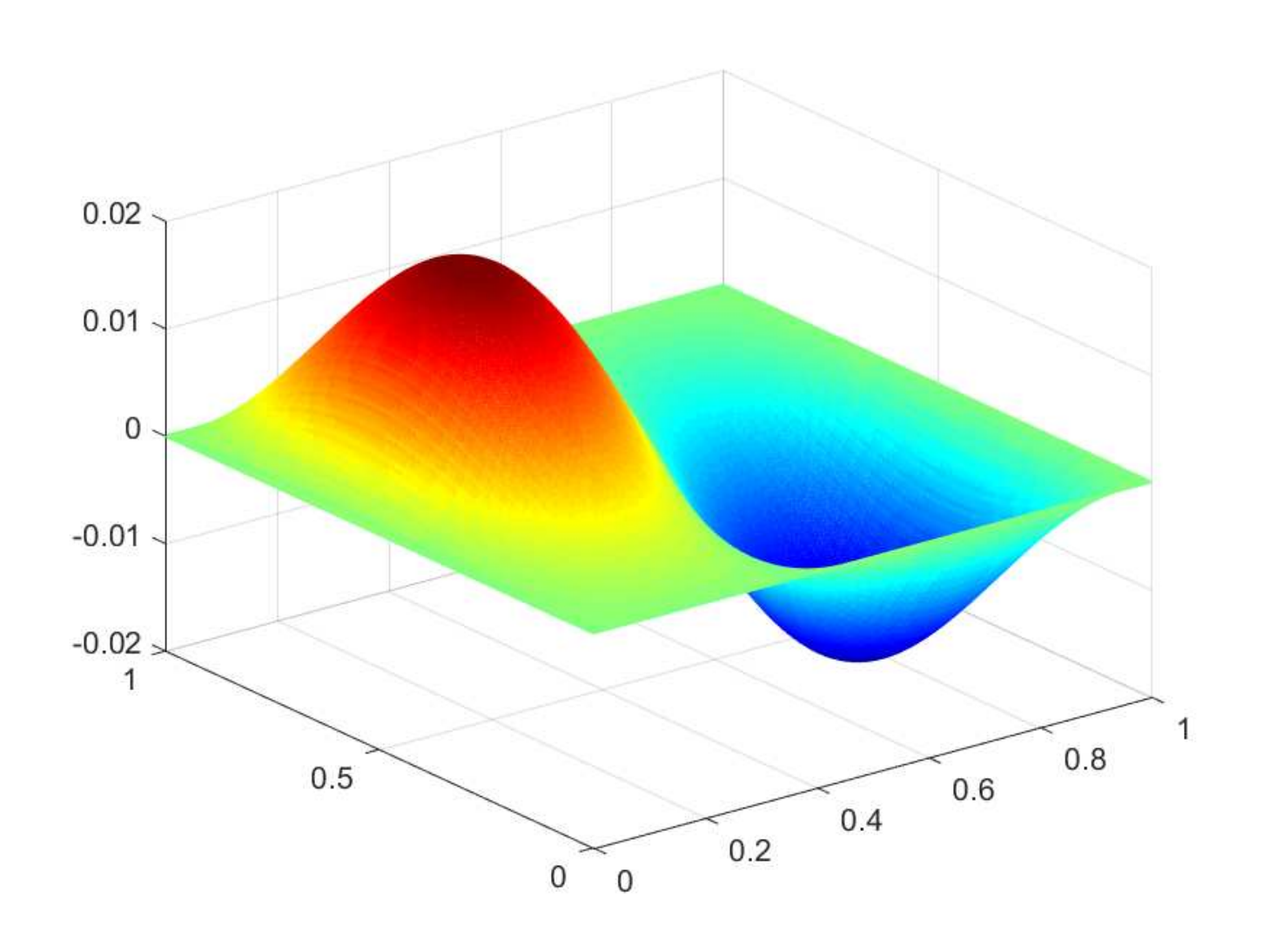}}
	\caption{Computed state $y^{\Delta t}_h$, error $y^{\Delta t}_h-y$ and $y^{\Delta t}_h-y_d$ (from left to right) at $t=0.25$ for Example 1.}
	\label{stateEx1_1}
\end{figure}
\begin{figure}[htpb]
	\centering{
		\includegraphics[width=0.3\textwidth]{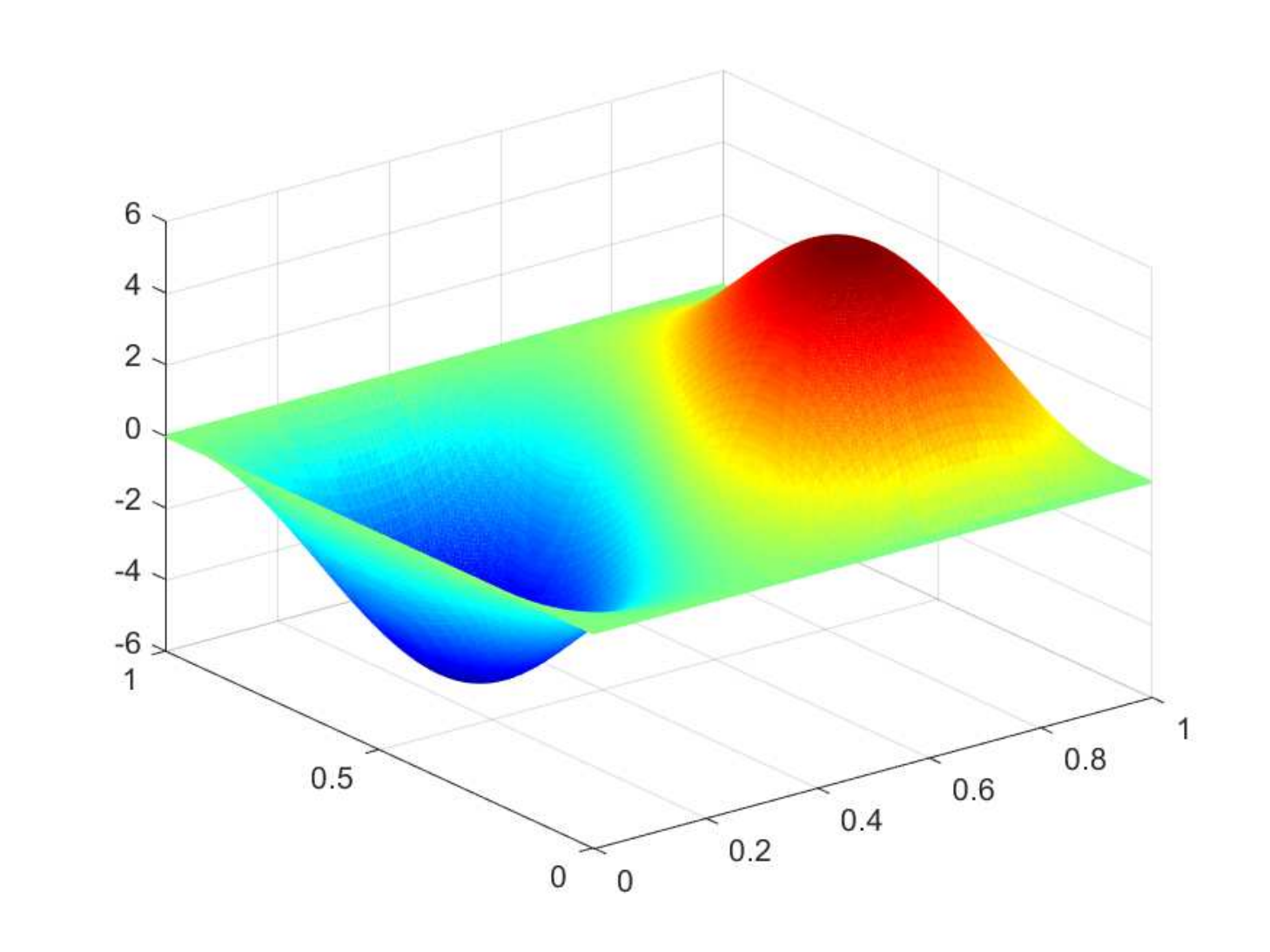}
		\includegraphics[width=0.3\textwidth]{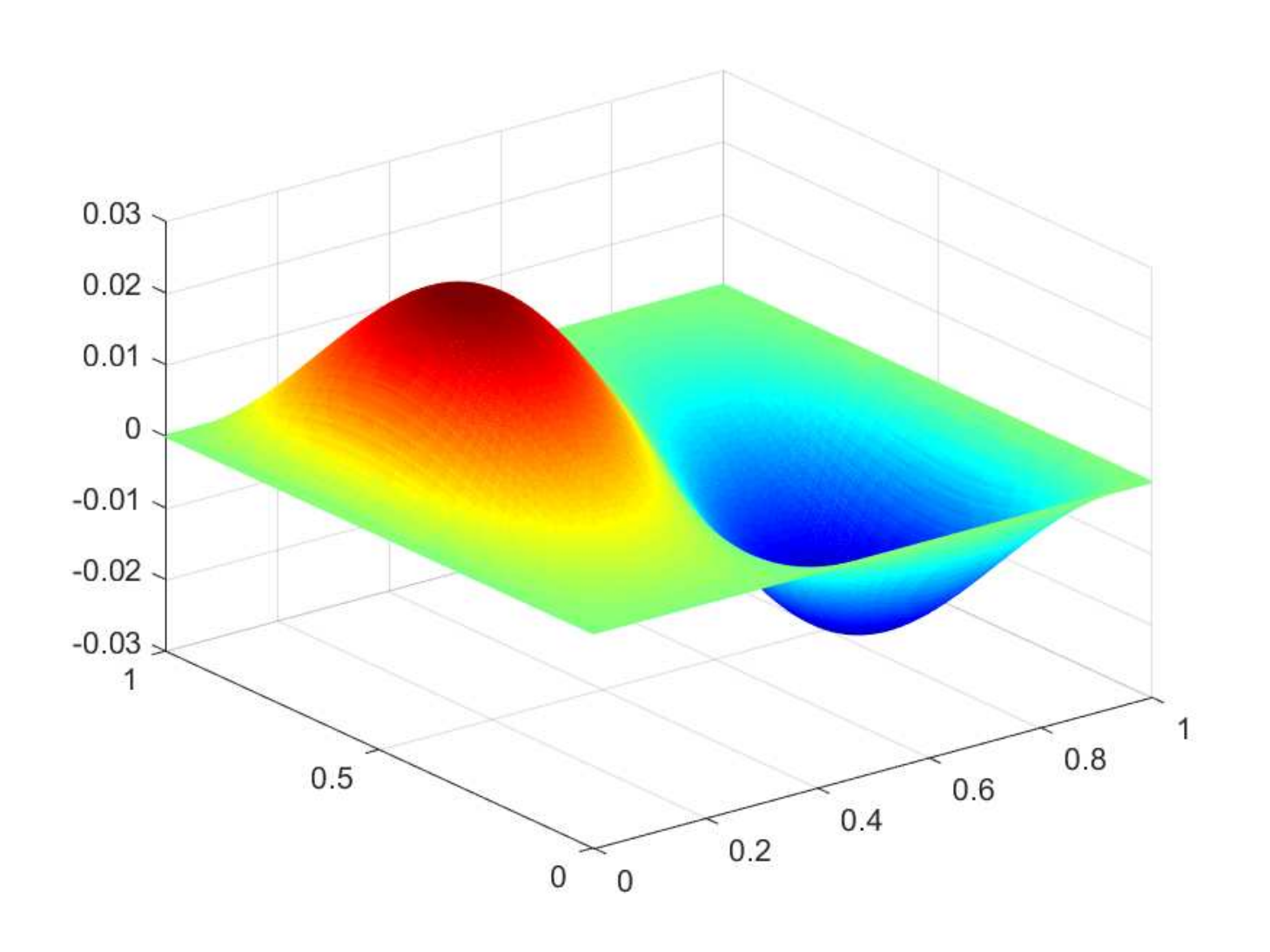}
		\includegraphics[width=0.3\textwidth]{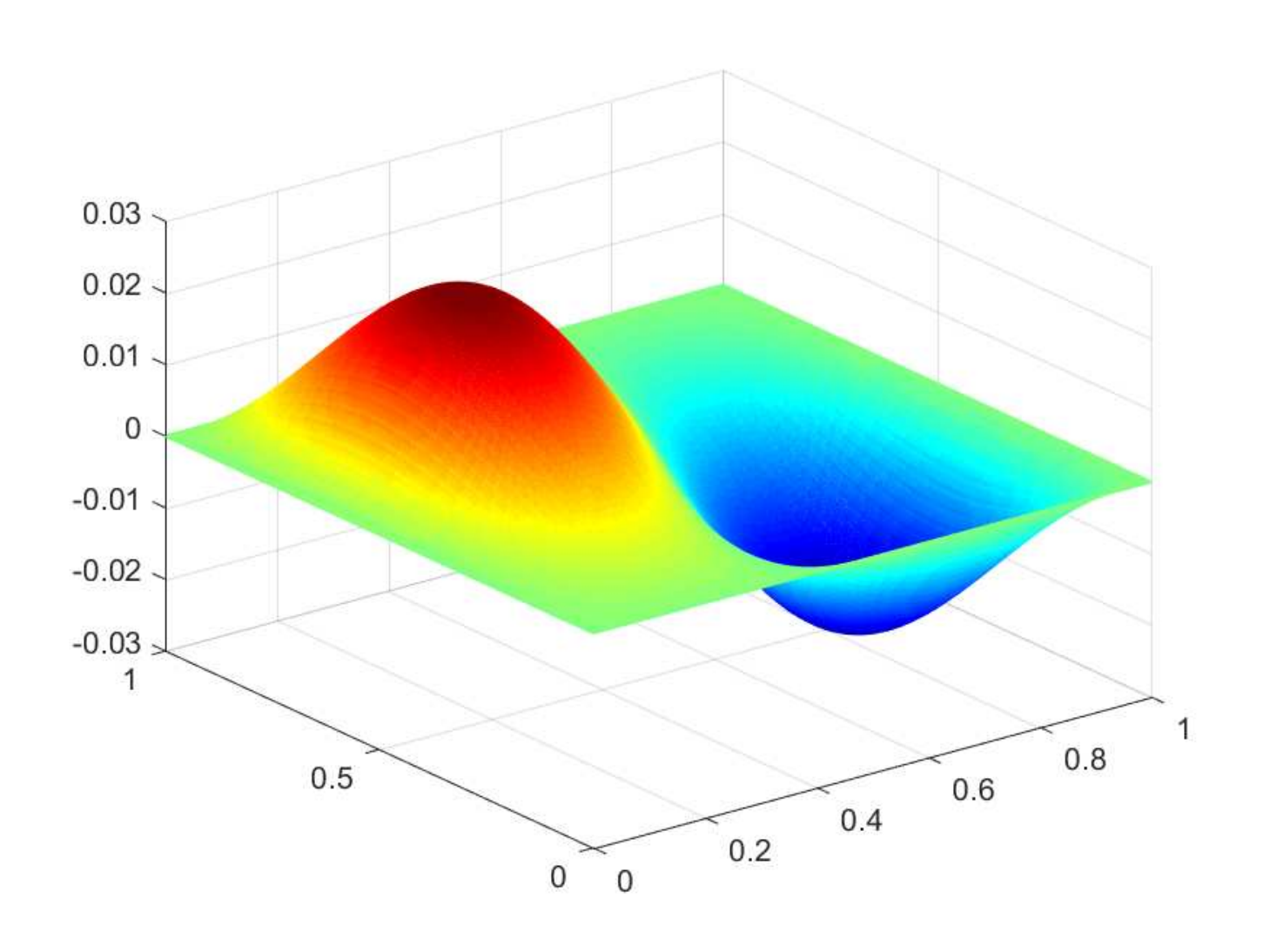}}
	\caption{Computed state $y^{\Delta t}_h$, error $y^{\Delta t}_h-y$ and $y^{\Delta t}_h-y_d$ (from left to right) at $t=0.5$ for Example 1.}
	\label{stateEx1_2}
\end{figure}
\begin{figure}[htpb]
	\centering{
		\includegraphics[width=0.3\textwidth]{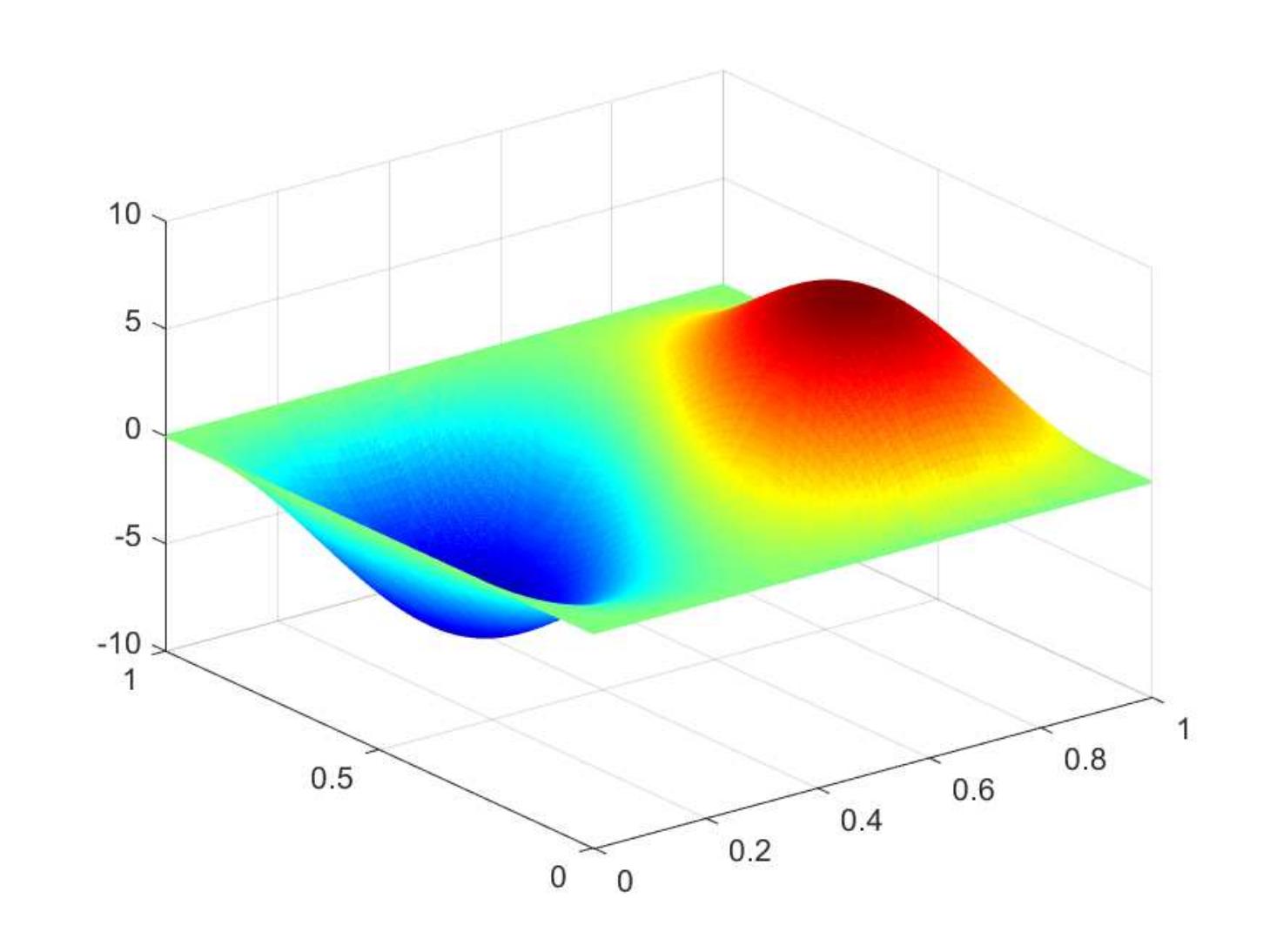}
		\includegraphics[width=0.3\textwidth]{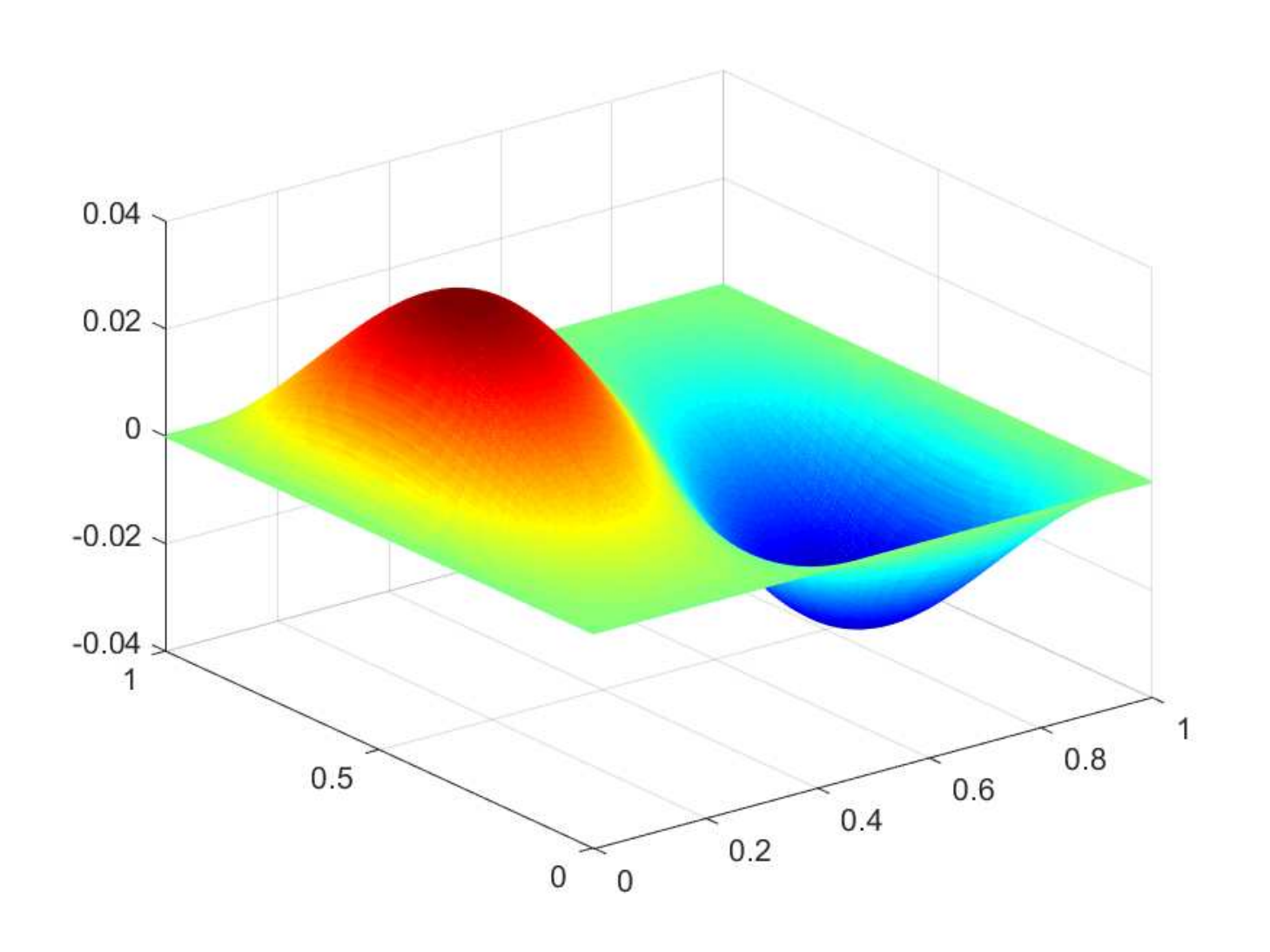}
		\includegraphics[width=0.3\textwidth]{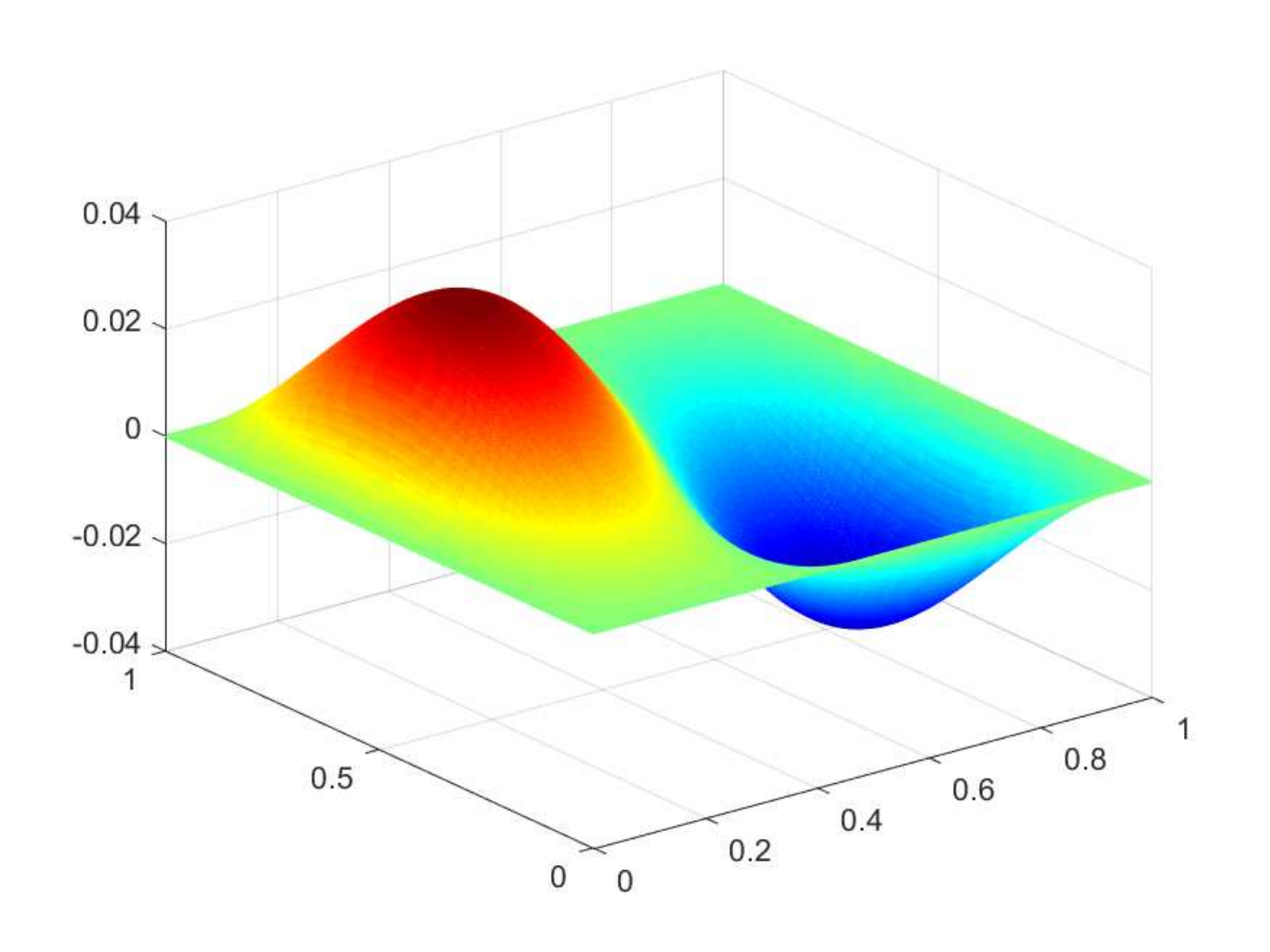}}
	\caption{Computed state $y^{\Delta t}_h$, error $y^{\Delta t}_h-y$ and $y^{\Delta t}_h-y_d$ (from left to right) at $t=0.75$ for Example 1.}
	\label{stateEx1_3}
\end{figure}
	\begin{figure}[htpb]
	\centering{
		\includegraphics[width=0.45\textwidth]{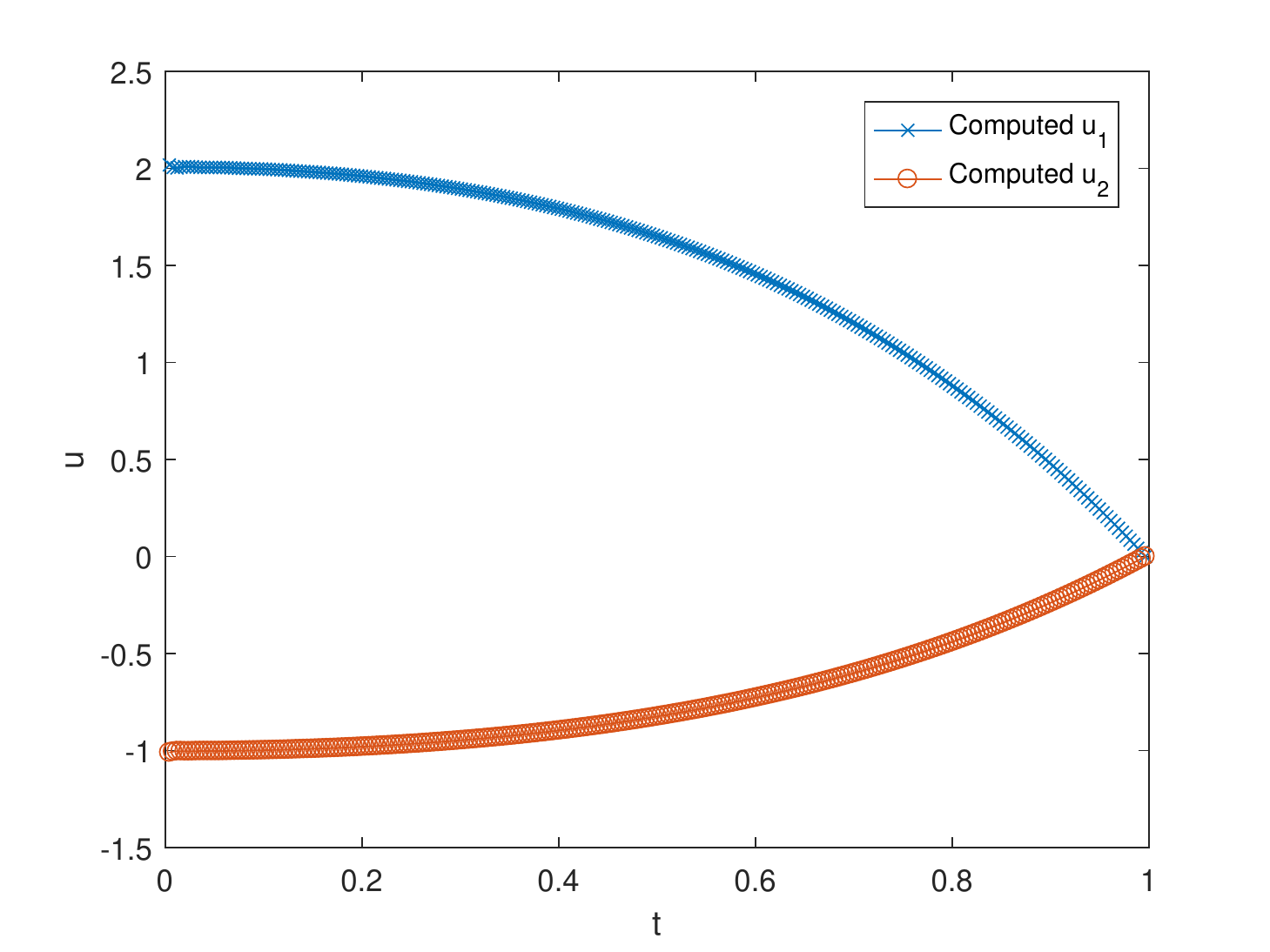}
	  \includegraphics[width=0.45\textwidth]{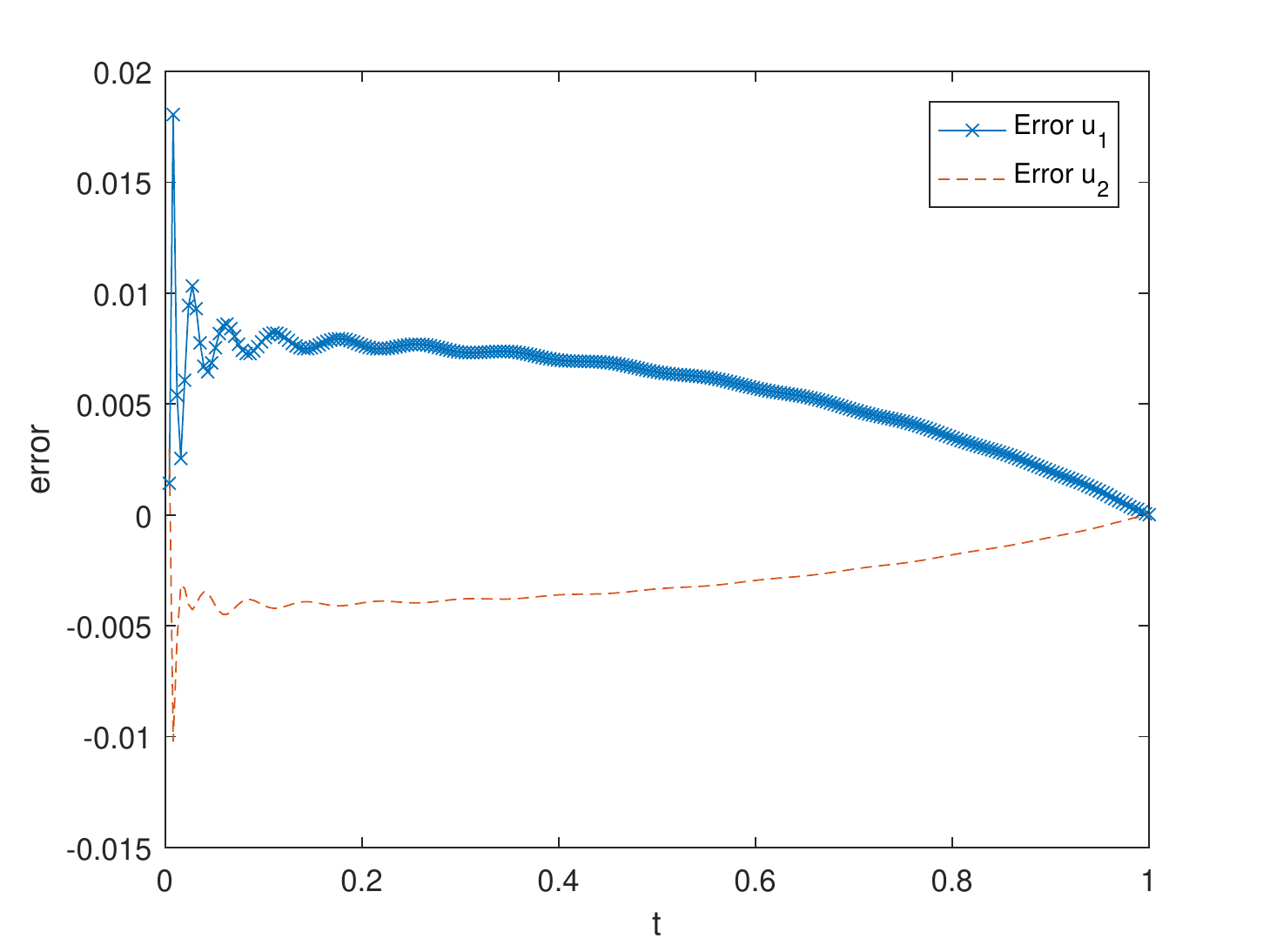}
	}
	\caption{Computed optimal control $\bm{u}^{\Delta t}$ and error $\bm{u}^{\Delta t}-\bm{u}$ for Example 1.}
	\label{controlEx1}
\end{figure}
Furthermore, we tested the proposed CG algorithm (\textbf{DI})--(\textbf{DV}) with $h=\frac{1}{2^6}$ and $\Delta t=\frac{1}{2^7}$ for different penalty parameter $\alpha_1$. The results reported in Table \ref{reg_EX1} show that the performance of the proposed CG algorithm is robust with respect to the penalty parameter, at least for the example being considered. We also observe that as $\alpha_1$ increases, the value of  $\frac{\|y_h^{\Delta t}-y_d\|_{L^2(Q)}}{\|y_d\|_{{L^2(Q)}}}$ decreases. This implies that, as expected, the computed state $y_h^{\Delta t}$ is closer to the target function $y_d$ when the penalty parameter gets larger.

\begin{table}[htpb]
	{\small
	\centering
	\caption{Results of the CG algorithm (\textbf{DI})--(\textbf{DV}) with different $\alpha_1$ for Example 1.} \begin{tabular}{|c|c|c|c|c|c|}
		\hline $\alpha_1$ &$Iter$& $CPU(s)$&$\|\bm{u}^{\Delta t}-\bm{u}\|_{L^2(0,T;\mathbb{R}^2)}$&$\|y_h^{\Delta t}-y\|_{L^2(Q)}$& $\frac{\|y_h^{\Delta t}-y_d\|_{L^2(Q)}}{\|y_d\|_{{L^2(Q)}}}$  \\
			\hline $10^4$   & 46 & 126.0666&1.3872$\times 10^{-2}$ &2.5739$\times 10^{-3}$  & 8.7666$\times 10^{-4}$ 
		\\
		\hline $10^5$   & 48 & 126.4185 &1.3908$\times 10^{-2}$ &2.5739$\times 10^{-3}$  &8.6596$\times 10^{-4}$ 
	\\
	\hline $10^6$   &48&128.2346
	&1.3912$\times 10^{-2}$
	& 2.5739$\times 10^{-3}$
	&8.5623$\times 10^{-4}$ 
	\\
		\hline $10^7$   &48 & 127.1858&1.3912$\times 10^{-2}$&2.5739$\times 10^{-3}$  &8.5612$\times 10^{-4}$  
		\\
		\hline $10^8$& 48 & 124.1160&1.3912$\times 10^{-2}$&2.5739$\times 10^{-3}$ &8.5610$\times 10^{-4}$ 
		\\
		\hline
	\end{tabular}
	\label{reg_EX1}
}
\end{table}

\medskip
\noindent\textbf{Example 2.}  As in Example 1, we consider the bilinear optimal control problem (BCP) on the domain $Q=\Omega\times(0,T)$ with $\Omega=(0,1)^2$ and $T=1$. Now, we take the control $\bm{v}(x,t)$ in the infinite-dimensional space $\mathcal{U}=\{\bm{v}|\bm{v}\in [L^2(Q)]^2, \nabla\cdot\bm{v}=0\}.$ 
We set $\alpha_2=0$ in (\ref{objective_functional}), $\nu=1$ and $a_0=1$ in (\ref{state_equation}), and consider the following tracking-type bilinear optimal control problem:
\begin{equation}\label{model_ex2}
\min_{\bm{v}\in\mathcal{U}}J(\bm{v})=\frac{1}{2}\iint_Q|\bm{v}|^2dxdt+\frac{\alpha_1}{2}\iint_Q|y-y_d|^2dxdt,
\end{equation}
where $y$ is obtained from $\bm{v}$ via the solution of the state equation (\ref{state_equation}).

First, we let
\begin{eqnarray*}
	&&y=e^t(-3\sin(2\pi x_1)\sin(\pi x_2)+1.5\sin(\pi x_1)\sin(2\pi x_2)),\\
	&&p=(T-t)\sin \pi x_1 \sin \pi x_2,
	~
	\text{and} ~\bm{u}=P_{\mathcal{U}}(p\nabla y),
\end{eqnarray*}
where $P_{\mathcal{U}}(\cdot)$ is the projection onto the set $\mathcal{U}$.

We further set
\begin{eqnarray*}
	&&f=\frac{\partial y}{\partial t}-\nabla^2y+{\bm{u}}\cdot \nabla y+y,
	\quad\phi=-3\sin(2\pi x_1)\sin(\pi x_2)+1.5\sin(\pi x_1)\sin(2\pi x_2),\\
	&&y_d=y-\frac{1}{\alpha_1}\left(-\frac{\partial p}{\partial t}
	-\nabla^2p-\bm{u}\cdot\nabla p +p\right),\quad g=0.
\end{eqnarray*}
Then, it is easy to show that $\bm{u}$ is a solution point of the problem (\ref{model_ex2}). We note that  $\bm{u}=P_{\mathcal{U}}(p\nabla y)$ has no analytical solution and it can only be solved numerically. Here, we solve $\bm{u}=P_{\mathcal{U}}(p\nabla y)$ by the preconditioned CG algorithm (\textbf{DG1})--(\textbf{DG5}) with $h=\frac{1}{2^9}$ and $\Delta t=\frac{1}{2^{10}}$, and use the resulting control $\bm{u}$ as a reference solution for the example we considered.
	\begin{figure}[htpb]
	\centering{
		\includegraphics[width=0.3\textwidth]{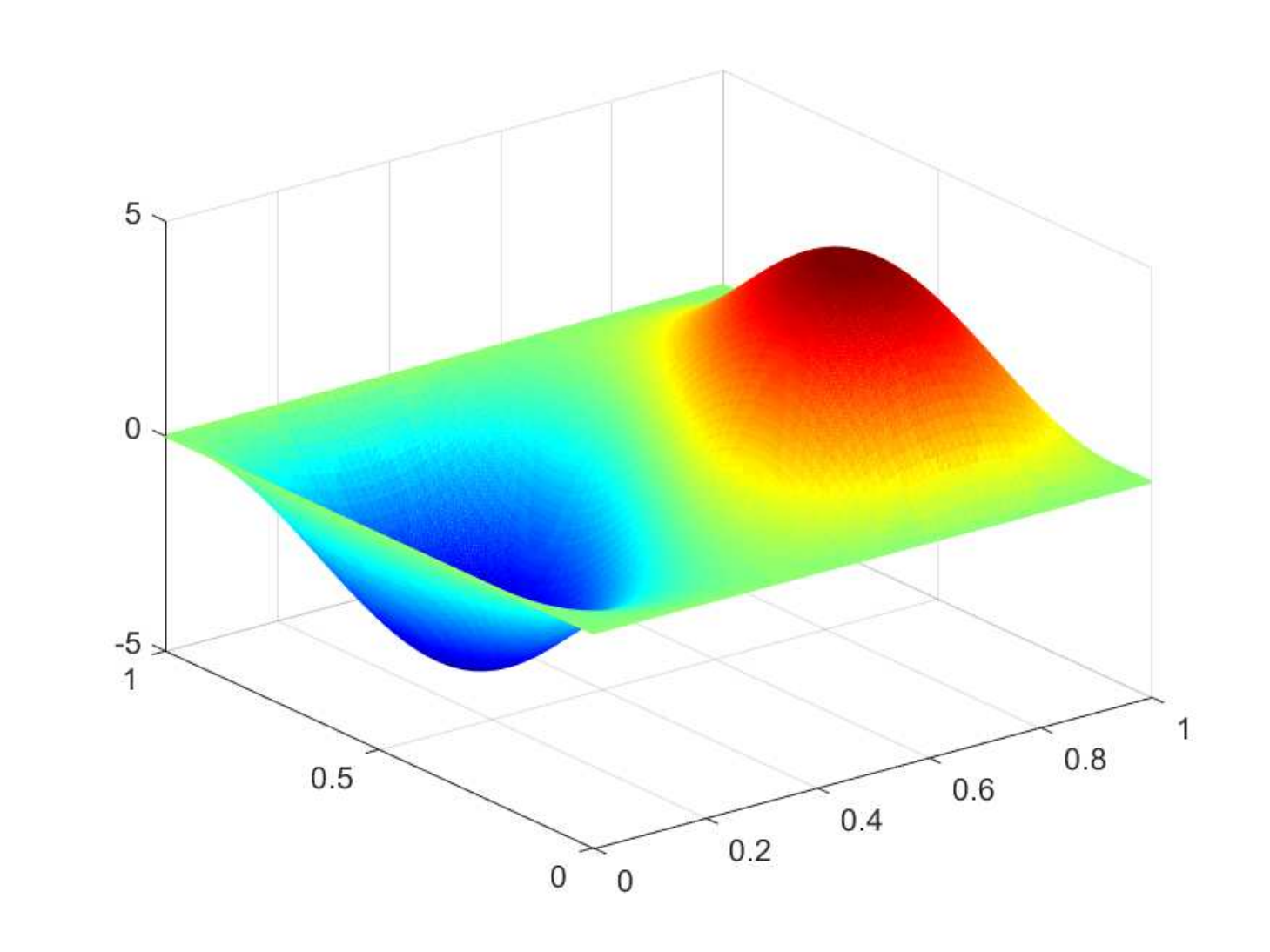}
		\includegraphics[width=0.3\textwidth]{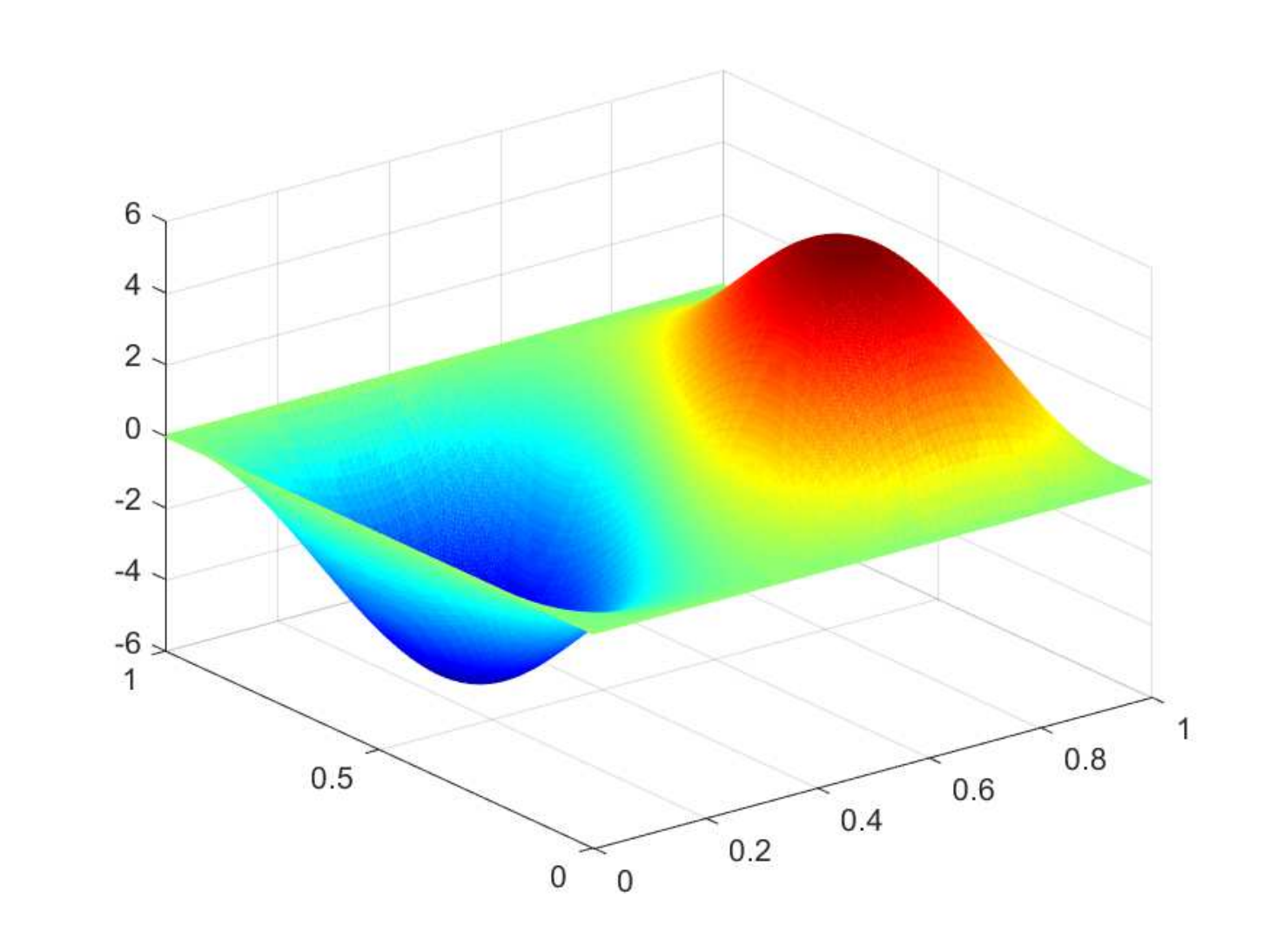}
		\includegraphics[width=0.3\textwidth]{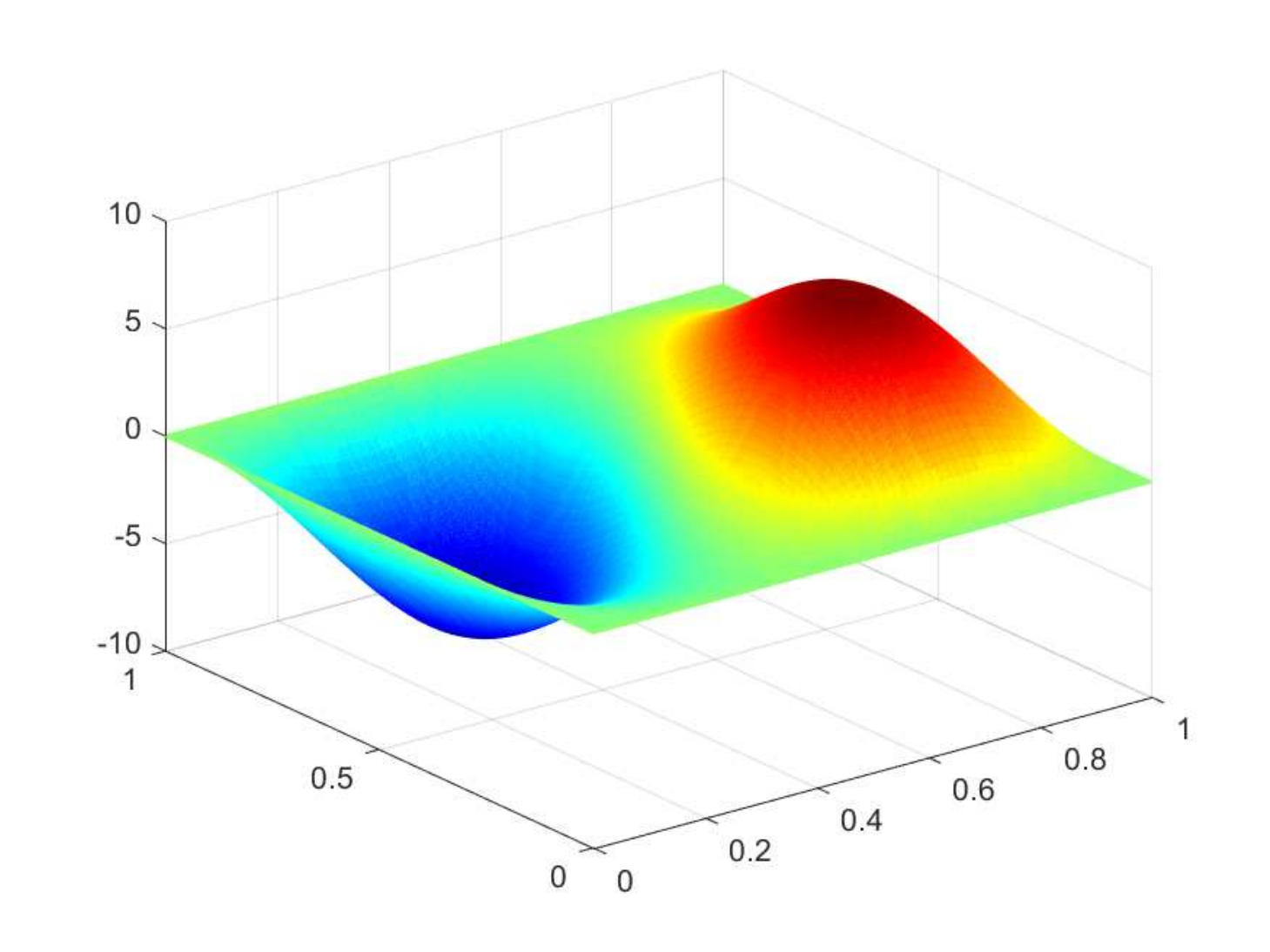}
	}
	\caption{The target function $y_d$ with $h=\frac{1}{2^7}$ and $\Delta t=\frac{1}{2^8}$  at $t=0.25, 0.5$ and $0.75$ (from left to right) for Example 2.}
	\label{target_ex2}
\end{figure}

The stopping criteria of the outer CG algorithm (\textbf{DI})--(\textbf{DV}) and the inner preconditioned CG algorithm (\textbf{DG1})--(\textbf{DG5}) are respectively set as
$$
\frac{\Delta t\sum_{n=1}^N\int_\Omega|\bm{g}_n^{k+1}|^2dx}{\Delta t\sum_{n=1}^N\int_\Omega|\bm{g}_n^{0}|^2dx}\leq 5\times10^{-8}, ~\text{and}~\frac{\int_\Omega|\nabla r^{k+1}|^2dx}{\max\{1,\int_\Omega|\nabla r^0|^2dx\}}\leq 10^{-8}.
$$
The initial values are chosen as $\bm{u}^0=(0,0)^\top$ and $\lambda^0=0$; and we denote by $\bm{u}_h^{\Delta t}$ and $y_h^{\Delta t}$ the computed control and state, respectively. 

First, we take $h=\frac{1}{2^i}, i=6,7,8$, $\Delta t=\frac{h}{2}$, $\alpha_1=10^6$, and implement the proposed nested CG algorithm (\textbf{DI})--(\textbf{DV}) for solving the problem (\ref{model_ex2}). The numerical results reported in Table \ref{tab:mesh_EX2} show that the CG algorithm converges fast and is robust with respect to different mesh sizes. In addition, the preconditioned CG algorithm (\textbf{DG1})--(\textbf{DG5}) converges within 10 iterations for all cases and thus is efficient for computing the gradient $\{\bm{g}_n\}_{n=1}^N$.
We also observe that the target function $y_d$ has been reached within a good accuracy. Similar comments hold for the approximation of the optimal control $\bm{u}$ and of the state $y$ of problem (\ref{model_ex2}).

	\begin{table}[htpb]
		{\small
	\centering
	\caption{Results of the nested CG algorithm (\textbf{DI})--(\textbf{DV}) with different $h$ and $\Delta t$ for Example 2.}
	\begin{tabular}{|c|c|c|c|c|c|c|}
		\hline Mesh sizes &{{$Iter_{CG}$}}&$MaxIter_{PCG}$& $\|\bm{u}_h^{\Delta t}-\bm{u}\|_{L^2(Q)}$&$\|y_h^{\Delta t}-y\|_{L^2(Q)}$& $\frac{\|y_h^{\Delta t}-y_d\|_{L^2(Q)}}{\|y_d\|_{{L^2(Q)}}}$
		\\
		\hline $h=1/2^6,\Delta t=1/2^7$   &443&9&3.7450$\times 10^{-3}$& 9.7930$\times 10^{-5}$&1.0906$\times 10^{-6}$  
		\\
		\hline $h=1/2^7,\Delta t=1/2^8$   &410&9&1.8990$\times 10^{-3}$&  1.7423$\times 10^{-5}$ & 3.3863$\times 10^{-7}$  
		\\
		\hline $h=1/2^8,\Delta t=1/2^9$& 405&8 &1.1223$\times 10^{-3}$ &4.4003$\times 10^{-6}$  &1.0378$\times 10^{-7}$ 
		\\
		\hline
	\end{tabular}
	\label{tab:mesh_EX2}
}

\end{table}

Taking $h=\frac{1}{2^7}$ and $\Delta t=\frac{1}{2^8}$, the computed state  $y_h^{\Delta t}$, the error $y_h^{\Delta t}-y$ and $y_h^{\Delta t}-y_d$ at $t=0.25,0.5,0.75$ are reported in Figures \ref{stateEx2_1}, \ref{stateEx2_2} and \ref{stateEx2_3}, respectively; and the computed control $\bm{u}_h^{\Delta t}$, the exact control $\bm{u}$, and the error $\bm{u}_h^{\Delta t}-\bm{u}$ at $t=0.25,0.5,0.75$ are presented in Figures \ref{controlEx2_1}, \ref{controlEx2_2} and \ref{controlEx2_3}.
\begin{figure}[htpb]
	\centering{
		\includegraphics[width=0.3\textwidth]{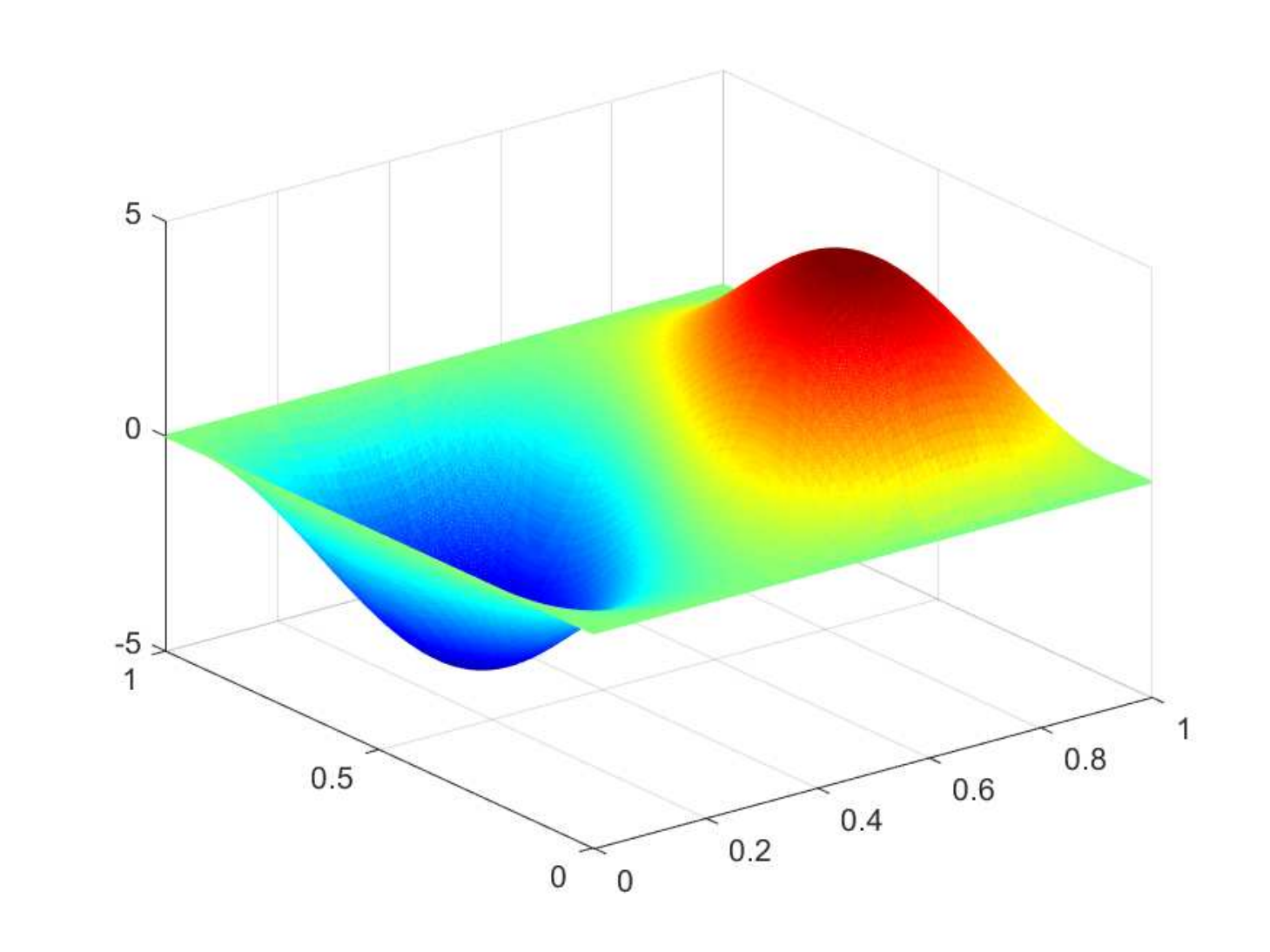}
		\includegraphics[width=0.3\textwidth]{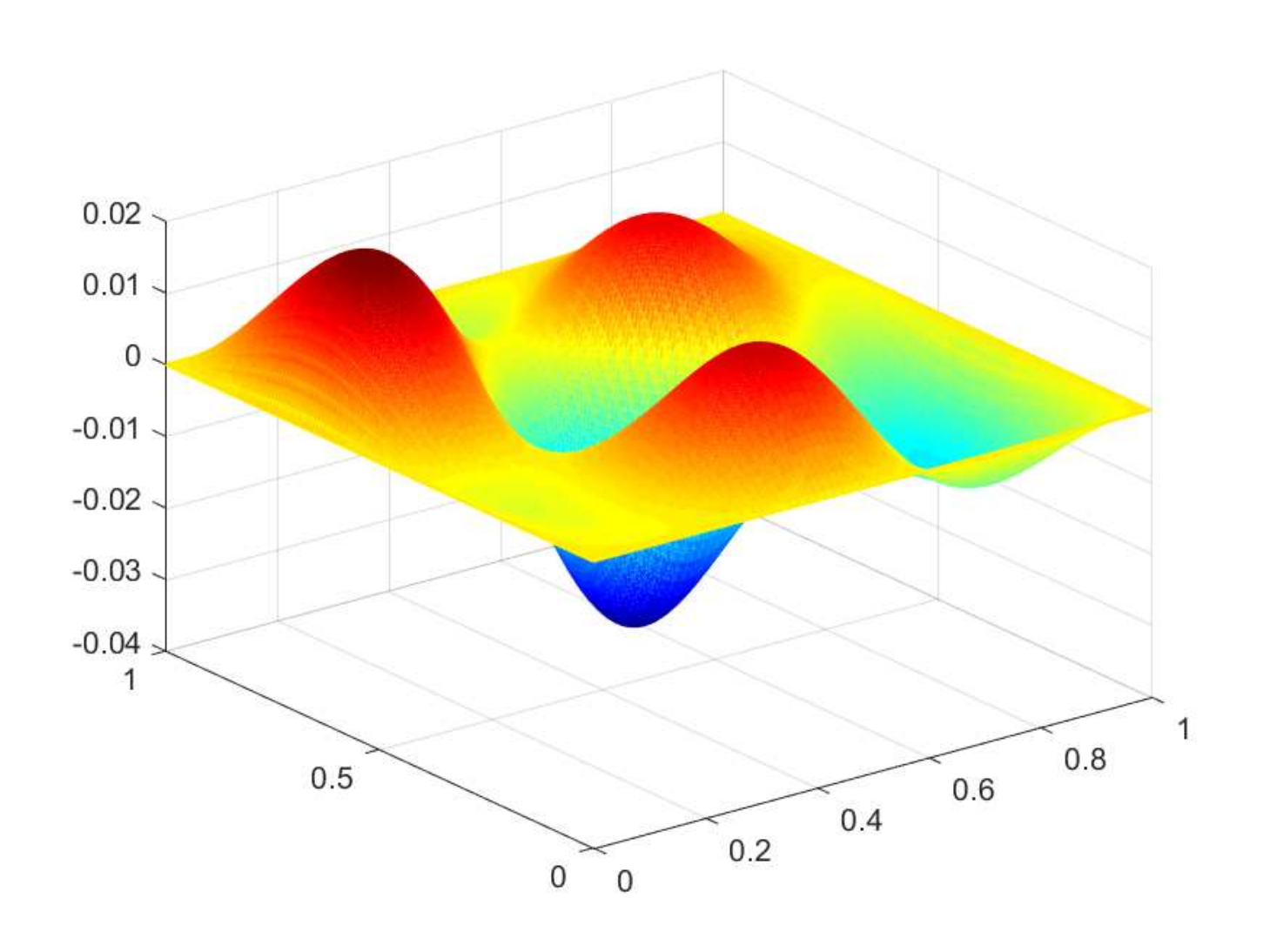}
		\includegraphics[width=0.3\textwidth]{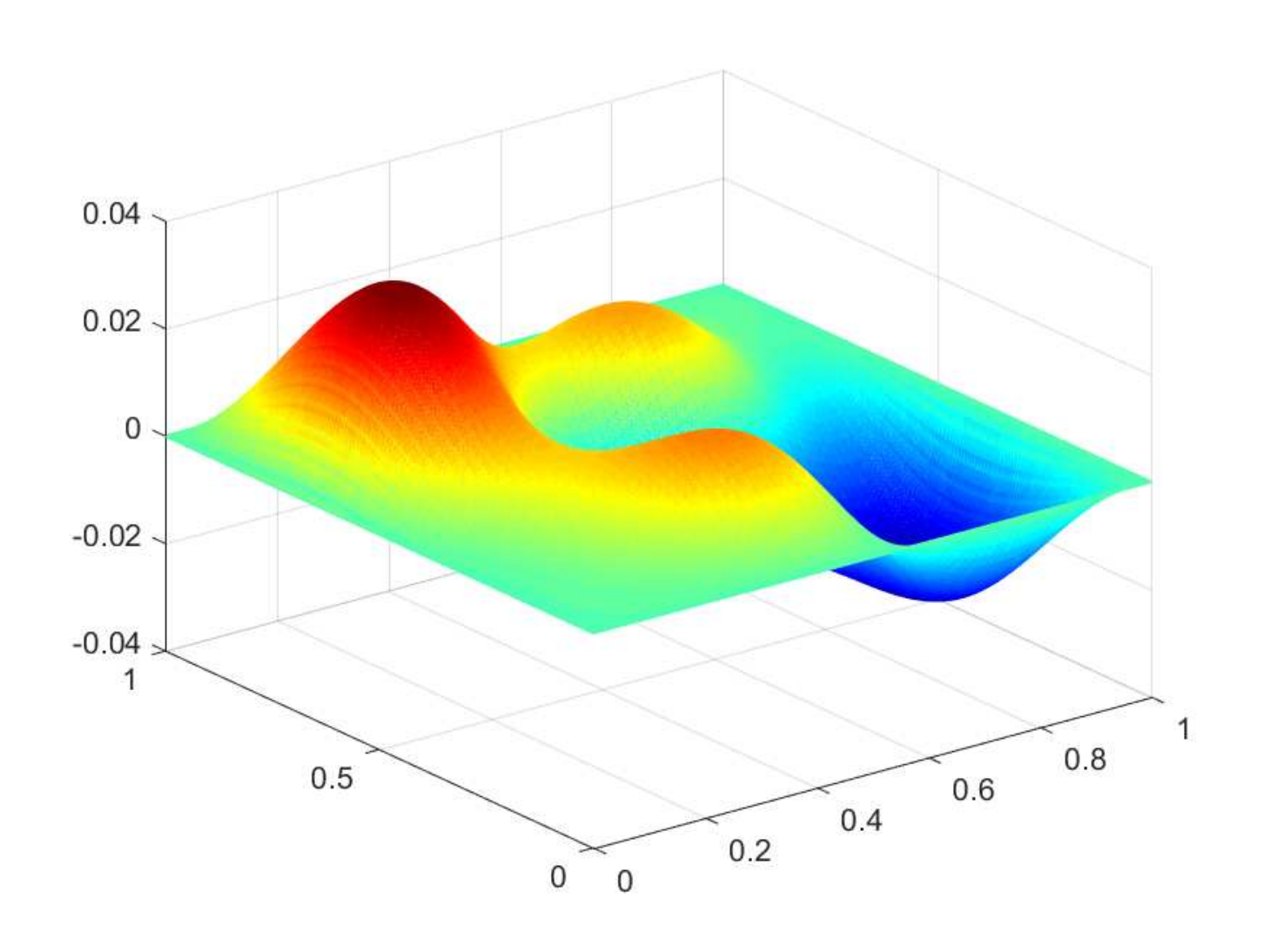}}
	\caption{Computed state $y^{\Delta t}_h$, error $y^{\Delta t}_h-y$ and $y^{\Delta t}_h-y_d$ with $h=\frac{1}{2^7}$ and $\Delta t=\frac{1}{2^8}$  (from left to right) at $t=0.25$ for Example 2.}
	\label{stateEx2_1}
\end{figure}
\begin{figure}[htpb]
	\centering{
		\includegraphics[width=0.3\textwidth]{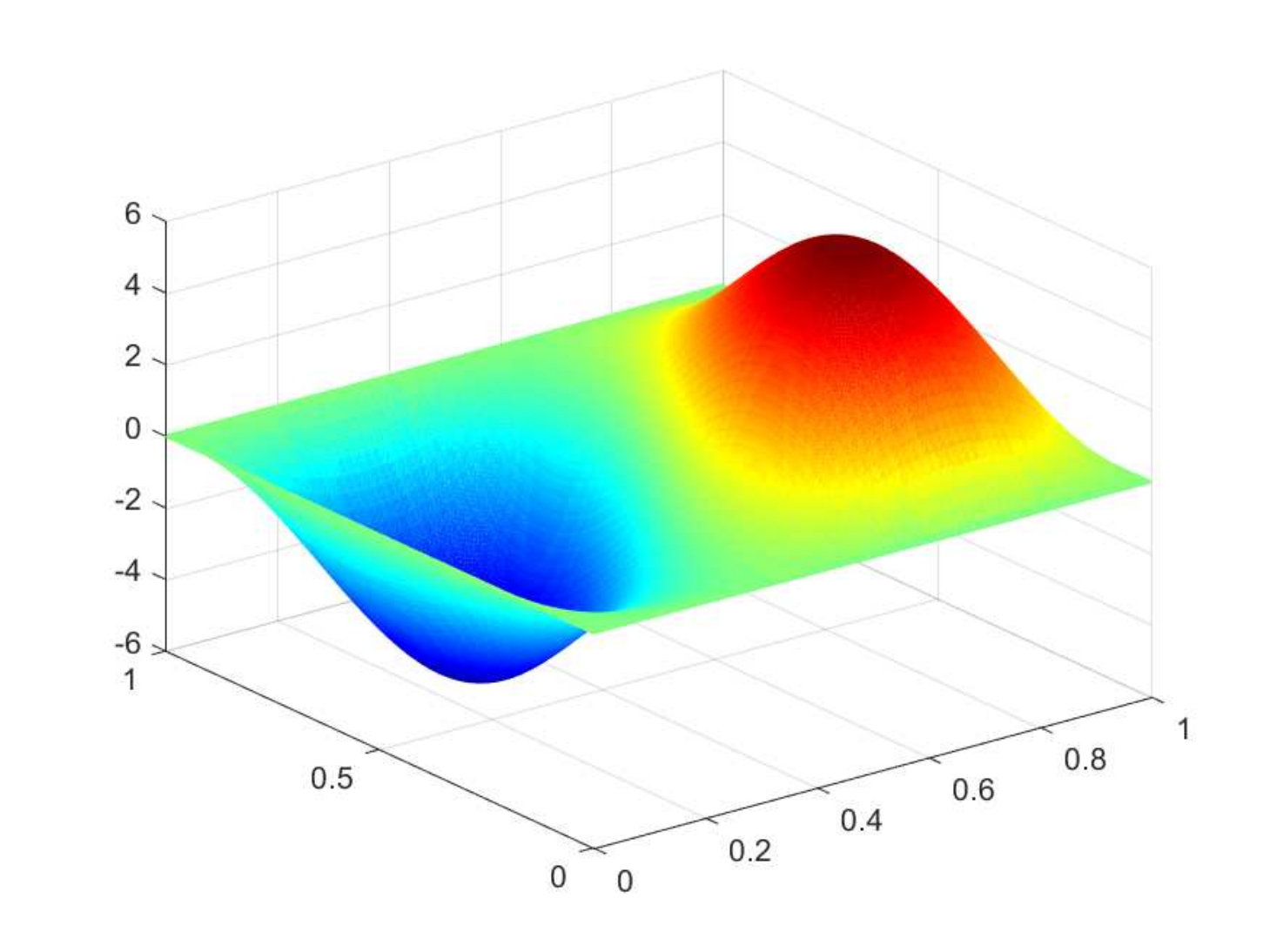}
		\includegraphics[width=0.3\textwidth]{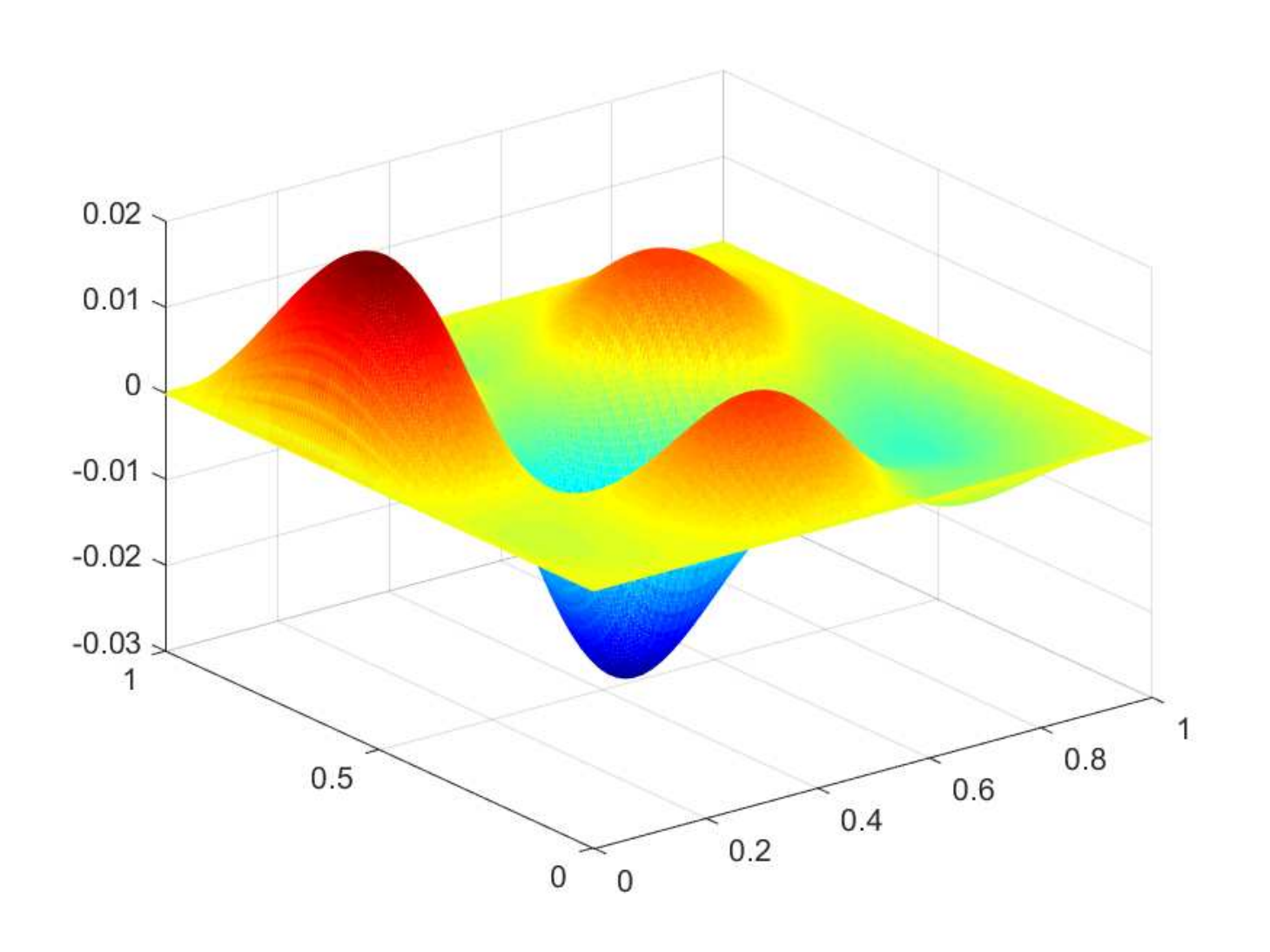}
		\includegraphics[width=0.3\textwidth]{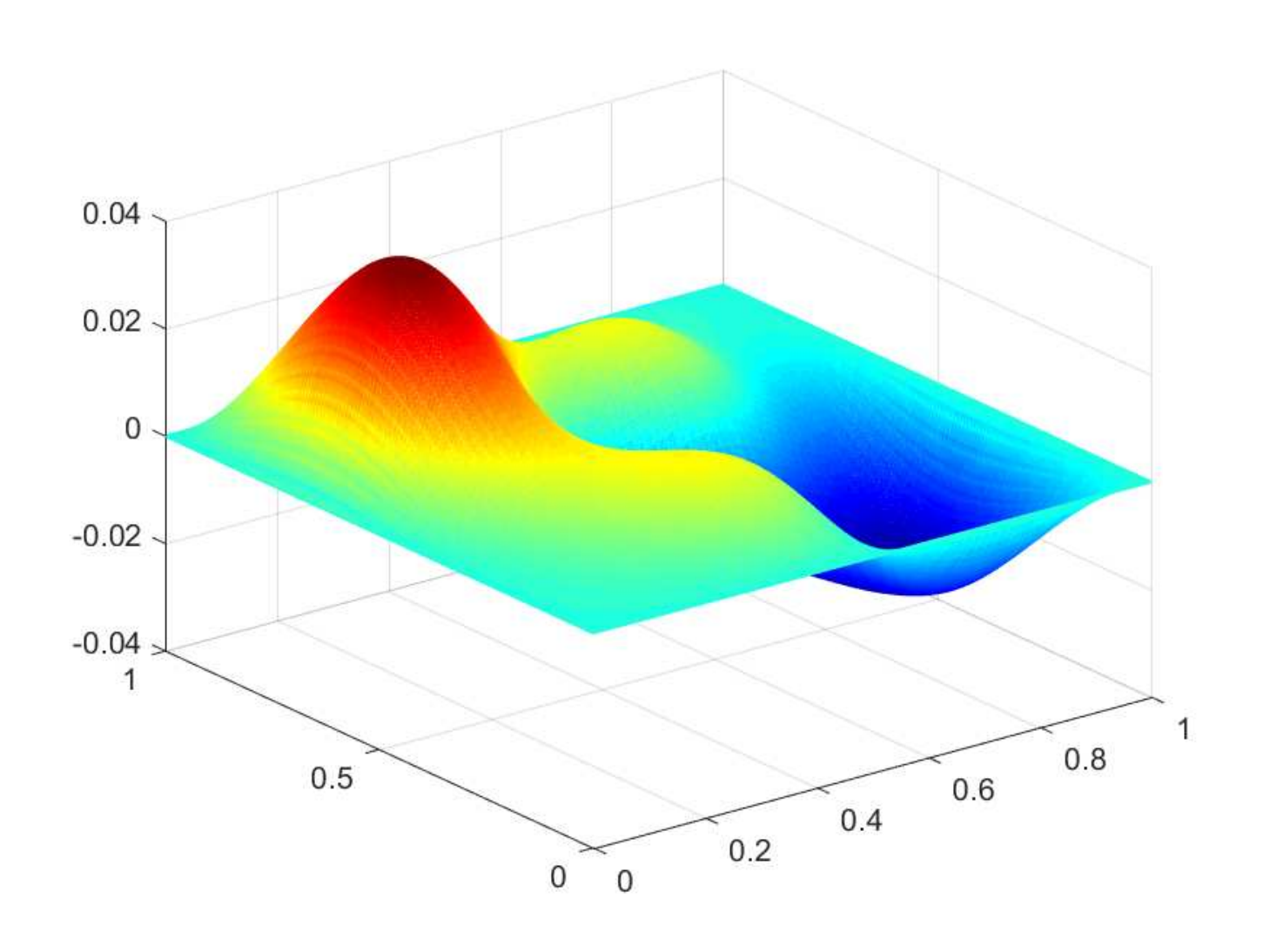}}
	\caption{Computed state $y^{\Delta t}_h$, error $y^{\Delta t}_h-y$ and $y^{\Delta t}_h-y_d$ with $h=\frac{1}{2^7}$ and $\Delta t=\frac{1}{2^8}$ (from left to right) at $t=0.5$ for Example 2.}
	\label{stateEx2_2}
\end{figure}
\begin{figure}[htpb]
	\centering{
		\includegraphics[width=0.3\textwidth]{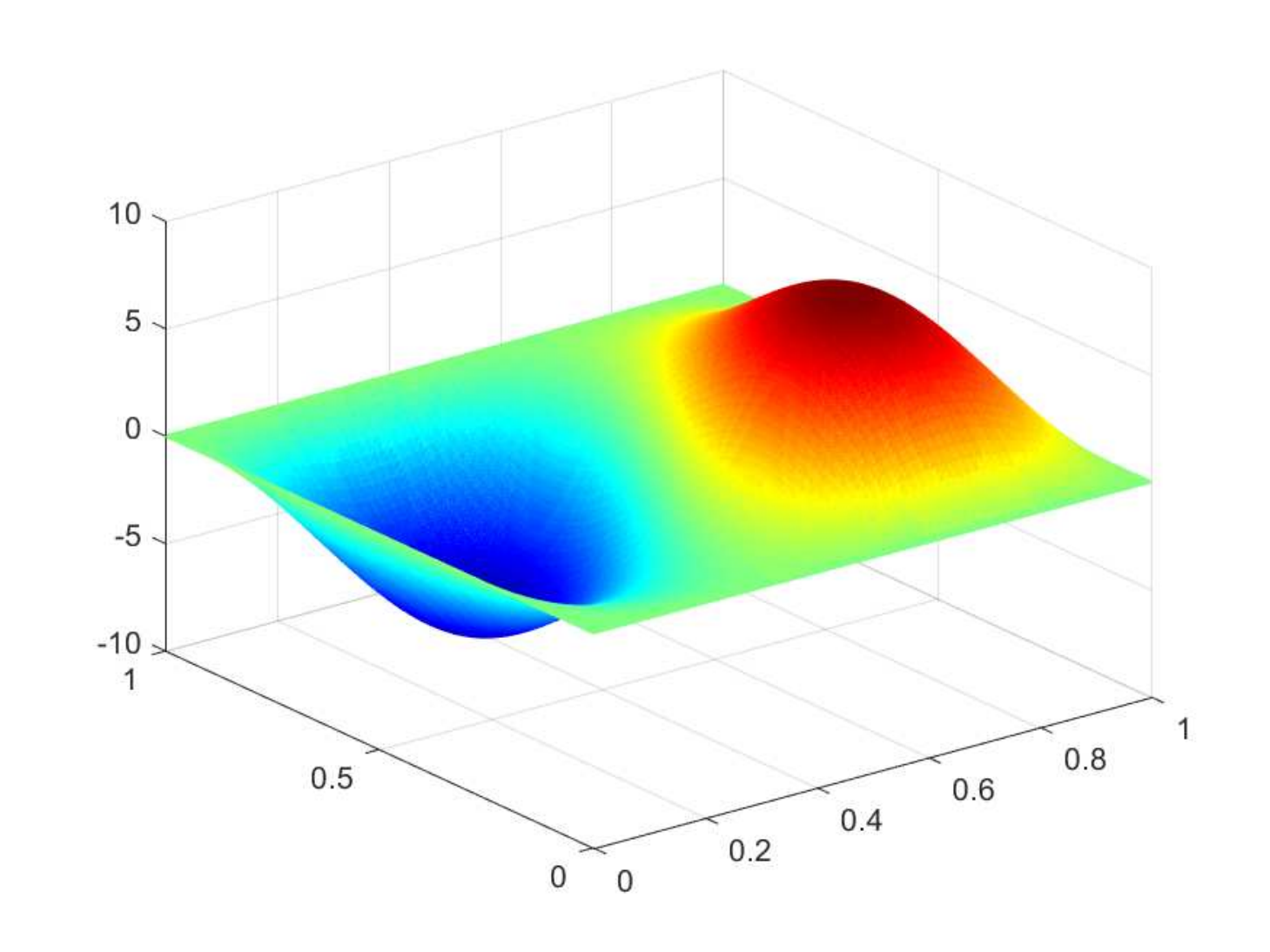}
		\includegraphics[width=0.3\textwidth]{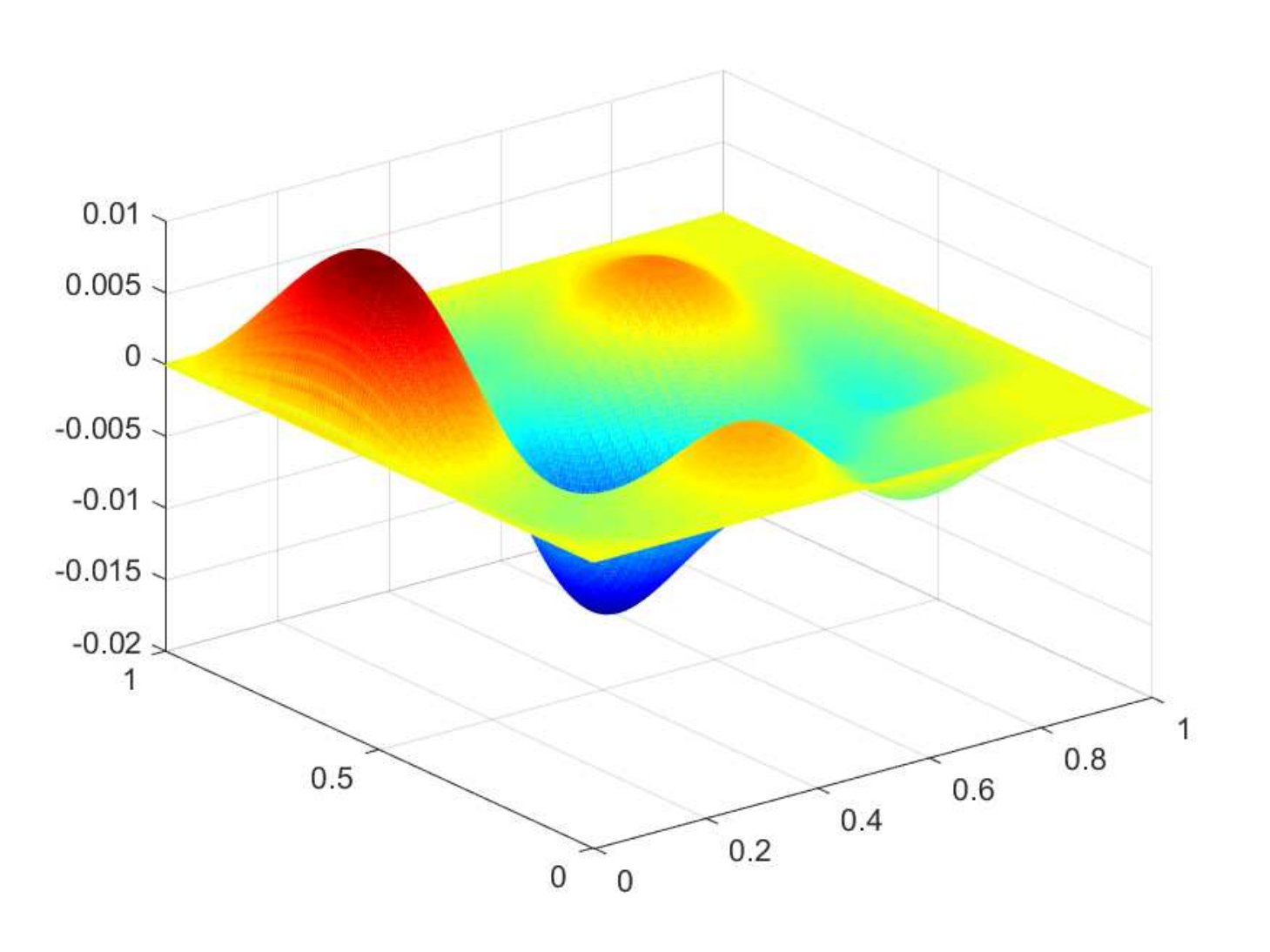}
		\includegraphics[width=0.3\textwidth]{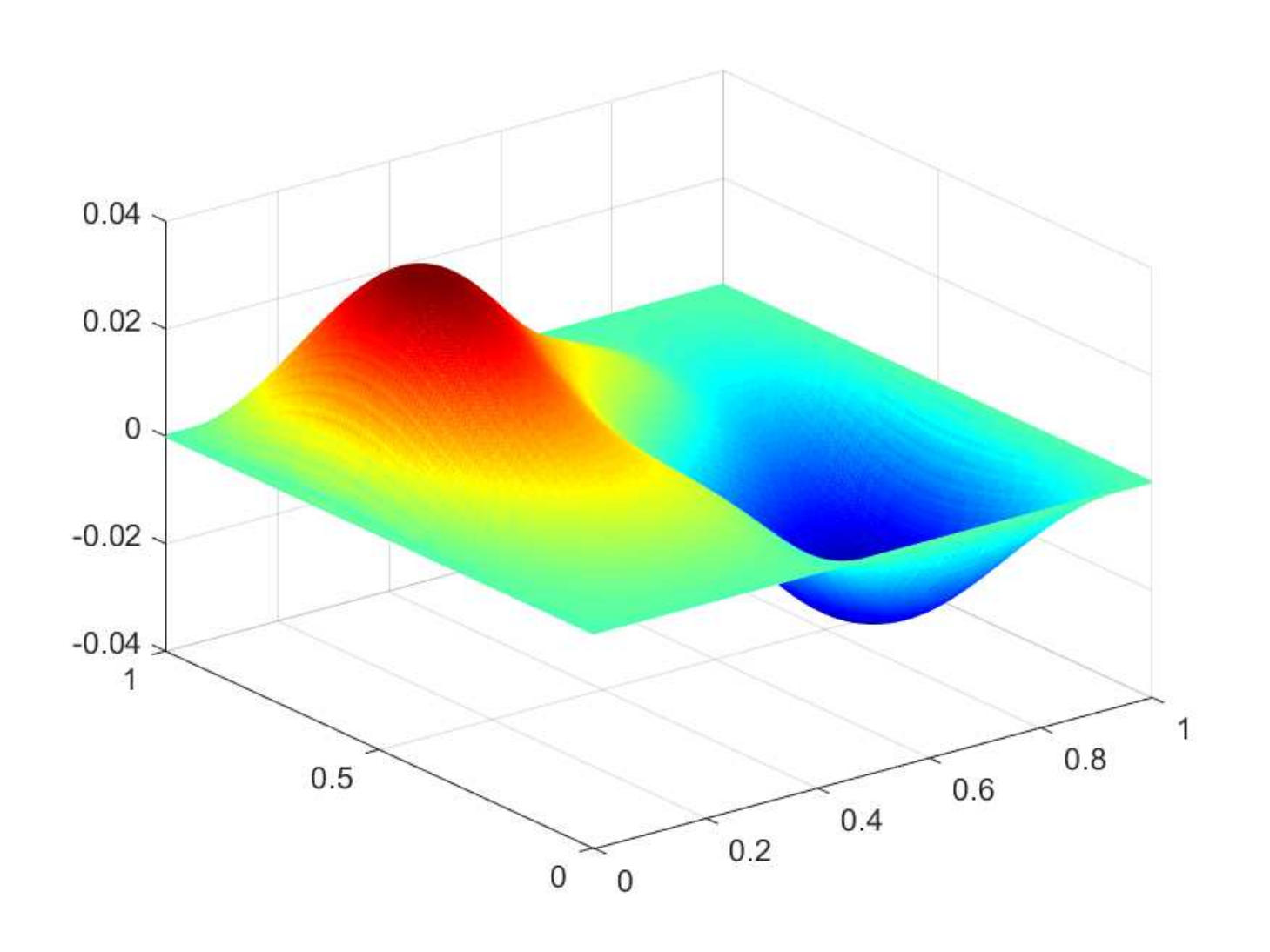}}
	\caption{Computed state $y^{\Delta t}_h$, error $y^{\Delta t}_h-y$ and $y^{\Delta t}_h-y_d$ with $h=\frac{1}{2^7}$ and $\Delta t=\frac{1}{2^8}$ (from left to right) at $t=0.75$ for Example 2.}
	\label{stateEx2_3}
\end{figure}

\begin{figure}[htpb]
	\centering{
		\includegraphics[width=0.45\textwidth]{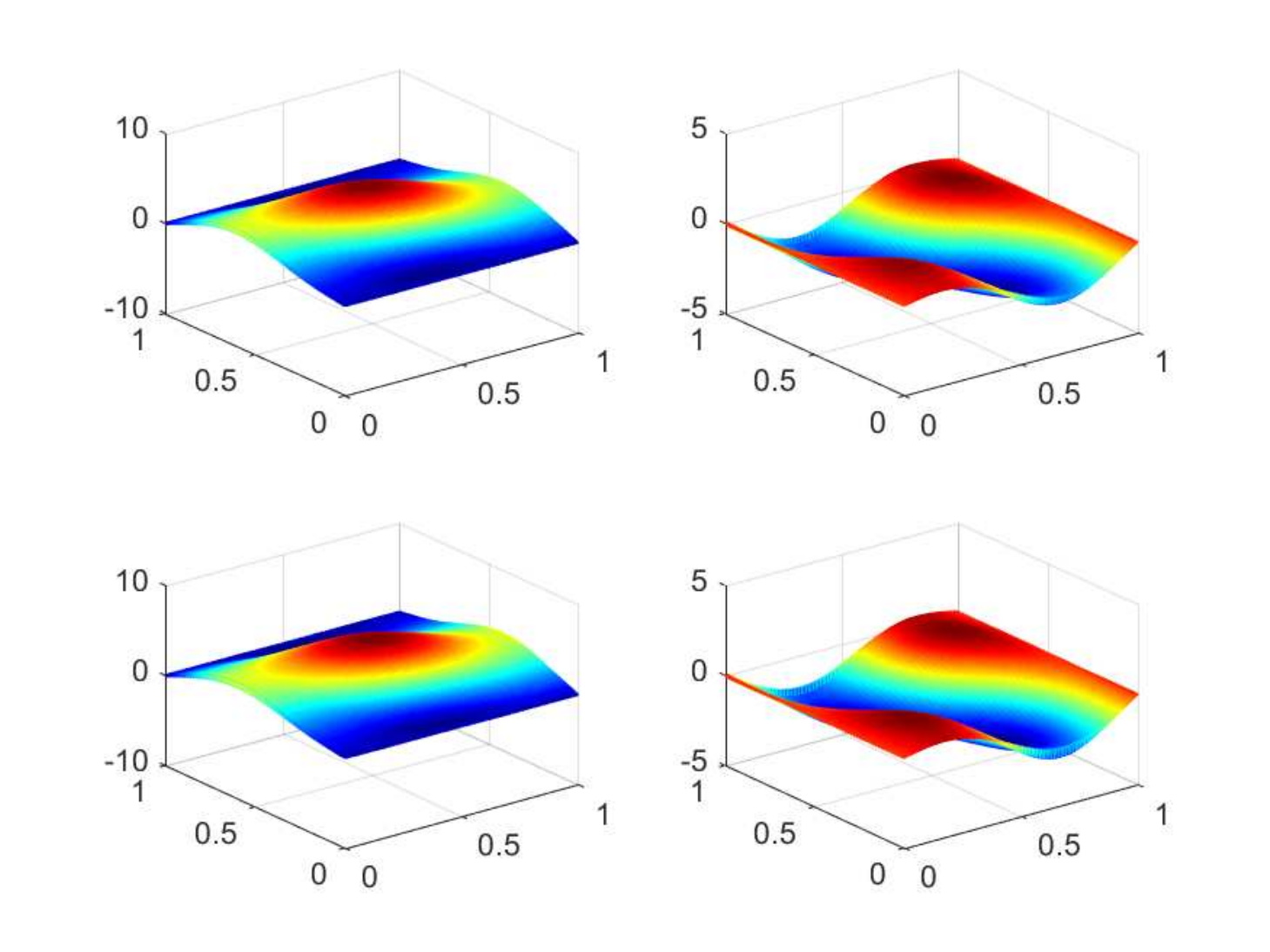}
	  \includegraphics[width=0.45\textwidth]{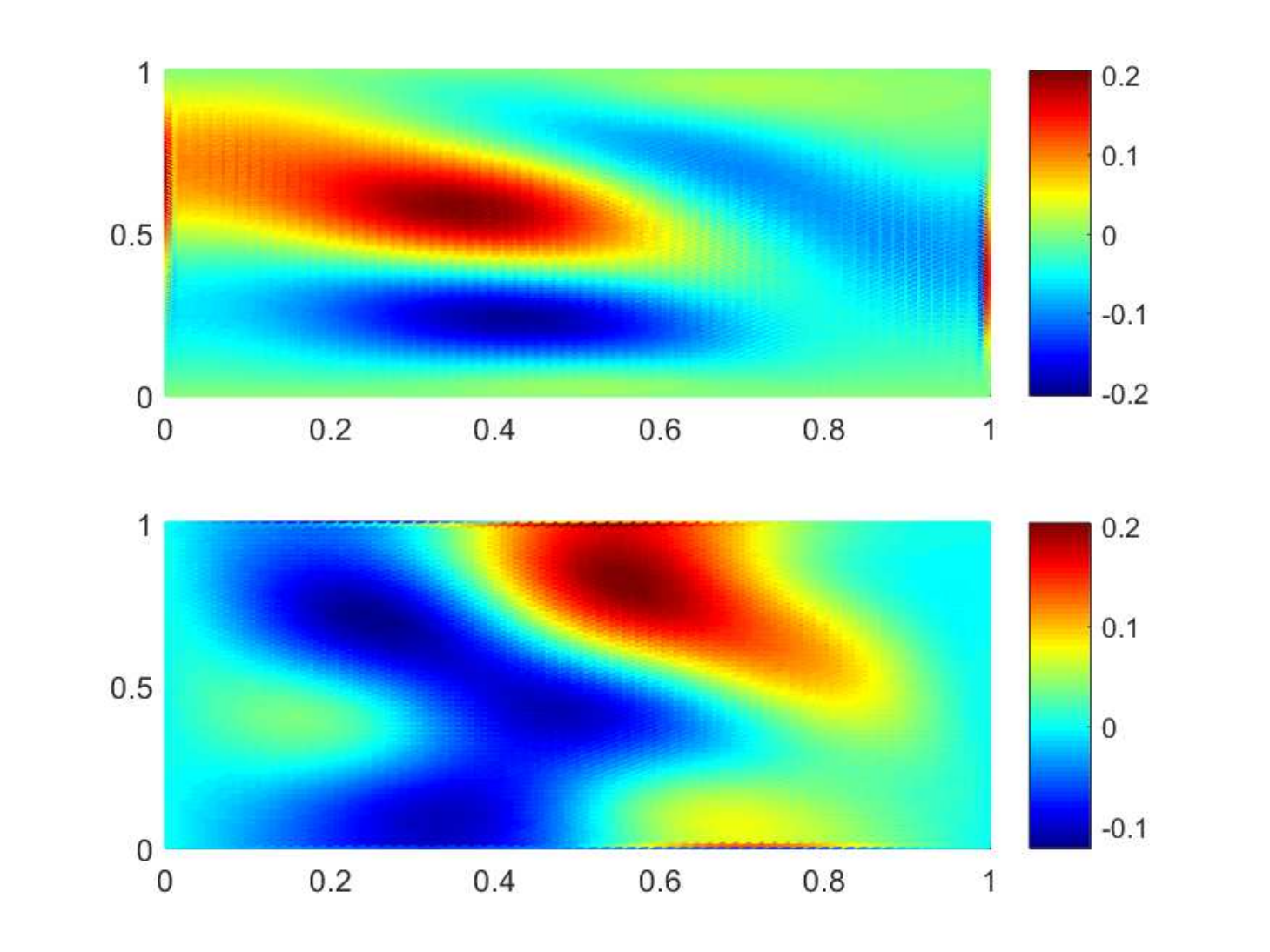}}
	\caption{Computed control $\bm{u}^{\Delta t}_h$ and exact control $\bm{u}$ (left, from top to bottom)  and the error $\bm{u}^{\Delta t}_h-\bm{u}$ (right) with $h=\frac{1}{2^7}$ and $\Delta t=\frac{1}{2^8}$ at $t=0.25$ for Example 2.}
	\label{controlEx2_1}
\end{figure}

\begin{figure}[htpb]
	\centering{
	\includegraphics[width=0.45\textwidth]{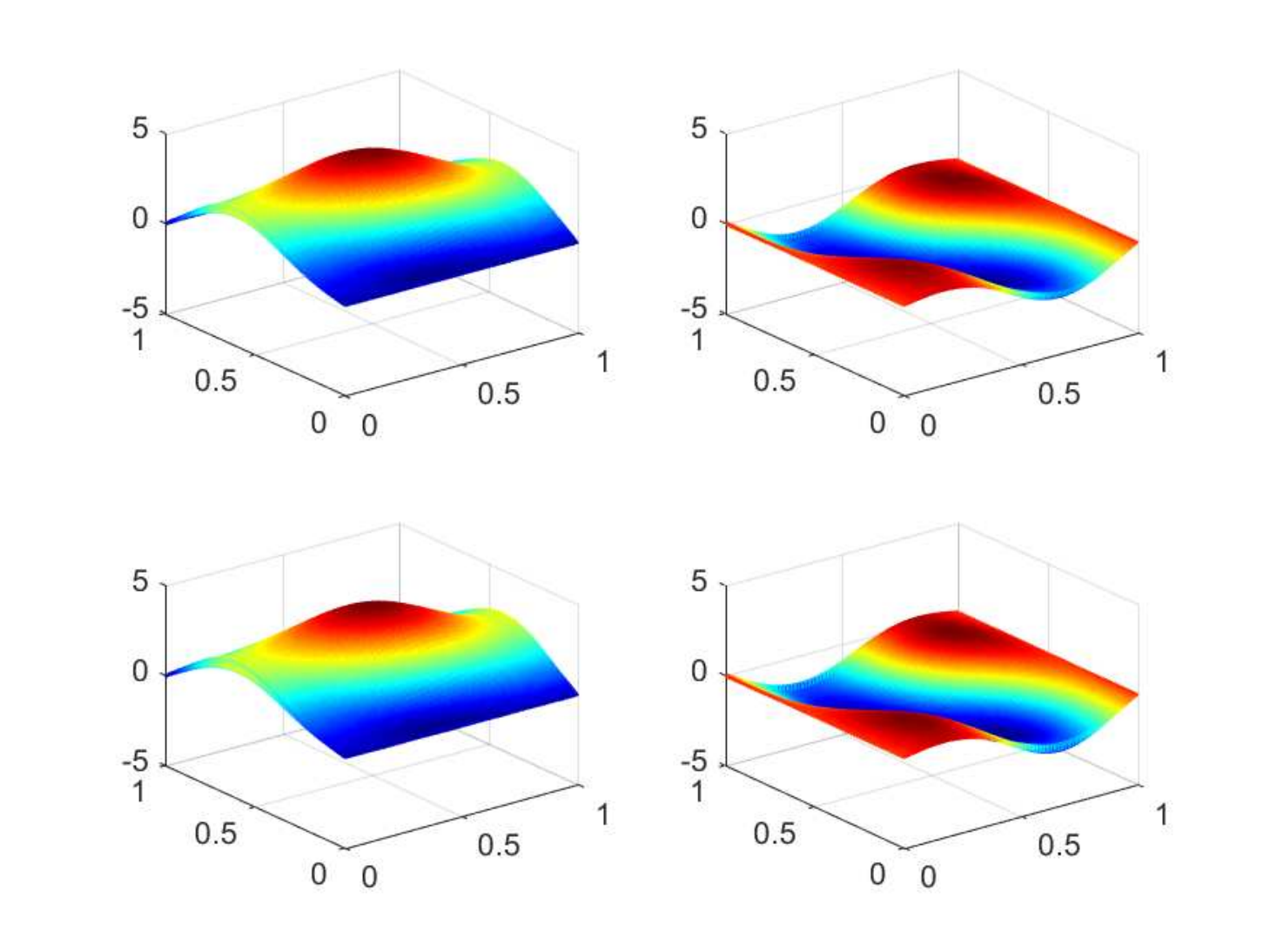}
		\includegraphics[width=0.45\textwidth]{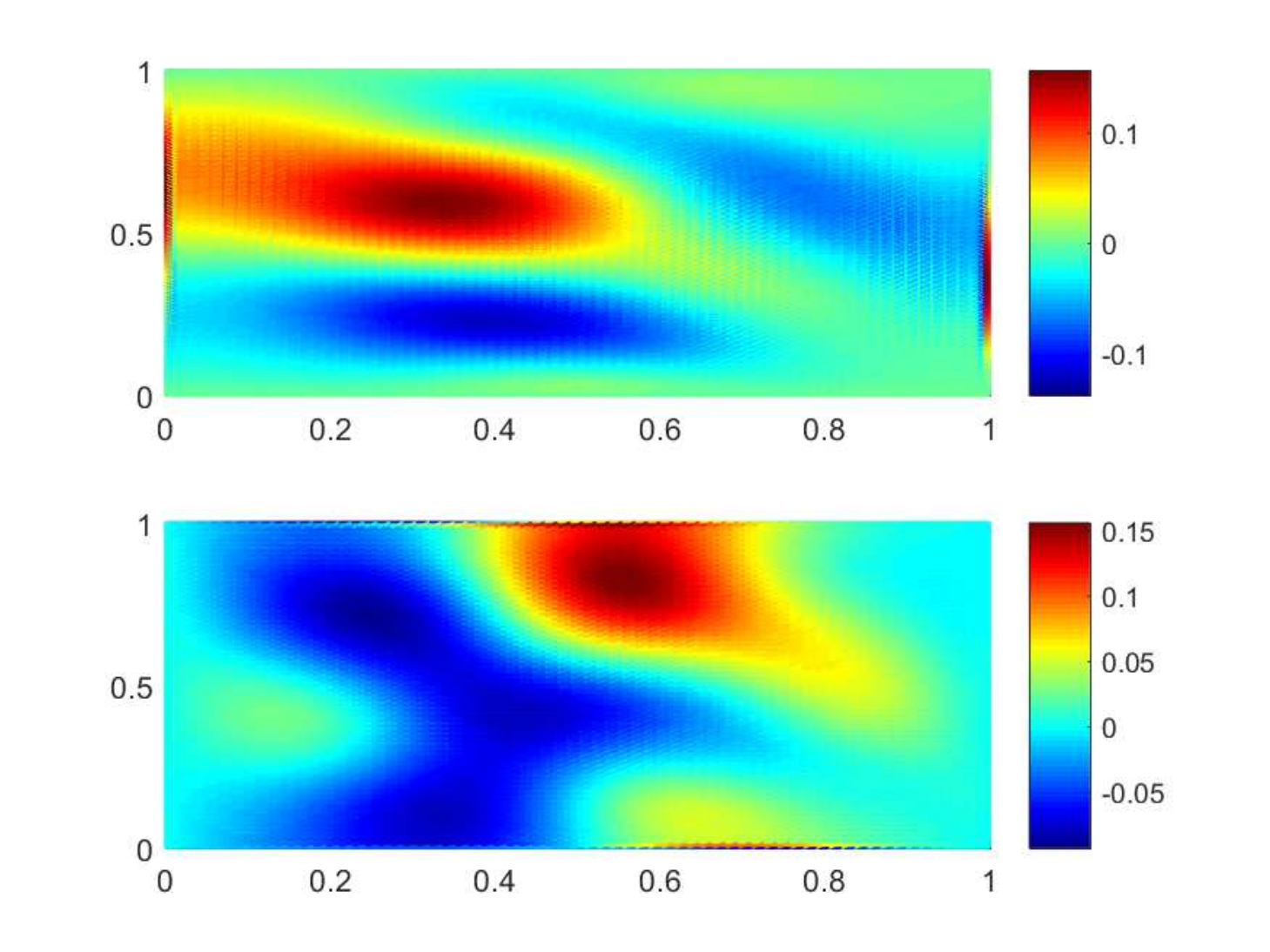}}
	\caption{Computed control $\bm{u}^{\Delta t}_h$ and exact control $\bm{u}$ (left, from top to bottom)  and the error $\bm{u}^{\Delta t}_h-\bm{u}$ (right) with $h=\frac{1}{2^7}$ and $\Delta t=\frac{1}{2^8}$ at $t=0.5$ for Example 2.}
	\label{controlEx2_2}
\end{figure}
\begin{figure}[htpb]
	\centering{
		\includegraphics[width=0.45\textwidth]{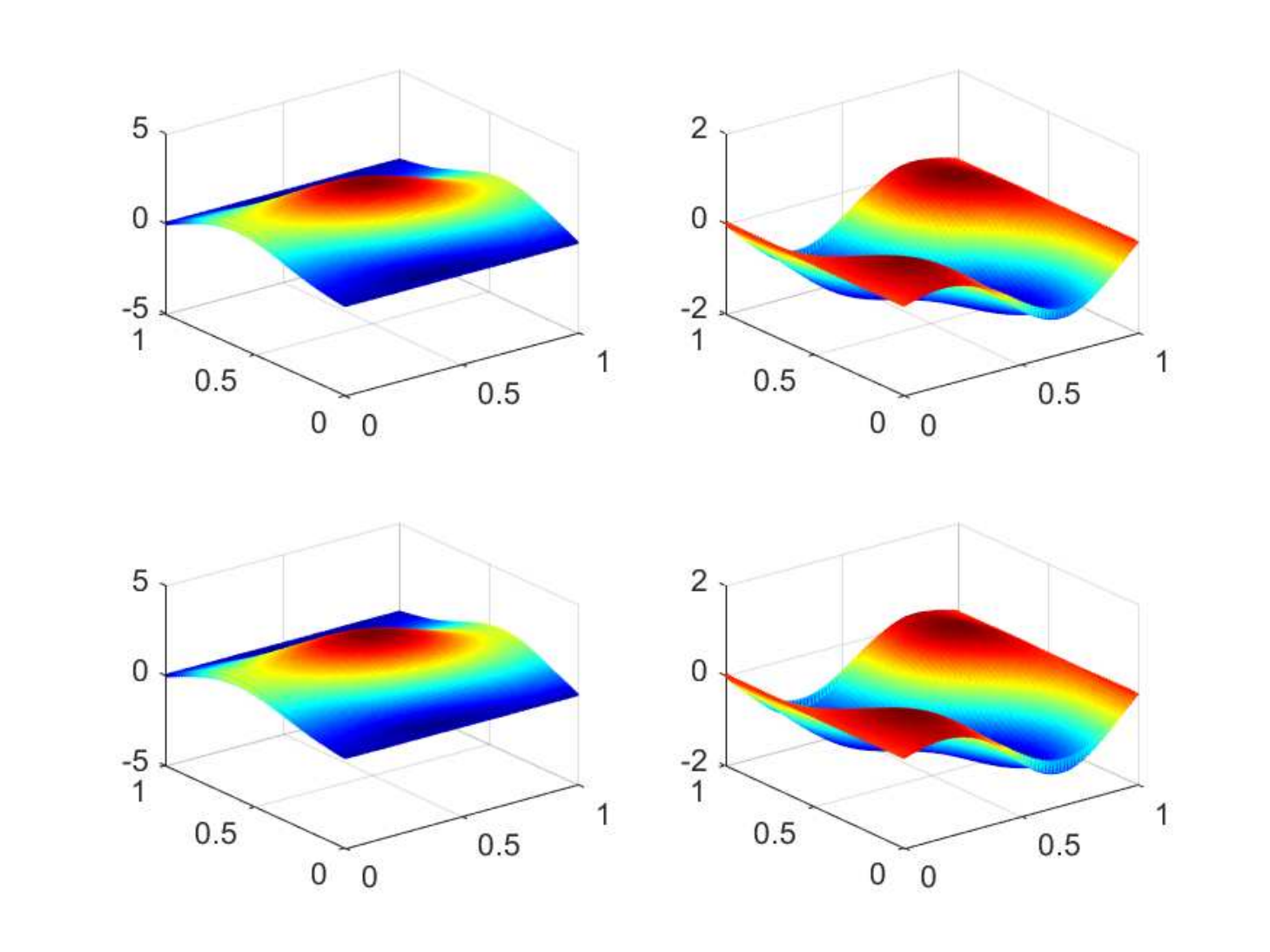}
		\includegraphics[width=0.45\textwidth]{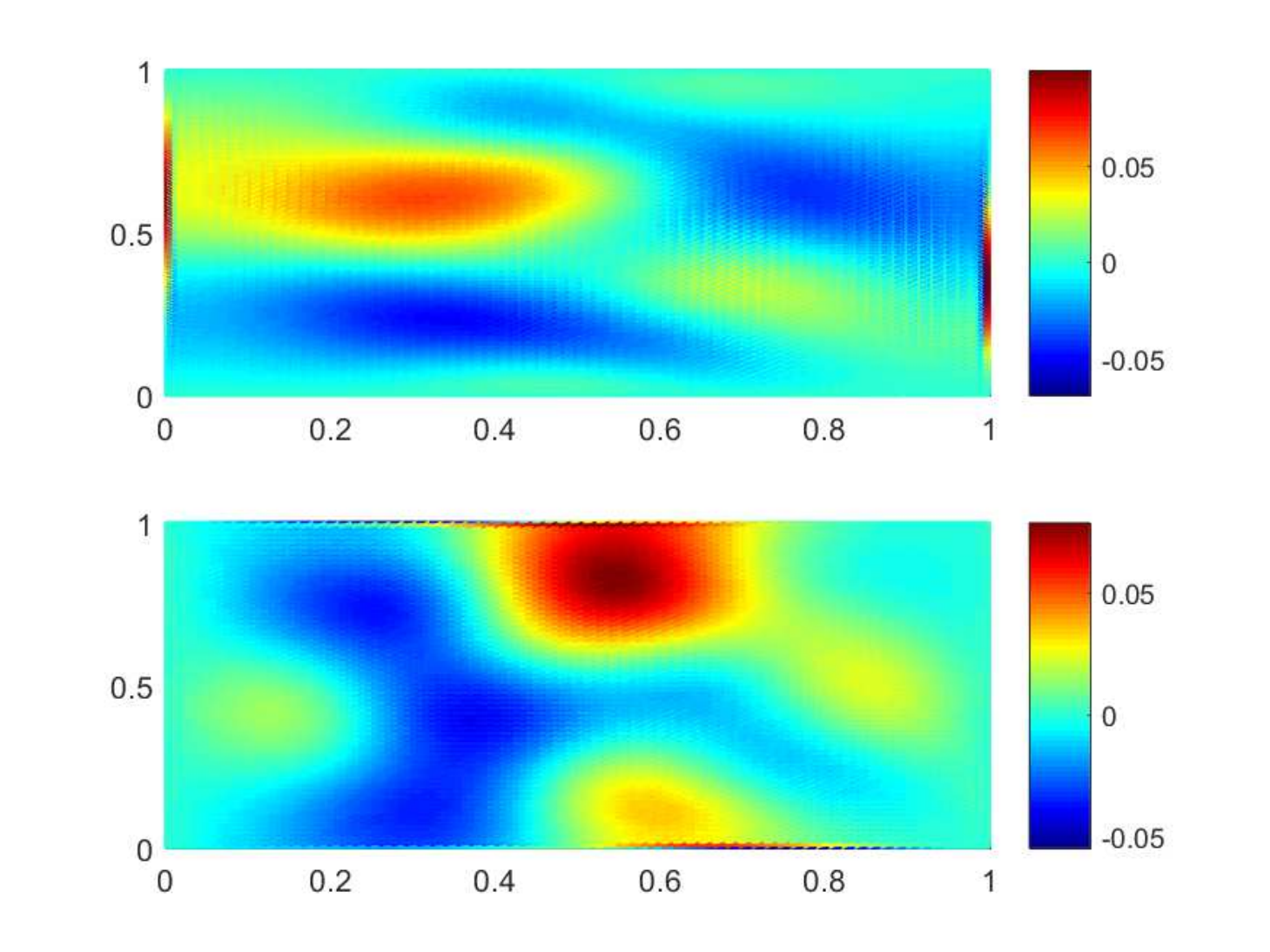}}
	\caption{Computed control $\bm{u}^{\Delta t}_h$ and exact control $\bm{u}$ (left, from top to bottom)  and the error $\bm{u}^{\Delta t}_h-\bm{u}$ (right) with $h=\frac{1}{2^7}$ and $\Delta t=\frac{1}{2^8}$ at $t=0.75$ for Example 2.}
	\label{controlEx2_3}
\end{figure}
\newpage
\section{Conclusion and Outlook}\label{se:conclusion}
We studied the bilinear control of an advection-reaction-diffusion system, where the control variable enters the model as a velocity field of the advection term. Mathematically, we proved the existence of optimal controls and derived the associated first-order optimality conditions. Computationally, the conjugate gradient (CG) method was suggested and its implementation is nontrivial. In particular, an additional divergence-free constraint on the control variable leads to a projection subproblem to compute the gradient; and the computation of a stepsize at each CG iteration requires solving the state equation repeatedly due to the nonlinear relation between the state and control variables. 
To resolve the above issues, we reformulated the gradient computation as a Stokes-type problem and proposed a fast preconditioned CG method to solve it. We also proposed an efficient inexactness strategy to determine the stepsize, which only requires the solution of one linear parabolic equation.  
An easily implementable nested
CG method was thus proposed. For the numerical discretization, we employed the standard piecewise linear finite element method and the Bercovier-Pironneau finite element method for the space discretizations of the bilinear optimal control and the Stokes-type problem, respectively, and a semi-implicit finite difference method for the time discretization. The resulting algorithm was shown to be
numerically efficient by some preliminary numerical experiments. 

 We focused in this paper on an advection-reaction-diffusion system controlled by a general form velocity field. In a real physical system, the velocity field may be
determined by some partial differential equations (PDEs), such as the Navier-Stokes equations. As a result, we meet some bilinear optimal control problems constrained by coupled PDE systems. Moreover, instead of (\ref{objective_functional}), one can also consider other types of objective functionals in the bilinear optimal control of an advection-reaction-diffusion system. For instance, one can incorporate $\iint_{Q}|\nabla \bm{v}|^2dxdt$ and $\iint_{Q}|\frac{\partial \bm{v}}{\partial t}|^2dxdt$ into the objective functional to promote that the optimal velocity field has the least rotation and is almost steady, respectively, which are essential in e.g., mixing enhancement for different flows \cite{liu2008}. All these problems are of practical interest but more challenging from algorithmic design perspectives, and they have not been well-addressed numerically in the literature. Our current work has laid a solid foundation for solving these problems and we leave them in the future.

\bibliographystyle{amsplain}

{\small
	\begin{thebibliography}{10}
\bibitem{BV07}
 R. Becker and B. Vexler, {\em Optimal control of the convection-diffusion equation using stabilized finite element methods}, Numerische Mathematik, 106 (2007),
pp.~349--367.
\bibitem{BP79}
M. Bercovier and O. Pironneau, {\em Error estimates for finite element method solution of the Stokes problem in the primitive variables}. Numerische Mathematik, 33 (1979), pp.~211--224.
\bibitem{borzi2015}
{ A. Borz{\`i}, E.-J. Park and M. Vallejos Lass}, {\em Multigrid optimization methods for the optimal control of convection-diffusion problems with bilinear control}, Journal of Optimization Theory and Applications, 168 (2016), pp.~510--533.

\bibitem{cannarsa2017}
P. Cannarsa, G. Floridia and A. Y. Khapalov, {\em Multiplicative controllability for semilinear reaction–diffusion equations with finitely many changes of sign}, Journal de Math{\'e}matiques Pures et Appliqu{\'e}es 108 (2017), pp. 425--458.
		
\bibitem{carthelglowinski1994}
C. Carthel, R. Glowinski and J. L. Lions, {\em On exact and approximate boundary controllabilities for the heat equation: a numerical approach}, Journal of Optimization Theory and Applications, 82 (1994), pp.~ 429--484.		
		
\bibitem{DQ05}
L. Dede' and A. Quarteroni, {\em Optimal control and numerical adaptivity for advection-diffusion equations}, ESAIM: Mathematical Modelling and Numerical Analysis, 39 (2005), pp.~1019--1040.
\bibitem{fleig2017}
{ A. Fleig and R. Guglielmi}, {\em Optimal control of the Fokker--Planck equation with space-dependent controls}, Journal of Optimization Theory and Applications, 174 (2017), pp.~408--427.

\bibitem{glowinski1992}
R. Glowinski, {\em Ensuring well-posedness by analogy; Stokes problem and boundary control for the wave equation}, Journal of Computational Physics, 103 (1992), pp. 189--221.

\bibitem{glowinski2003}
{R.~Glowinski}, {\em Finite Element Methods for Incompressible Viscous Flow}, Handbook of Numerical
Analysis, Vol. 9, Elsevier, Amsterdam, 2003, pp. 3-1176.


\bibitem{glowinski2015}
 R. Glowinski, {\em Variational Methods for the Numerical Solution of Nonlinear Elliptic Problems}, Society
for Industrial and Applied Mathematics, Philadelphia, 2015.

\bibitem{GH1998}
R. Glowinski and J. He, {\em On shape optimization and related issues}, In Computational
Methods for Optimal Design and Control, J. Borggaard, J. Burns, E. Cliff \& S. Schreck
(eds.), Birkhäuser, Boston, MA, 1998, pp.~151--179.

\bibitem{glowinski1994exact}
{ R.~Glowinski and J.~L.~Lions}, {\em Exact and approximate controllability for distributed parameter systems, Part I}, Acta Numerica, 3 (1994), pp.~269--378.

\bibitem{glowinski1995exact}
{ R.~Glowinski and J.~L.~Lions}, {\em Exact and approximate controllability for distributed parameter systems, Part II}, Acta Numerica, 4 (1995), pp.~159--328.

\bibitem{glowinski2008exact}
{ R.~Glowinski, J.~L.~Lions and J.~He}, {\em Exact and Approximate Controllability for Distributed Parameter Systems: A Numerical Approach (Encyclopedia of Mathematics and its Applications)}, Cambridge University Press, 2008.

\bibitem{lenhart1998}
{ N. Handagama and S. Lenhart}, {\em Optimal control of a PDE/ODE system modeling a gas-phase bioreactor}, In Mathematical Models in Medical and Health Sciences, M. A. Horn,
G. Simonett, and G. Webb (eds.), Vanderbilt University Press, Nashville, TN, 1998.

\bibitem{kunisch2007}
{ K. Ito and K. Kunisch}, {\em Optimal bilinear control of an abstract Schr{\" o}dinger equation}, SIAM Journal on Control and Optimization, 46 (2007), pp.~ 274--287.

\bibitem{joshi2005}
{ H. R. Joshi}, {\em Optimal control of the convective velocity coefficient in a parabolic problem}, Nonlinear Analysis: Theory, Methods \& Applications, 63 (2005), pp.~ e1383--e1390.

\bibitem{khapalov2003}
{ A. Y. Khapalov}, {\em Controllability of the semilinear parabolic equation governed by a multiplicative control in the reaction term: a qualitative approach},
SIAM Journal on Control and Optimization, 41 (2003), pp.~1886--1900.

\bibitem{khapalov2010}
{ A. Y. Khapalov}, {\em Controllability of Partial Differential Equations Governed by Multiplicative Controls}, Springer, 2010.

\bibitem{kroner2009}
{ A. Kr{\"o}ner and B. Vexler}, {\em A priori error estimates for elliptic optimal control problems with a bilinear state equation}, Journal of Computational and Applied Mathematics, 230 (2009), pp.~ 781--802.

\bibitem{lenhart1995}
{ S. Lenhart}, {\em Optimal control of a convective-diffusive fluid problem}, Mathematical Models and Methods in Applied Sciences, 5 (1995), pp.~225--237.

\bibitem{lions1971optimal}
{ J.~L.~Lions}, {\em Optimal Control of Systems Governed by Partial Differential Equations (Grundlehren der Mathematischen Wissenschaften)}, Vol.~170, Springer Berlin, 1971.
\bibitem{lions1988}
{ J. L. Lions}, {\em Exact controllability, stabilization and perturbations for distributed systems}, SIAM Review, 30 (1988), pp.~1--68.

\bibitem{liu2008}
W. Liu. {\em Mixing enhancement by optimal flow advection}, SIAM Journal on Control and Optimization, 47 (2008), pp.~624--638.


\bibitem{nocedal2006}
J. Nocedal, and S.J. Wright, {\em Numerical Optimization}, Second Edition, Springer, 2006.

\bibitem{troltzsch2010optimal}
{ F.~Tr{\"o}ltzsch}, {\em Optimal Control of Partial Differential Equations: Theory, Methods, and Applications}, Vol.~112, American Mathematical Society, 2010.

\bibitem{zuazua2005}
{ E. Zuazua}, {\em Propagation, observation, and control of waves approximated by finite difference methods}, SIAM Review,  47 (2005), pp.~197--243.
\bibitem{zuazua2006}
E. Zuazua, {\em Controllability of Partial Differential Equations}, 3rd cycle, Castro Urdiales (Espagne), 2006, pp.311. cel-00392196.

\bibitem{zuazua2007}
{ E. Zuazua}, {\em Controllability and observability of partial differential equations: some results and open problems}, Handbook of differential equations: evolutionary equations, Vol. 3, North-Holland, 2007, pp.~527--621.

	\end{thebibliography}
}
\end{document}